\documentclass[9pt]{amsart}
\usepackage{amssymb}
\usepackage{comment}
\usepackage{hyperref}
\usepackage[all]{xy}
\usepackage{enumerate}
\usepackage{ marvosym }
\usepackage{ bbold }
\usepackage{ tensor }

\usepackage{bbm}

\usepackage[dvipsnames]{xcolor}
\usepackage{graphicx}
\usepackage{extpfeil}

\newcommand{\red}[1]{\leavevmode{\color{red}{#1}}}

\DeclareMathOperator{\Gra}{{Gr}}

\DeclareMathOperator{\fact}{{fact}}
\DeclareMathOperator{\pt}{{pt}}

\DeclareMathOperator{\ind}{{ind}}
\DeclareMathOperator{\Jets}{{Jets}}
\DeclareMathOperator{\mer}{{mer}}
\DeclareMathOperator{\reg}{{reg}}

\DeclareMathOperator{\hormerxJets}{{Jets_{\nabla}^{mer, x}}}
\DeclareMathOperator{\hormerxJetsF}{{{Jets}}}

\DeclareMathOperator{\aff}{{aff}}
\DeclareMathOperator{\Div}{{Div}}
\DeclareMathOperator{\laxPreSt}{{laxPreSt}}
\DeclareMathOperator{\Corr}{{Corr}}
\DeclareMathOperator{\prjctn}{{pr}}

\usepackage{tikz-cd}

\newcommand{\bA}{{\mathbb A}}

\newcommand{\bE}{{\mathbb E}}
\newcommand{\bF}{{\mathbb F}}
\newcommand{\bG}{{\mathbb G}}

\newcommand{\bL}{{\mathbb L}}

\newcommand{\bN}{{\mathbb N}}

\newcommand{\bP}{{\mathbb P}}
\newcommand{\bQ}{{\mathbb Q}}

\newcommand{\bZ}{{\mathbb Z}}

\newcommand{\BA}{{\mathbf A}}
\newcommand{\BB}{{\mathbf B}}
\newcommand{\BC}{{\mathbf C}}
\newcommand{\BD}{{\mathbf D}}
\newcommand{\BE}{{\mathbf E}}

\newcommand{\BH}{{\mathbf H}}

\newcommand{\BM}{{\mathbf M}}
\newcommand{\BN}{{\mathbf N}}

\newcommand{\Bh}{{\mathbf h}}
\newcommand{\one}{{\mathbf 1}}

\newcommand{\cA}{{\mathcal A}}

\newcommand{\cC}{{\mathcal C}}

\newcommand{\cE}{{\mathcal E}}
\newcommand{\cF}{{\mathcal F}}
\newcommand{\cG}{{\mathcal G}}

\newcommand{\cL}{{\mathcal L}}
\newcommand{\cM}{{\mathcal M}}

\newcommand{\cO}{{\mathcal O}}

\newcommand{\cT}{{\mathcal T}}

\newcommand{\cX}{{\mathcal X}}
\newcommand{\cY}{{\mathcal Y}}
\newcommand{\cZ}{{\mathcal Z}}

\newcommand{\Ranp}{{\on{Ran}}}
\newcommand{\Randr}{\on{Ran}_{X_{\dR}}}
\newcommand{\Randrpt}{\operatorname{Ran}_{X_{\dR}, x}}

\newcommand{\colim}{\text{colim}}

\newcommand{\nc}{\newcommand}
\nc{\renc}{\renewcommand}
\nc{\ssec}{\subsection}
\nc{\sssec}{\subsubsection}
\nc{\on}{\operatorname}

\nc\Gr{\on{Gr}}

\nc\Fl{\on{Fl}}

\newtheorem{thm}[subsection]{Theorem}

\newtheorem{conj}{Conjecture}
\newtheorem*{conj*}{Conjecture}
\DeclareMathOperator{\DGCat}{{DGCat}}
\DeclareMathOperator{\QLisse}{{QLisse}}
\DeclareMathOperator{\Lisse}{{Lisse}}
\DeclareMathOperator{\IndLisse}{{IndLisse}}
\DeclareMathOperator{\QCoh}{{QCoh}}

\DeclareMathOperator{\Maps}{{Maps}}
\DeclareMathOperator{\Res}{{Res}}
\DeclareMathOperator{\oblv}{{oblv}}
\DeclareMathOperator{\Oblv}{{Oblv}}

\DeclareMathOperator{\add}{add}

 \DeclareMathOperator{\Bun}{{Bun}}

\DeclareMathOperator{\Isom}{{Isom}}
\DeclareMathOperator{\Lie}{{Lie}}
\DeclareMathOperator{\LocSys}{{LS}}
\DeclareMathOperator{\Loc}{{Loc}}

\DeclareMathOperator{\Sym}{{Sym}}
\DeclareMathOperator{\CommAlg}{{CAlg}}

\DeclareMathOperator{\univ}{{univ}}

\DeclareMathOperator{\BG}{{BG}}
\DeclareMathOperator{\PreSt}{{PreSt}}
\DeclareMathOperator{\Aff}{{Aff}}
\DeclareMathOperator{\Fin}{{Fin}}
\DeclareMathOperator{\Arr}{{Arr}}

\DeclareMathOperator{\EComm}{{EComm^{\vee}}}
\DeclareMathOperator{\ECMod}{{ECMod^{\vee}}}
\DeclareMathOperator{\Strcirc}{{\stackrel{\circ}{Str}}}
\DeclareMathOperator{\Str}{{Str}}

\newcommand{\limto}{{\displaystyle\lim_{\longrightarrow}}}
\newcommand{\rightlim}{\mathop{\limto}}

\newcommand{\leftlim}{\mathop{\displaystyle\lim_{\longleftarrow}}}
\newcommand{\limfromn}{\leftlim\limits_{\raise3pt\hbox{$n$}}}
\newcommand{\limton}{\rightlim\limits_{\raise3pt\hbox{$n$}}}

\newcommand{\rightlimit}[1]{\mathop{\lim\limits_{\longrightarrow}}\limits%
	_{\raise3pt\hbox{$\scriptstyle #1$}}}

\newcommand{\leftlimit}[1]{\mathop{\lim\limits_{\longleftarrow}}\limits%
	_{\raise3pt\hbox{$\scriptstyle #1$}}}

\DeclareMathOperator{\Id}{{Id}}

\DeclareMathOperator{\Hom}{{Hom}} \DeclareMathOperator{\Ind}{{Ind}}

 \DeclareMathOperator{\Mmod}{{-mod}}

 \DeclareMathOperator{\op}{{op}}

\DeclareMathOperator{\Spec}{{Spec}}

\DeclareMathOperator{\Vect}{{Vect}}

\DeclareMathOperator{\HH}{{HH}}

\DeclareMathOperator{\Stab}{{Stab}}

\DeclareMathOperator{\ev}{{ev}}

\DeclareMathOperator{\Rep}{{Rep}}

\makeatletter

\newcommand{\Rmnum}[1]{\expandafter\@slowromancap\romannumeral #1@}
\makeatother

\newtheorem{pr}[subsubsection]{Proposition}
\newtheorem{lm}[subsubsection]{Lemma}

\newtheorem{cor}[subsubsection]{Corollary}

\newtheorem{cnstr}[subsubsection]{Construction}

\newtheorem{df}[subsubsection]{Definition}
\newtheorem{thmdefn}[subsubsection]{Theorem-Definition}
\newtheorem{dfprop}[subsubsection]{Definition-Proposition}
\newtheorem{rem}[subsubsection]{Remark}

\newtheorem{ex}[subsubsection]{Example}
\newtheorem{ntn}[subsubsection]{Notation}

\numberwithin{equation}{section}

\newcommand{\dR}{\mathrm{dR}}

\DeclareMathOperator{\plaxlim}{{plaxlim}}
\DeclareMathOperator{\plaxcolim}{{plaxcolim}}
\DeclareMathOperator{\laxlim}{{laxlim}}
\DeclareMathOperator{\ins}{{ins}}
\DeclareMathOperator{\Funct}{{Funct}}
\DeclareMathOperator{\AffSch}{{AffSch}}

\DeclareMathOperator{\Groth}{{Groth}}

\DeclareMathOperator{\D}{{DMod}}
\DeclareMathOperator{\BiCat}{{BiCat_{\infty}}}
\DeclareMathOperator{\Ranu}{{Ran_U}}

\nc{\WFactAlg}{\on{WFactAlg}}
\nc{\FactAlg}{\on{FactAlg}}
\nc{\bFact}{\mathbf{Fact}}
\nc{\FactAlgCat}{\mathbf{FactAlgCat}}
\nc{\WFactAlgCat}{\mathbf{WFactAlgCat}}
\nc{\LWFactAlgCat}{\mathbf{LWFactAlgCat}}
\nc{\WFactModCat}{\on{-}\!\mathbf{WFactModCat}}
\nc{\LWFactModCat}{\on{-}\!\mathbf{LWFactModCat}}
\nc{\FactModCat}{\on{-}\!\mathbf{FactModCat}}
\nc{\WFactMod}{\on{-WFactMod}}
\nc{\FactMod}{\on{-FactMod}}
\nc{\weak}{{\on{weak}}}
\nc{\lax}{{\on{lax}}}
\nc{\str}{{\on{str}}}
\nc{\laxun}{{\on{laxun}}}
\nc{\strun}{{\on{strun}}}
\nc{\enh}{{\on{enh}}}
\nc{\IC}{{\on{IndCoh}}}
\nc{\restr}{{\on{restr}}}
\nc{\redu}{{\on{red}}}
\nc{\naive}{{\on{naive}}}
\nc{\indhol}{{\on{ind-hol}}}
\nc{\Sets}{{\on{Sets}}}
\nc{\cone}{{\on{cone}}}
\nc{\cha}{{\on{char}}}
\nc{\triv}{{\on{triv}}}
\nc{\redd}{{\on{red}}}
\nc{\CFactAlg}{{\on{CFactAlg}}}
\nc{\proje}{{\on{pr}}}
\nc{\action}{{\on{act}}}
\nc{\Cat}{{\on{Cat}}}
\nc{\Grpd}{{\on{Grpd}}}
\nc\wt{\widetilde}
\newcommand{\et}{\on{et}}

\begin{document}
	
	\title[Local systems with restricted variation via factorization]{Local systems with restricted variation on the formal punctured disc via factorization}
	
	\author[E.~Bogdanova]{Ekaterina Bogdanova}
	\address{Harvard University,  USA}
	\email{ebogdanova@math.harvard.edu}
	
	\begin{abstract}
		We define the stack of $G$-local systems with restricted variation on the formal puntured disc and study its properties. We embed sheaves of categories over this stack into the category of factorization module categories over $\Rep(G)$. Along the way we develop a theory of factorization structures in families and study functorialities of such under changes of the base curve.
	\end{abstract}
	
	\maketitle
	
	\tableofcontents
	
\section{Introduction}
Let $k$ be a ground field of characteristic zero, $G$ be a split connected reductive group over $k$, $\check{G}$ be the Langlands dual group to $G$.
\subsubsection{Local geometric Langlands} The local geometric Langlands program from \cite{FG06a} suggests that, roughly speaking, DG categories strongly acted on by the loop group $LG$ are equivalent to categories over the
moduli space of $\check{G}$-local systems on the formal punctured disc, i.e. $\check{G}$-bundles on the formal punctured disc with connection. The equivalence reads as

$$\bL: \D(LG)\mathbf{-ModCat}  \xrightarrow{\sim} \QCoh(\LocSys_{\check{G}}(D^{\circ}))\mathbf{-ModCat}.$$

The first orienting remark is that $\bL$ does not preserve the functor of forgetting the module structures. Rather $\bL$ should be thought of as a categorical Morita equivalence. 

Local geometric Langlands is supposed to satisfy various compatibilities as in the arithmetic theory, e.g. there should be a compatibilities with parabolic induction and with Whittaker models.

There are a number of examples providing the evidence for the existence of such an equivalence. Among them are derived Satake equivalence (\cite{BF}), Bezrukavnikov's theory (\cite{BR}), works of Frenkel-Gaitsgory (\cite{FG04}, \cite{FG06a}, \cite{FG09a}, \cite{FG09b}, \cite{FG09c}) and Raskin (\cite{Ras15b}, \cite{Ras21}). One expects that these works combined with the compatibilities are supposed to pin down the equivalence $\bL$. 

One distinguishing feature of this setup is the geometry on the set of spectral parameters, which has led to the suggestion (c.f. \cite{BD}) that pointwise spectral descriptions should extend in families over these moduli spaces. More precisely, there should be a natural functor (different from $\bL$)
\begin{equation}\label{spectral decomposition}
	\D(LG)\mathbf{-ModCat} \rightarrow \QCoh(\LocSys_{\check{G}}(D^{\circ}))\mathbf{-ModCat}
\end{equation}
preserving the functor forgetting the module structures. 

\subsubsection{A sibling: restricted local geometric Langlands.} A recent progress in local arithmetic Langlands has been the formulation of the conjecture using families of Galois representations  (\cite{Zhu21}). Roughly speaking, it says that there exists a fully faithful functor from representations of $G(\bF_q((t)))$ to (appropriately understood) coherent sheaves on a certain stack parametrizing Galois representations. 

In an unpublished work, Gaitsgory came up with a categorical incarnation of this conjecture, which provides a link (via methods of \cite{AGKRRV2}) between geometric and arithmetic worlds. Moreover, Gaitsgory`s conjecture is suited to any ground field and any sheaf theory. It roughly states that there exists a natural equivalence
\begin{equation}\label{restricted Langlands}
	\bL^{\restr}: LG\mathbf{-ModCat}_{\restr} \xrightarrow{\sim} \QCoh(\LocSys_{\check{G}}^{\restr}(D^{\circ}))\mathbf{-ModCat},
\end{equation}
where on the left-hand side we have to work with a subtler notion of categorical representation (in order to fix an issue with the failure of Kunneth formula in the constructible setting), and on the right-hand side we need accordingly to replace $\LocSys_{\check{G}}(D^{\circ})$ by the moduli of {\it local systems with restricted variation}.
\subsection{Statement of the result}

Let $X$ be a smooth affine curve, $x \in X$ be a closed point. Let $U = X \setminus x_0$.

One of the difficulties in local geometric Langlands is the fact that $\LocSys_{\check{G}}(D^{\circ})$ is complicated as an algebro-geometric object. For instance, this is not an Artin stack, and it is not even locally of finite type. One way to avoid dealing with this problem is to introduce an additional structure on the objects with rather simple geometry. Namely, factorization structure. 
The general heuristic is that factorization structures tend to "decrease the complexity" in the following sense. A factorization algebra whose fibers involve only the formal disc will have invariants encoding the information about the whole formal punctured disc. For instance, factorization geometry of Beilinson-Drinfeld grassmannian encodes usual geometry of the loop group, and factorization geometry of classifying space of $\check{G}$ encodes geometry of $\check{G}$-local systems on the formal punctured disc. We refer to \cite{R} for a more detailed introduction to the factorizable geometry. 

\subsubsection{From local systems on the punctured disc to factorization modules over $\Rep(\check{G})$.}
\begin{conj}\label{intro conj}
	There exists  a fully faithful functor 
	$$\bFact: \QCoh(\LocSys_{\check{G}}(D^{\circ}))\mathbf{-ModCat} \rightarrow \Rep({\check{G}})\mathbf{-FactModCat}$$
	that preserves the forgetful functors from both sides to $\mathbf{DGCat}$.
	Here 
	\begin{itemize}
		\item[-] On the left-hand side $\QCoh(\LocSys_{\check{G}}(D^{\circ}))$ is viewed as a monoidal category with usual tensor product. 
		\item[-] On the right-hand side $\Rep({\check{G}})$ is the DG category of representation of ${\check{G}}$ equipped with factorization algebra category structure coming from symmetric monoidal structure on $\Rep({\check{G}})$. 
	\end{itemize}
\end{conj}

\subsubsection{} As in \cite[1.1.1]{text}, the functor will be given as tensor product with a certain factorization $\Rep({\check{G}})$-module category, $\bFact(\QCoh(\LocSys_{\check{G}}(D^{\circ})))$, equipped with a commuting action of $\QCoh(\LocSys_{\check{G}}(D^{\circ})$. 
The fiber at $x$ of this category will identify with $\QCoh(\LocSys_{\check{G}}(D^{\circ}))$, and the action will be given by tensor product of quasi-coherent sheaves. In more detail, $\bFact(\QCoh(\LocSys_{\check{G}}(D^{\circ})))$ is defined as the category of quasi-coherent sheaves on a certain factorization prestack $\LocSys_{\check{G}}(D^{\circ}_{\Ranp})$ degenerating  $\LocSys_{\check{G}}(D)$ to  $\LocSys_{\check{G}}(D^{\circ})$. 

In this paper we prove an analogue of Conjecture \ref{intro conj} in the context of {\it restricted} local geometric Langlands. Before we formulate the statement, we need to introduce the replacement for the left-hand side of the Conjecture \ref{intro conj}.

\subsubsection{The stack of local systems with restricted variation on the puntured disc. }

Although our main theorem is in the de Rham setting, we define and study the stack of local systems with restricted variation on the puntured disc in all three frameworks: Betti, de Rham, and etale. We refer to the introduction of \cite{AGKRRV2} for the overview of the idea of local systems with restricted variation. Here we highlight the features of the local context. 
\begin{itemize}
\item In the Betti situation $\QLisse(D^{\circ})$ is equivalent to $\QLisse(\bG_m)$, so we can directly use results from \cite{AGKRRV2}.

\item Let us turn to the de Rham context. Although it is complicated to define the category of all $D$-modules on the formal punctured disc, distinguishing the subcategory of lisse ones is rather straightforward. We set $\Lisse(D^{\circ})$ to be the category of free $k((t))$-modules of finite rank equipped with a connection. We also define $$\QLisse(D^{\circ}) = \IndLisse(D^{\circ}).$$ Note that this category is already left complete, which is not always the case in the global setting. 

\item In the etale context we define $\Lisse(D^{\circ})$ to be the category of lisse $\ell$-adic sheaves on $\Spec(k((t)))$, and $\QLisse(D^{\circ}) := \IndLisse(D^{\circ}).$ Here characteristic of $k$ could be greater than zero.
\end{itemize}
Results of Katz (\cite{Katz2}, \cite{Katz3}) provide a connection between $\QLisse(D^{\circ})$ and $\QLisse(\bG_m)$ in de Rham and etale situations. Namely, we have a section of a natural restriction map 
\begin{equation}\label{intro section}
		\QLisse(\bG_m) \rightarrow  \QLisse(D^{\circ}).
\end{equation}

For an affine test scheme $S=\Spec(A)$, an $S$-point of $\LocSys_{\check{G}}^{\restr}(D^{\circ})$ is symmetric monoidal functor $$\Rep({\check{G}}) \rightarrow \QLisse(D^{\circ}) \otimes (A\Mmod).$$ 
Note that (\ref{intro section}) implies that there is a section of a natural map 
$$\LocSys_{\check{G}}^{\restr}(\bG_m) \rightarrow \LocSys_{\check{G}}^{\restr}(D^{\circ}).$$
We remark that unlike the global case, the stack $\LocSys_{\check{G}}^{\restr}(D^{\circ})$ is classical. Nevertheless, we still must work in the context of derived algebraic geometry.

As in the global case, the stack $\LocSys_{\check{G}}^{\restr}(D^{\circ})$ splits into a
disjoint union of prestacks $\cZ_{\sigma}$ parameterized by isomorphism classes of semi-simple ${\check{G}}$-local systems $\sigma$ on $D^{\circ}$, and each $\cZ_{\sigma}$ can be written as a quotient of a formal affine scheme by ${\check{G}}$. 

\subsubsection{Main result.}

\begin{thm}\label{intro theorem}
	There exists  a fully faithful functor 
	$$\bFact^{\restr}: \QCoh(\LocSys_{\check{G}}^{\restr}(D^{\circ}))\mathbf{-ModCat} \rightarrow \Rep({\check{G}})\mathbf{-FactModCat}$$
	that preserves the forgetful functors from both sides to $\mathbf{DGCat}$.
	Here 
	\begin{itemize}
		\item[-] On the left-hand side $\QCoh(\LocSys_{\check{G}}^{\restr}(D^{\circ}))$ is viewed as a monoidal category with usual tensor product. 
		\item[-] On the right-hand side $\Rep({\check{G}})$ is the DG category of representation of ${\check{G}}$ equipped with factorization algebra category structure coming from symmetric monoidal structure on $\Rep({\check{G}})$. 
	\end{itemize}
\end{thm}

\subsubsection{} As in Conjecture \ref{intro conj}, the functor $\bFact^{\restr}$ is given by a certain factorization $\Rep({\check{G}})$-module category, $\bFact^{\restr}(\QCoh(\LocSys_{\check{G}}^{\restr}(D^{\circ})))$, equipped with a commuting action of $\QCoh(\LocSys_{\check{G}}^{\restr}(D^{\circ}))$. We deinfe $\bFact^{\restr}(\QCoh(\LocSys_{\check{G}}^{\restr}(D^{\circ})))$ as the category of quasi-coherent sheaves on the prestack $$\LocSys_{\check{G}}(D^{\circ}_{\Ranp}) \times_{\LocSys_{\check{G}}(D^{\circ})}\LocSys_{\check{G}}^{\restr}(D^{\circ}).$$

\subsection{Motivation}

\subsubsection{} First, we provide pieces of motivation of why one might expect such a statement to be true. In topology, the analogue of  Conjecture \ref{intro conj} and Theorem \ref{intro theorem} are transparent. Namely, the parallel to factorization structure in algebraic geometry is $\bE_2$-structure in topology, so we are interested in the category of $\bE_2$-modules over $\Rep(\check{G})$ (this is a symmetric monoidal category, so, in particular, an $\bE_2$-category). But recall that for any $\bE_2$-category $\mathbf{A}$ one has 
$$\mathbf{A}\Mmod^{\bE_2} \cong (\int_{S^1} \mathbf{A})\Mmod,$$
where on the right-hand side we consider left modules over chiral homology of $\mathbf{A}$. We can also rewrite $$(\int_{S^1} \mathbf{A})\Mmod \cong \HH(\mathbf{A})\Mmod.$$
So for $\mathbf{A} = \Rep(\check{G})$ we get 
$$\Rep(\check{G})\Mmod^{\bE_2} \cong \QCoh(\Loc_{\check{G}}(S^1))\Mmod.$$

\subsubsection{} Next we discuss how Theorem \ref{intro theorem} fits in the context of restricted local Langlands conjecture \ref{restricted Langlands}. One prediction of the local restricted Langlands correspondence is that any $\mathbf{C} \in LG\mathbf{-ModCat}_{\restr}$ is fibered over $\LocSys_G^{\restr}(D^{\circ})$, i.e. one expects there to be a natural functor (parallel to (\ref{spectral decomposition}))
\begin{equation}\label{spectral decomp}
	\D(LG)\mathbf{-ModCat}_{\restr} \dashrightarrow \QCoh(\LocSys_G^{\restr}(D^{\circ})) \mathbf{-ModCat}
\end{equation}
that preserves the forgetful functors from both sides to $\mathbf{DGCat}$. This functor would serve as the spectral decomposition in local restricted geometric Langlands program. Let us outline the construction of this functor.

\subsubsection{} There is a dual statement to Conjecture \ref{intro conj} which is the main theorem of \cite{text}:

\begin{thm} 
	If $G$ is connected and reductive, then there is a fully faithful functor:
	\[
	\bFact: \D(LG)\mathbf{-ModCat} \to \D(\Gr_{G})\mathbf{-FactModCat}
	\]
	that preserves the forgetful functors from both sides to $\mathbf{DGCat}$.
\end{thm} 
Above:
\begin{itemize}
	\item 
	In the left-hand side, $\D(LG)$ is the DG category of D-modules on the loop group $L G$, equipped with the convolution monoidal structure. 
	\item
	On the right-hand side, $\D(\Gr_G)$ is the DG category of D-modules on the affine Grassmannian $\Gr_G$, equipped with the factorization algebra category structure inherited from the factorization structure of the Beilinson--Drinfeld affine Grassmannians .
	\end{itemize}

The fully faithful functor $\bFact$ is also given by  a ``bimodule'' object equipped with commuting actions of $\D(L G)$ (in the usual sense) and $\D(\Gr_G)$ (in the factorization sense). The underlying DG category of this bimodule is just $\D(L G)$.

\begin{cor}
	If $G$ is connected and reductive, then there is a fully faithful functor:
	\[
	\bFact: \D(LG)\mathbf{-ModCat}_{\restr} \to \D(\Gr_{G})\mathbf{-FactModCat}
	\]
	that preserves the forgetful functors from both sides to $\mathbf{DGCat}$.
\end{cor}

The desired functor (\ref{spectral decomp}) would make the following diagram commute:
	This follows by definition from the fact that for $f: \Spec(A) \rightarrow \Randrpt$ the diagram 
\[
\begin{tikzcd}
\D(LG)\mathbf{-ModCat}_{\restr} \ar[r, dashrightarrow, ""']\ar[d, hookrightarrow, ""']&   	\QCoh(\LocSys_{\check{G}}^{\restr}(D^{\circ}))\mathbf{-ModCat} \ar[d, hookrightarrow, ""']  \\
\D(\Gr_{G})\mathbf{-FactModCat}\ar[r, rightarrow, ""']&   \Rep({\check{G}})\mathbf{-FactModCat},
\end{tikzcd}
\]

where the bottom functor is given by restriction along the (factorization) Satake functor
\[
\on{Sat}: \Rep({\check{G}})\to  \D(\Gr_G).
\]

\subsubsection{} It will follow from this construction that for any ${\check{G}}$-local system $\sigma$ on the punctured disk and $\BM \in \mathbf{ShvCat}(\on{LocSys}_{{\check{G}}}(D^{\circ} ))$, the fiber DG category $\BM_\sigma$ can be described as
\[
\BM_\sigma \simeq R_{{\check{G}},\sigma}\on{-FactMod}(\check \bFact(\BM)),
\]
where $R_{{\check{G}},\sigma}$ is a factorization algebra object in $\on{Rep}({\check{G}})$ obtained by twisting the regular representation $R_{\check{G}}$ by $\sigma$. As an application, for any categorical representations $\mathbf{C}$ of $LG$ and a ${\check{G}}$-local system $\sigma$ on the punctured disk, we obtain a conjectural description of the fiber of the spectral decomposition of $\mathbf{C}$ at $\sigma$. Namely, we conjecture it to be
\[
\on{Sat}(R_{{\check{G}},\sigma})\on{-FactMod}( \bFact(\mathbf{C}) ).
\]

\subsection{Outline of the argument}
Below we highlight the main ideas that go into the proof of  Theorem \ref{intro theorem}. 

We first notice that the question is of purely local nature, so should be independent of what curve we are using. To that end, for a map of curves $f: X \rightarrow Y$ we develop a notion of pullback of factorization structures and prove that for a fixed $\mathbf{A} \in \mathbf{FactAlg}(\Ranp_Y)$ we have 
\begin{equation}\label{pullback intro}
f^!: \mathbf{A}-\mathbf{FactMod}(\Ranp_{Y, y}) \xrightarrow{\sim} f^!(\mathbf{A})-\mathbf{FactMod}(\Ranp_{X, x}). 
\end{equation}

We also prove that (\ref{pullback intro}) is compatible with the functor $\bFact^{\restr}$ in the sense that the following diagram commutes

\begin{equation}
\begin{tikzcd}
&   	\Rep({\check{G}})-\mathbf{FactMod}(\Ranp_{Y, y}) \ar[d, dash, "\cong"']  \\
\QCoh(\LocSys_{\check{G}}^{\restr}(D^{\circ}))\mathbf{-ModCat}  \ar[ur, rightarrow, "\bFact^{\restr}_Y"]\ar[r, rightarrow, "\bFact^{\restr}_X"']&   \Rep({\check{G}})-\mathbf{FactMod}(\Ranp_{X, x}).
\end{tikzcd}
\end{equation}

Taking $f$ to be a smooth function on $X$ we reduce to the case $X = \bA^1$, $x = 0$. 

Let $\BC_1$ and $\BC_2$ be DG categories acted on by $\QCoh(\LocSys_G^{\restr}(D^{\circ})$.
Theorem \ref{mainthm} asserts that the functor 
\begin{equation}\label{main assertion intro}
\on{Fun}_{\QCoh(\LocSys_{\check{G}}^{\restr}(D^{\circ}))}(\BC^1, \BC^2 ) \rightarrow  \on{FactFun}_{\Rep({\check{G}})}(\bFact^{\restr}(\BC^1), \bFact^{\restr}(\BC^2))
\end{equation}
is an equivalence. Let us illustrate the plan of the proof in the simplest case: when $\BC_i$ for $i=1,2$ are $\Vect$ as plain categories, but the action of  $\QCoh(\LocSys_{\check{G}}^{\restr}(D^{\circ}))$ is via pullback to a point $\pt_{\sigma_i}$ corresponding to a local system on the punctured disc. Denote these $\BC_i$ by $\Vect_{\sigma_i}$.

Then left-hand side of (\ref{main assertion intro}) has a clear geometric meaning: 
$$\on{Fun}_{\QCoh(\LocSys_{\check{G}}^{\restr}(D^{\circ}))}(\BC^1, \BC^2 ) \cong \QCoh(\pt_{\sigma_1} \times_{\LocSys_{\check{G}}^{\restr}(D^{\circ})} \pt_{\sigma_2}).$$

In order to give a geometric description of the category on the right-hand side of (\ref{main assertion intro}), we use tools introduced in \cite{text} for working with factorization modules. Also, to avoid complicated geometry of the formal punctured disc we use \cite[Theorem 2.4.10]{Katz2}, which allows us to extend $\sigma_i$ to local systems $\widetilde{\sigma}_1$ on $\bG_m$. Combining these methods, we get 
$$\on{FactFun}_{\Rep({\check{G}})}(\bFact^{\restr}(\BC^1), \bFact^{\restr}(\BC^2)) \cong {\prescript{}{\widetilde{\sigma}_1}{R_{{\check{G}}}}}_{\widetilde{\sigma}_2}\on{-FactMod}.$$
Here ${\prescript{}{\widetilde{\sigma}_1}{R_{\check{G}}}}_{\widetilde{\sigma}_2}$ is a factorization algebra obtained by twisting the regular representation $R_{\check{G}}$ by $\widetilde{\sigma}_1$ and $\widetilde{\sigma}_2$. Note that $${\prescript{}{\widetilde{\sigma}_1}{R_{\check{G}}}}_{\widetilde{\sigma}_2}$$ is a commutative factorization algebra corresponding to the D-scheme $\Isom(\widetilde{\sigma}_1, \widetilde{\sigma}_2)$. But Beilinson-Drinfeld theory of chiral algebras gives a concise geometric description of the category of modules over such factorization algebras: the answer is the category of quasi-coherent sheaves of the space of horizontal sections of the corresponding D-scheme on the punctured disc. It is left to notice that in the case of the D-scheme $\Isom(\widetilde{\sigma}_1, \widetilde{\sigma}_2)$ it is precisely $$\QCoh(\pt_{\sigma_1} \times_{\LocSys_{\check{G}}^{\restr}(D^{\circ})} \pt_{\sigma_2}).$$

\subsection{Notation and conventions}
 Denote by $X$ a fixed smooth and connected (but not necessarily proper) curve over $k$.
\subsubsection{Lie theory.} Let $G$ be a connected reductive group with a fixed standard Borel $B$ and a maximal torus $T$. Let $(\check{\Lambda}, \check{\Delta}, \Lambda, \Delta)$ denote the root datum corresponding to $(G, T)$.
Langlands duality attaches to every $(G, T)$ a dual group $(\check{G}, \check{T})$ with root datum $( \Lambda, \Delta, \check{\Lambda}, \check{\Delta})$.

We define the loop group ind-scheme $LG$ (resp. the arc group scheme $L^+G$) as 
$$LG(\Spec A) := G(\Spec A((t))),$$
$$L^+G(\Spec A) := G(\Spec A[[t]]).$$
We define the affine Grassmannian as the fppf quotient $$\Gra_G := LG/ L^+G.$$
\subsubsection{Higher categories.} We will substantially use the language of $\infty$-categories as developed in \cite{HTT}. We write {\it category} for an $(\infty, 1)$-category unless specified otherwise. Our main result is a statement about $(\infty, 2)$-categories. We refer to \cite{GR1}, \cite{cospan}, \cite{CF} and \cite{text} for the overview of relevant foundations of $(\infty, 2)$-categories.
We write $\Cat_{\infty}$ for an $(\infty, 1)$-category of $(\infty, 1)$-categories and $\mathbf{Cat}_{\infty}$ for its enhancement to an $(\infty, 2)$-category. We write $\BiCat$ for an $(\infty, 1)$-category of $(\infty, 2)$-categories and $\mathbf{BiCat}_{\infty}$ for its enhancement to an $(\infty, 2)$-category.

\subsubsection{Higher algebra.} We will also use higher algebra as in \cite{HA}. 
We will denote by $\Vect$ the stable category of chain complexes of vector spaces over $k$. This category is equipped with a natural symmetric monoidal structure (in ther sense of $\infty$-categories). 

\subsubsection{DG categories.} By a DG category we mean a category equipped with a module structure over $\Vect$ with respect to the symmetric monoidal structure on the category of presentable stable categories given by the Lurie tensor product. DG categories form an $(\infty, 1)$-category (where 1-morphisms are {\it colimit-preserving} functors) denoted by $\DGCat$, which carries a symmetric monoidal structure via Luire tensor product. Denote by $\mathbf{DGCat}$ its enhancement to an $(\infty, 2)$-category. In particular, we can talk about algebras in $\mathbf{DGCat}$ and modules over them. For an algebra $\BA$ in $\mathbf{DGCat}$ denote the $(\infty,2)$-category of modules over it by $\BA-\mathbf{ModCat}$.
Unless specified otherwise, all monoidal DG categories will be assumed unital. We also utilize the formalism of categorical invariant/coinvatiants developed in \cite{B}.

Given a DG category $\mathbf{C}$ we can talk about t-structures on it. Given a t-structure we denote by 
$$\mathbf{C}^{\leq 0}, \text{  }\mathbf{C}^{\geq 0}, \mathbf{C}^{\heartsuit}$$
the corresponding subcategories according to the cohomological conventions. 

\subsubsection{Derived algebraic geometry.} Most of the geometry in the present paper is taking place at the spectral side of local geometric Langlands, where we must work in the context {\it derived} algebraic geometry as in \cite{GR1}. We begin with the category $\AffSch$ of {\it derived} affine schemes over $k$. All of our geometric objects can be viewed as either a prestack, i.e. a functor 
$$\AffSch^{\op} \rightarrow \Grpd_{\infty},$$
or an lax prestack, i.e. a functor 
$$\AffSch^{\op} \rightarrow \Cat_{\infty}.$$
By a stack we will mean a prestack satisfying etale descent. We denote by $\laxPreSt_{\Corr}$ the category of lax prestacks with morphisms given by correspondences.
\subsection{Organization of the paper} In Section \ref{prelims} we recall the notions and statements regarding factorization structures relevant to the present paper. In Section \ref{section: statement} we define our main geometric object: the stack of local systems with restricted variation on the formal punctured disc, and study its relationship to the existing moduli stack in the de Rham setting. We define the functor in Theorem \ref{intro theorem} and formulate our main result. In Section \ref{section: reductions} we perform a series of reductions of the statement of Theorem \ref{intro theorem}. Namely, we first reduce to the case when $X = \bA^1$. Next, recall that Theorem \ref{intro theorem} asserts that the functor 
\begin{equation}\label{org equation}
\on{Fun}_{\QCoh(\LocSys_G^{\restr}(D^{\circ}))}(\BC^1, \BC^2 ) \rightarrow  \on{FactFun}_{\Rep(G)}(\bFact^{\restr}(\BC^1), \bFact^{\restr}(\BC^2))
\end{equation}
is an equivalence. We reduce to the case when $\BC^1 \cong \QCoh(R_1)$ and $\BC^2 \cong \QCoh(R_2)$, where the action of $\QCoh(\LocSys_G^{\restr}(D^{\circ})$ on $\QCoh(R_i)$ comes from a pullback along a map $R_i \rightarrow \LocSys_G^{\restr}(D^{\circ})$. We then describe both sides of (\ref{org equation}) explicitly. In Section \ref{section: main proof} we conclude the proof of Theorem \ref{intro theorem} using this description. In Appendix \ref{appendix: multidiscs} we recall the notion of multidiscs and describe the correct may to descend multidiscs to the de Rham prestack. We define the factorization prestack degenerating horizontal jets into a fixed prestack $\cY$ to meromorphic horizontal jets into $\cY$ and its parameterized analogue. In Appendix \ref{appendix: etale factorization} we study functoriality of factorization patterns under etale pullbacks of the base curve. In Appendix \ref{paramfact} we establish foundations of parameterized crystals of categories and parameterized factorization patterns. Finally, in Appendix \ref{appendix: factmodules over commutative alg} we describe in geometric terms the category of (parameterized) factorization modules over a {\it commutative} (parameterized) factorization algebra.

\subsection{Acknowledgments} First, I would like to thank Dennis Gaitsgory for proposing this project and for the many helpful discussions throughout. I am grateful to Lin Chen and Yuchen Fu for their help with the foundations of factorization structures. For instance, the contents of Appendix \ref{appendix: etale factorization} were explained to us by Lin Chen. I am also grateful to Kevin Lin, Sasha Petrov, Sam Raskin, and David Yang for fruitful discussions related to this project.
Finally, I thank Kevin Lin for many very helpful comments and corrections on the early version of this paper. 

Part of this work was completed during my stay at the Max Planck Institute for Mathematics in Bonn, I am grateful to MPIM for hospitality. 

\section{Preliminaries on factorization structures }\label{prelims}

In this section we give a somewhat informal introduction to the formalism of factorization geometry.  Everything apart from subsections \ref{factfamilies} and \ref{ranu factorization} has appeared in either \cite{BD}, \cite{R}, \cite{CF} or \cite{text}.

\subsection{Factorization algebras}

\subsubsection{Ran space and crystals of categories.}

\begin{df}\label{def ran}
	The unital Ran space of $X$, denoted by $\Ranp$, is a lax prestack that assigns to $S \in \AffSch$ the category $\Ranp(S)$ of finite subsets of $X(S)$, where morphisms are given by inclusions of subsets.
\end{df}

\begin{rem}
	Definition \ref{def ran} makes sense for any prestack $Y$ in place of $X$. We will use the notation $\Ranp_Y$ in such cases.
\end{rem}

\begin{df}[{\cite{CF}}]  Let $Y$ be a laft\footnote{I.e. as a functor $\on{AffSch}^{\on{op}} \to \on{Cat}_\infty$ it is obtained by left Kan extension along its restriction on the category of affine schemes of finite type $\on{AffSch}_{\on{ft}} \subset \on{AffSch}$.} lax prestack. A \emph{crystal of categories} $\BA$ on $Y$ consists of the following data:
	\begin{itemize}
		\item For any $s:S\to Y$ with $S$ a finite type affine scheme, there is a $\D(S)$-module DG category:
		\[ \BA_s \in \D(S)-\mathbf{ModCat}. \]
		
		\item For any 2-cell
		\[
		\xymatrix{
			& S \ar[rd]^-{s} \ar@{=>}[d]^-\theta \\
			T \ar[ru]^-v \ar[rr]_{t} & & Y,
		}
		\]
		there is a functor
		\[
		\theta_\dagger: \BA_s \to \BA_t
		\]
		intertwining the symmetric monoidal functor $u^!:\D(S) \to \D(T)$, such that if $\theta$ is invertible, then the induced functor
		\[
		\D(T) \otimes_{\D(S)} \BA_s \to \BA_t
		\]
		is an equivalence.
		\item Certain higher compatibilities. (see \cite{CF}.)
	\end{itemize}
\end{df}

There are two notions of morphisms between crystals of categories on a lax prestack: lax ones and strict ones.

\begin{df}[{\cite{CF}}] \label{defn-lax-morphism-cryscat}
	
	A \emph{morphism $F: \BA\to \BB$ between two crystals of categories on $Y$} consists of the following data
	\begin{itemize}
		\item For any $s:S\to Y$ with $S$ being affine, there is a $\D(S)$-linear functor
		\[F_u:\BA_s\to \BB_s;\]
		\item For any 2-cell
		\[
		\xymatrix{
			& S \ar[rd]^-{s} \ar@{=>}[d]^-\theta \\
			T \ar[ru]^-v \ar[rr]_{t} & & Y,
		}
		\]
		with $S$ and $T$ being affine, there is a $\D(S)$-linear natural transformation
		\[
		\xymatrix{
			\BA_s  \ar[r]^-{\theta_\dagger} \ar[d]_-{F_s} &
			\BA_t \ar[d]^-{F_t} \\
			\BB_s  \ar[r]_-{\theta_\dagger} \ar@{=>}[ru] &
			\BB_t
		}
		\]
		which is invertible if $\theta$ is so.
		\item Certain higher compatibilities.
	\end{itemize}
	
	\medskip
	
	A morphism $F: \BA\to \BB$ is \emph{strict} if the above natural transformations are always invertible.
	
	\medskip 
	
	Natural transformations between two morphisms are defined in the obvious way.
	
\end{df}

\begin{ntn}
	Let 
	\[\mathbf{CrysCat}(Y),\; \mathbf{CrysCat}^{\str}(Y)\] 
	be the $(\infty,2)$-category of crystals of categories on $Y$, with 1-morphisms given by morphisms (resp. strict morphisms) between crystals of categories, and 2-morphisms given by natural transformations.
\end{ntn}

\subsubsection{Factorization algebra categories.}
\begin{cnstr}
	\label{constr-Ran-group} 
	
	By definition, $\Ranp \times \Ranp$ is the lax prestack that assignes for any $S$ the category $\Ranp(S) \times \Ranp(S)$ of two finite subsets $P_1,P_2$ of $X(S)$. We have a morphism
	\[
	\wt{\on{add}}: \Ranp\times \Ranp \to \Ranp,\; (P_1,P_2) \mapsto P_1\cup P_2
	\]
	which makes $\Ranp$ an abelian group prestack.

\end{cnstr}

\begin{cnstr}
	\label{constr-Ran-commalg} 
	
	For any affine test scheme $S$, let 
	\[
	(\Ranp \times \Ranp)_{\on{disj}}(S) \subset \Ranp(S) \times \Ranp(S)
	\]
	be the full subcategory where we require the union of the graphs of $P_1$, viewed as a closed subset of $X\times S$, is disjoint from that of $P_2$. This defines a lax prestack $(\Ranp \times \Ranp)_{\on{disj}}$ and an open embedding $j: (\Ranp \times \Ranp)_{\on{disj}} \to \Ranp \times \Ranp$.
	
	\medskip
	
	Consider the morphism $\on{add}:= \wt{\on{add}}\circ j$. The correspondence
\begin{equation}\label{corr ran}
	\xymatrixrowsep{0.5cm}
	\xymatrixcolsep{0.5cm}
	\xymatrix{
		& (\Ranp \times \Ranp)_{\on{disj}} \ar[ld]_-{j} \ar[rd]^-{\on{add}} \\
		\Ranp \times \Ranp & &  \Ranp
	}
\end{equation}
	makes $\Ranp$ a \emph{commutative algebra object} in the category of lax prestacks with morphisms given by correspondences.
	
	\medskip
	
\end{cnstr}

	\begin{df}[{\cite{CF}}] \label{defn-laxunital-factalg-cat}
	A \emph{factorization algebra category} (over $\Ranp$) consists of the following data:
	\begin{itemize}
		\item An object
		\[ \BA\in \mathbf{CrysCat}(\Ranp), \]
		
		\item An isomorphism
		\[\on{unit}:\Vect_{\on{pt}} \to \emptyset_{\on{pt}}^\bullet(\BA)\]
		in $\mathbf{CrysCat}(\on{pt})$;
		
		\item An isomorphism
		\[\on{mult}:j^\bullet(\BA \boxtimes \BA) \to \on{add}^\bullet(\BA)\] 
		in $\mathbf{CrysCat}((\Ranp\times \Ranp)_{\on{disj}})$;
		
		\item Higher compatibilities between the above isomorphisms and the commutative algebra structure on $\Ranp$ defined in Construction \ref{constr-Ran-commalg}.
	\end{itemize}
	
	We will denote a factorization algebra category as above by $\BA$ and treat other data as implicit.
	
\end{df}

\begin{rem} \label{rem-lax unital-factalg-cat}
	Informally, for a crystal of categories $\BA$ on $\Ranp$, a factorization algebra structure on it consists of the following data:
	\begin{itemize}
		\item An equivalence
		\[\on{unit}:\Vect \to \BA_{\emptyset};\]
		\item For any two finite sets $\underline{x}$, $\underline{y}$ of closed points of $X$ such that $\underline{x}\cap \underline{y}=\emptyset$, an equivalence
		\begin{equation}
		\label{eqn-mult-explicit}
		\on{mult}_{\underline{x},\underline{y}}:\BA_{\underline{x}} \otimes \BA_{\underline{y}} \to \BA_{\underline{x}\sqcup \underline{y}};
		\end{equation}
		\item
		For any finite set $\underline{x}$, a commutative diagram
		\[
		\xymatrix{
			& \BA_{\emptyset} \otimes \BA_{\underline{x}} \ar[rd]^-{\on{mult}_{\emptyset,\underline{x}}} \\
			\Vect\otimes \BA_{\underline{x}} \ar[ru]^-{\on{unit}\otimes\on{Id}} \ar[rr]^-\simeq & &
			\BA_{\underline{x}}
		}
		\]
		\item For $\underline{x}\subset \underline{x}', \underline{y}\subset \underline{y}'$, a commutative diagram
		\[
		\xymatrix{
			\BA_{\underline{x}}\otimes  \BA_{\underline{y}}  \ar[rr]^-{\on{ins}_{\underline{x}\subset \underline{x}'} \otimes \on{ins}_{\underline{y}\subset \underline{y}'}} \ar[d]^-{\on{mult}_{\underline{x},\underline{y}}} & &
			\BA_{\underline{x}'}\otimes  \BA_{\underline{y}'} 
			\ar[d]^-{\on{mult}_{\underline{x}',\underline{y}'}} \\
			\BA_{\underline{x}\sqcup \underline{y}} 
			\ar[rr]_-{\on{ins}_{\underline{x}\sqcup \underline{y} \subset \underline{x}'\sqcup \underline{y}'}} & &
			\BA_{\underline{x}' \sqcup \underline{y}'};
		}
		\]
		\item Certain higher compatibilities similar to those for a commutative algebra.
	\end{itemize}
	
\end{rem}

\begin{df}\label{defn-laxunital-laxfactalg-cat}
	Parallel to Definition \ref{defn-laxunital-factalg-cat} we can define \emph{lax factorization algebra categories}, where we do not require $\on{mult}$ to be an isomoprhism.
\end{df}

\begin{df}[{\cite{CF}}] \label{defn-fact-functor}
	Given two (lax) factorization algebra categories $\BA_1,\BA_2$, a \emph{factorization functor from $\BA_1$ to $\BA_2$} consists of the following data:
	\begin{itemize}
		\item A morphism 
		\[\Phi:\BA_1 \to \BA_2\] 
		in $\mathbf{CrysCat}(\Ranp)$;
		\item A commutative diagram
		\begin{equation}
		\label{eqn-defn-fact-functor-1}
		\xymatrix{
			\Vect_{\on{pt}} \ar[r]^-{\on{unit}} \ar[d]_-= &
			\emptyset_{\on{pt}}^\bullet(\BA_1) \ar[d]^-{\emptyset_{\on{pt}}^\bullet(\Phi)} \\
			\Vect_{\on{pt}} \ar[r]^-{\on{unit}}  &
			\emptyset_{\on{pt}}^\bullet(\BA_2)
		}
		\end{equation}
		in $\mathbf{CrysCat}(\on{pt})$;
		\item A commutative diagram
		\begin{equation}
		\label{eqn-defn-fact-functor-2}
		\xymatrix{
			j^\bullet(\BA_1 \boxtimes \BA_1) \ar[r]^-{\on{mult}} \ar[d]_-{j^\bullet(\Phi\boxtimes \Phi)} & \on{add}^\bullet(\BA_1) \ar[d]^-{\on{add}^\bullet(\Phi)} \\
			j^\bullet(\BA_2 \boxtimes \bA_2) \ar[r]^-{\on{mult}}  & \on{add}^\bullet(\BA_2)
		}
		\end{equation}
		in $\mathbf{CrysCat}^{\lax}((\Ranp \times \Ranp)_{\on{disj}})$;
		\item Certain higher compatibilities similar to those for a morphism between commutative algebras.
	\end{itemize}
	
	\medskip
	
	A factorization functor is \emph{strictly unital} if the morphism $\Phi$ is contained in $\mathbf{CrysCat}^{\str}(\Ranp)$, i.e., is a strict morphism. 
	
	\medskip
	
	We will denote a factorization functor as above by $\Phi$ and treat other data as implicit.

	\medskip
	
	Natural transformations between factorization functors are defined in the obvious way. 
	
\end{df}

\begin{ntn} Let
	\[ \on{FactFun}(\BA_1,\BA_2) \]
	be the $\infty$-category of factorization functors $\BA_1 \to \BA_2$.
	
	\medskip
	Let
	\[
	\on{FactFun}_\str(\BA_1,\BA_2) \subset \on{FactFun}(\BA_1,\BA_2) 
	\]
	be the full subcategory of strictly unital ones.
\end{ntn}

\begin{ntn} \label{notn-2-cat-factalgcat}
	Let 
	\[\FactAlgCat \text{  } (\mathbf{LaxFactAlgCat})\]
	be the $(\infty,2)$-category of (lax) factorization algebra categories, factorization functors, and natural transformations.
	
	\medskip
	Let
	\[
	\FactAlgCat^\str \subset \FactAlgCat \text{  } (\mathbf{LaxFactAlgCat}^{\str} \subset \mathbf{LaxFactAlgCat})
	\]
	be the 1-full subcategory containing strictly unital factorization functors.
	
\end{ntn}

\subsubsection{Factorization algebras.}
\begin{df} \label{defn-fact-alg-obj}
	Let $\BA$ be a factorization algebra category. A \emph{factorization algebra $\cA$} in $\BA$ is a factorization functor $\Vect \to \BA$. Let
	\[
	\FactAlg(\BA):= \on{FactFun}(\Vect,\BA)
	\]
	be the $\infty$-category of factorization algebras in $\BA$.
	
\end{df}
\begin{rem}
	In the above definition, we do \emph{not} require the factorization functor $\Vect_{\Ranp} \to \BA$ to be strictly unital, i.e., it is only \emph{lax} unital. It is helpful to keep the following analog in mind: an associative algebra object in a monoidal $\infty$-category $\BC$ is the same as a (right) \emph{lax} monoidal functor $\Vect \to \BC$.
\end{rem}

\begin{rem} \label{rem-factalg-obj}
	Informally, a factorization algebra in $\BA$ consists of the following data:
	\begin{itemize}
		\item A global section of $\BA$, i.e., an object
		\[\cA \in \mathsf{\Gamma}(\Ranp,\BA);\]
		\item An equivalence 
		\[ \on{unit}(\mathsf{e}) \to \cA_\emptyset \]
		in $\BA_\emptyset$, given by the commutative diagram (\ref{eqn-defn-fact-functor-1});
		\item For any two finite sets $\underline{x},\underline{y}$ of closed points of $X$ such that $\underline{x}\cap \underline{y} = \emptyset$, an equivalence
		\[ \on{mult}_{\underline{x},\underline{y}}(\cA_{\underline{x}} \boxtimes \cA_{\underline{y}}) \to \cA_{\underline{x}\sqcup \underline{y}} \]
		in $\BA_{\underline{x}\sqcup \underline{y}}$, given by the commutative diagram (\ref{eqn-defn-fact-functor-2});
		\item For $\underline{x}\subset\underline{x}',\underline{y}\subset\underline{y}'$, a commutative diagram
		\[
		\xymatrix{
			\on{ins}_{\underline{x}\sqcup \underline{y} \subset \underline{x}'\sqcup \underline{y}'}\circ \on{mult}_{\underline{x},\underline{y}}(\cA_{\underline{x}} \boxtimes \cA_{\underline{y}}) \ar[d]^-\simeq \ar[rr]^-\simeq & &
			\on{ins}_{\underline{x}\sqcup \underline{y} \subset \underline{x}'\sqcup \underline{y}'}(\cA_{\underline{x}\sqcup \underline{y}}) \ar[d] \\
			\on{mult}_{\underline{x}',\underline{y}'}( \on{ins}_{\underline{x}\subset\underline{x}'}(\cA_{\underline{x}})\boxtimes \on{ins}_{\underline{y}\subset\underline{y}'}(\cA_{\underline{y}}) ) \ar[r] &
			\on{mult}_{\underline{x}',\underline{y}'}( \cA_{\underline{x}'}\boxtimes \cA_{\underline{y}'} ) \ar[r]^-\simeq & 
			\cA_{\underline{x}'\sqcup \underline{y}'},
		}
		\]
		where recall the left vertical arrow, i.e. the equivalence
		\[ \on{ins}_{\underline{x}\sqcup \underline{y} \subset \underline{x}'\sqcup \underline{y}'}\circ \on{mult}_{\underline{x},\underline{y}} \to \on{mult}_{\underline{x}',\underline{y}'} \circ ( \on{ins}_x\otimes \on{ins}_y), \]
		is part of the data in the definition of $\BA$, and the morphisms
		\[ \on{ins}_{\underline{x}\subset \underline{x}'}(\cA_{\underline{x}}) \to \cA_{\underline{x}'},\; \on{ins}_{\underline{y}\subset \underline{y}'}(\cA_{\underline{y}}) \to \cA_{\underline{y}'},\;\on{ins}_{\underline{x} \sqcup \underline{y} \subset \underline{x}'\sqcup \underline{y}'}(\cA_{\underline{x}\sqcup \underline{y}}) \to \cA_{\underline{x}'\sqcup \underline{y}'} \]
		are part of the data in the definition of $\cA$;
		\item Certain higher compatibilities.
	\end{itemize}
	
\end{rem}

\begin{cnstr} \label{constr-image-fact-alg}
	Let $\Phi: \BA_1 \to \BA_2$ be in $ \FactAlgCat$. We have a functor
	\[ \FactAlg(\BA_1) \to \FactAlg(\BA_2),\; \cA_1 \mapsto \Phi(\cA_1):=\Phi\circ \cA_1.\]
	We call $\Phi(\cA_1)$ the \emph{image of $\cA_1$ under $\Phi$}.
	
\end{cnstr}

\subsection{Factorization modules}

\begin{df} \label{defn-arrow-lax-prestack}
	Let $Y$ be a lax prestack. Define $\on{Arr}(Y)$ to be the lax prestack such that $$\on{Arr}(Y)(S):= \on{Arr}(Y(S)),$$ i.e. the category of arrows in $Y(S)$.
	
	\medskip
	
	Let $p_s,p_t:\on{Arr}(Y) \to Y$ be the morphisms remembering respectively the sources and the targets of the arrows.
\end{df}

\begin{df} \label{defn-mark-Ran}
	Let $x$ be a closed point of $X$ and $\underline{x}:\on{pt}\to \Ranp$ be the corresponding morphism. The \emph{Ran space with a marked point $x$} is the lax prestack $\Ranp$ defined as the following fiber product:
	\[
	\xymatrix{
		\Ranp \ar[r] \ar[d] & \on{Arrow}(\Ranp) \ar[d]^-{p_s} \\
		\on{pt} \ar[r]^-{\underline{x}} & \Ranp.
	}
	\]
	Unless otherwise stated, we always view $\Ranp$ as a lax prestack over $\Ranp$ by the composition 
	\[
	\Ranp \to \on{Arrow}(\Ranp) \xrightarrow{p_t} \Ranp.
	\]
\end{df}

\begin{rem}
	By definition, knowing a $k$-valued point of $\Ranp$ is equivalent to knowing a finite subset $\underline{y}$ of closed points of $X$ such that $\underline{x}\subset \underline{y}$.
	
\end{rem}

\begin{cnstr}
	\label{constr-Ranu-module} 
	
	For $\Ranp$ as above, let $(\Ranp \times \Ranp)_{\on{disj}}$ be the following fiber product:
	\[
	\xymatrix{
		(\Ranp \times \Ranp)_{\on{disj}} \ar[r]^-j \ar[d] & \Ranp \times \Ranp \ar[d]	\\
		(\Ranp \times \Ranp)_{\on{disj}} \ar[r]^-j & \Ranp \times \Ranp,
	}
	\]
	where we abuse notation by denoting the top horizontal morphism, which is an schematic open embedding, by the same letter $j$.
	
	\medskip
	
	We have an obvious schematic étale morphism $\on{add}: (\Ranp \times \Ranp)_{\on{disj}} \to \Ranp$.
	
	\medskip
	
	The correspondence
	\[
	\xymatrixrowsep{0.5cm}
	\xymatrixcolsep{0.5cm}
	\xymatrix{
		& (\Ranp \times \Ranp)_{\on{disj}} \ar[ld]_-{j} \ar[rd]^-{\on{add}} \\
		\Ranp \times \Ranp & &  \Ranp
	}
	\]
	makes $\Ranp$ a \emph{$\Ranp$-module object} in the category of lax prestacks with morphisms given by correspondences. 
\end{cnstr}

\subsubsection{Factorization module categories}

\begin{df} \label{defn-fact-mod-cat}
	Let $\BA$ be a factorization algebra category. A \emph{factorization $\BA$-module category supported on $\Ranp$} consists of the following data:
	\begin{itemize} 
		\item An object
		\[ \BC \in \mathbf{CrysCat}(\Ranp); \]
		\item An equivalence 
		\[
		\on{act}: j^\bullet(\BC\boxtimes \BA) \to \on{add}^\bullet(\BC)
		\]
		in $\mathbf{CrysCat}((\Ranp \times \Ranp)_{\on{disj}})$;
		\item Higher compatibilities between the above morphism and the $\Ranp$-module structure on $\Ranp$ defined in Construction \ref{constr-Ranu-module}.
	\end{itemize}
	
\end{df}

\begin{df}[\cite{CF}]\label{defn-core-factmod}
	Let $\BA$ be a factorization algebra category and $\BC$ be a factorization $\BA$-module category supported on $\Ranp$. We define the \emph{core of $\BC$} to be $\BC_x:=i^\bullet \BC$, where $i: \on{pt} \to \Ranp$ is the morphism corresponding to $\underline{x}\subset \underline{x}$.
	
	\medskip
	
	We also say $\BC_x$ is a \emph{factorization $\BA$-module category concentrated at $x$}.
\end{df}

\begin{df}
	Similar to Definition \ref{defn-fact-functor}, given a morphism $\Phi: \BA_1\to \BA_2$ in $\FactAlgCat$, and factorization $\BA_i$-module categories $\BC_i$, we can define the notion of \emph{factorization functors $\Xi:\BC_1\to \BC_2$ that intertwine the action of $\Phi$}. Let 
	\[
	\on{FactFun}_{\Phi}(\BC_1,\BC_2)
	\]
	be the $\infty$-category of such functors.
	
	\medskip
	When $\Phi$ is strictly unital, we say $\Xi$ is \emph{strictly unital} if its underlying morphism in $\mathbf{CrysCat}(\Ranp)$ is strict. Let
	\[
	\on{FactFun}_{\str,\Phi}(\BC_1,\BC_2) \subset \on{FactFun}_{\Phi}(\BC_1,\BC_2)
	\]
	be the full subcategory of strictly unital factorization functors.
	
	\medskip
	
	When $\Phi=\Id_\BA$ is the identity functor, we also write 
	\[
	\on{FactFun}_{\str,\BA}(\BC_1,\BC_2) \subset \on{FactFun}_{\BA}(\BC_1,\BC_2)
	\]
	for the above inclusion, call their objects \emph{strictly unital factorization $\BA$-linear functors}.

\end{df}

\begin{ntn} \label{factmodcat}
	Similar to Notation \ref{notn-2-cat-factalgcat}, let
	\[ \mathbf{FactModCat} \]
	be the $(\infty,2)$-category such that:
	\begin{itemize}
		\item
		An object is a pair $(\BA,\BC)$ such that $\BA$ is a factorization algebra category and $\BC$ is a factorization $\BA$-module category;
		\item
		A morphism is a pair $(\Phi,\Xi):(\BA_1,\BC_1) \to (\BA_2,\BC_2)$ of factorization functors such that $\Xi$ intertwines the action of $\Phi$.
	\end{itemize}
	
	\medskip
	
	Let
	\[
	\mathbf{FactModCat}^\str \subset \mathbf{FactModCat}  
	\]
	be the 1-full subcategory containing strictly unital factorization functors as morphisms.
	
	\medskip
	
	Consider the obvious forgetful functor
	\begin{equation}
	\label{eqn-factmod-to-factalg}
	\mathbf{FactModCat}  \to  \mathbf{FactAlgCat}
	\end{equation}
	For an object $\BA$ in the target, we denote the fiber at it to be
	\[ \BA\mathbf{-FactModCat}.\]
	
\end{ntn}

\begin{rem}
	By \cite[Lemma B.5.11]{CF} the similarly defined category $$\BA\mathbf{-FactModCat}^{\str}$$ coincides with  $\BA\mathbf{-FactModCat}.$ 
\end{rem}

\subsubsection{Factorization modules.}

\begin{df} Let $\BA$ be a factorization algebra category and $\BC$ be a factorization $\BA$-module category. Similar to Definition \ref{defn-fact-alg-obj}, we define
	\[
	\on{FactMod}(\BA,\BC):= \on{FactFun}( (\Vect, \Vect), (\BA,\BC)). 
	\]
	
	\medskip
	
	For a factorization algebra $\cA\in \FactAlg(\BA)$, define
	\[
	\cA\on{-FactMod}(\BC) := \on{FactMod}(\BA,\BC) \times_{ \FactAlg(\BA) } \{ \cA \},
	\]
	and call its objects \emph{factorization $\cA$-modules in $\BC$}.
\end{df}

\begin{df} For
	\[ (\cA,\cC)\in \on{FactMod}(\BA,\BC),\]
	the \emph{core of $\cC$} is defined to be the restriction of the corresponding global section along $i: \on{pt} \to \Ranp$, i.e., the object 
	\[
	\cC_x \in  \BC_x.
	\]
	We also say $\cC_x$ is a \emph{factorization $\cA$-module concentrated at $\BC_x$}.
\end{df}

\begin{lm}[{\cite{CF}}] \label{lem-mod-for-unit-factalg}
	Let $\BA$ be a factorization algebra category and $\BC$ be a factorization $\BA$-module category. Then the functor
	\[
	\on{unit}_\BA\on{-FactMod}(\BC) \to \BC_x,\; \cC \mapsto \cC_x
	\]
	is an equivalence.
\end{lm}

\subsubsection{Factorization restriction of modules.}
One of the important ingredients in the proof of Theorem \ref{intro theorem} is the following general result established in \cite{CF}.

\begin{thm}[{\cite{CF}}]
	\label{thmdefn-factres}
	Given a morphism $\Phi: \BA_1\to \BA_2$ in the target, there exists a natural  contravariant functor
	\[
	\mathbf{Res}_\Phi: \BA_2\mathbf{-FactMod} \to \BA_1\mathbf{-FactMod},
	\]
	we call it the \emph{factorization restriction functor along $\Phi$}. Moreover, for $\BC_2\in \BA_2\mathbf{-FactMod}$, there is a canonical equivalence
	\begin{equation}
	\label{eqn-thmdefn-factres}
	\mathbf{Res}_\Phi(\BC_2)_x \simeq \Phi(\on{unit}_{\BA_1})\on{-FactMod}(\BC_2).
	\end{equation}
	In particular, if $\Phi$ is strictly unital, then $\mathbf{Res}_\Phi(\BC_2)$ and $\BC_2$ have equivalent cores.
\end{thm}

\begin{pr}[{\cite{CF}}]
	\label{cor-factmod-in-factres}
	Let $\Phi:\BA\to \BA'$ be a morphism in $\mathbf{FactAlgCat}$ and $\BC' \in \BA'\mathbf{-FactMod}$. Let $\cA\in \FactAlg(\BA)$. Then there is a canonical equivalence
	\[
	\cA\on{-FactMod}( \mathbf{Res}_\Phi(\BC') ) \simeq \Phi(\cA)\on{-FactMod}(\BC').
	\]
\end{pr}

\begin{df}[{\cite{CF}}]
	Let $\BA$ be a factorization algebra category and $\BC$ be a factorization $\BA$-module category. Let $\phi:\cA_1\to \cA_2$ be a morphism in $\FactAlg(\BA)$, viewed as a 2-morphism in $\mathbf{FactAlgCat}$. Then we have a factorization $\Vect_\Ranp$-linear functor 
	\[
	\mathbf{Res}_{\cA_2}(\Vect_{\Ranp}) \to \mathbf{Res}_{\cA_1}(\Vect_{\Ranp}).
	\]
	The core of it is a functor
	\[
	\cA_2\on{-FactMod}(\BC) \to  \cA_1\on{-FactMod}(\BC),
	\]
	which we denote by $\on{Res}_\phi$ and call it the \emph{factorization restriction functor along $\phi$}.
	
\end{df}

\sssec{The basic adjunction}

Recall that for any $(\infty,2)$-category $\BE$, we have the notion of adjunctions in $\BE$. Namely, it consists of a diagram
\[
F: U \rightleftarrows V: G
\] 
and 2-morphisms $F\circ G \to \Id_V$, $\Id_U \to G\circ F$ satisfying the usual axioms. 

\begin{lm}[{\cite{CF}}]
	Let 
	\[
	\Phi: \BA \rightleftarrows \BA': \Psi
	\]
	be an adjunction in the 2-category $\mathbf{FactAlgCat}$. Then $\Phi$ is strictly unital.
\end{lm}

The following result follows from Theorem \ref{thmdefn-factres}.

\begin{thm}[{\cite{CF}}]
	\label{thm-basic-adj}
	Let $\Phi: \BA \rightleftarrows \BA': \Psi$ be an adjunction in the 2-category $\mathbf{FactAlgCat}$. Then there is a canonical adjunction in $\mathbf{BiCat}_\infty$:
	\[
	\mathbf{Res}_\Phi: \BA'\mathbf{-FactMod} \rightleftarrows \BA\mathbf{-FactMod}: \mathbf{Res}_\Psi.
	\]
\end{thm}

\begin{rem}
	Let $\Phi: \BA \rightleftarrows \BA': \Psi$ be an adjunction in the 2-category $\mathbf{FactAlgCat}$. Note that
	\begin{itemize}
		\item[(1)]
		For $\BC'\in \BA'\mathbf{-FactModCat}$, the core of $\mathbf{Res}_\Phi(\BC')$ is equivalent to that of $\BC'$ because $\Phi$ is strictly unital.
		\item[(2)]
		For $\BC\in \BA\mathbf{-FactModCat}$, the core of $\mathbf{Res}_\Psi(\BC)$ is equivalent to $\Psi(\on{unit}_{\bA})\on{-FactMod}(\BC)$.
	\end{itemize}
	Hence we can formulate the basic adjunction as the following equivalence:
	\[
	\on{FactFun}_{\BA}( \BC'_x, \BC_x   ) \simeq \on{FactFun}_{\BA'}( \BC'_x, \Psi(\on{unit}_{\BA})\on{-FactMod}_{\BC_x}),
	\]
	where we use the point of view of factorization module categories \emph{concentrated} at $x$ (rather than \emph{supported} on $\Ranp$).
	
\end{rem}

\subsection{Factorization patterns in families}\label{factfamilies} In the present paper we will need to use a slight generalization of the above formalism. Namely, we need factorization patterns in families parameterized by a fixed prestack $S$ with some condition that make technical work easier: we require $S$ to be 1-affine and the category $\QCoh(S)$ to be semi-rigid (see \cite[Appendix C]{AGKRRV2} for what this means).

First, we informally formulate a $S$-linear version of the notion of crystal of categories on a lax prestack.

\begin{df} Let $Y$ be a laft lax prestack. A \emph{$S$-linear crystal of categories} $\BA$ on $Y$ consists of the following data:
	\begin{itemize}
		\item For any $s:R\to Y$ with $R$ being a finite type affine scheme, there is a $\D(R)\otimes \QCoh(S)$-module DG category:
		\[ \BA_s \in \D(R)\otimes \QCoh(S)-\mathbf{ModCat}. \]
		
		\item For any 2-cell
		\[
		\xymatrix{
			& R \ar[rd]^-{s} \ar@{=>}[d]^-\theta \\
			T \ar[ru]^-v \ar[rr]_{t} & & Y,
		}
		\]
		there is a functor
		\[
		\theta_\dagger: \BA_s \to \BA_t
		\]
		intertwining the symmetric monoidal functor $$u^! \otimes \Id:\D(S)\otimes \QCoh(S) \to \D(T)\otimes \QCoh(S),$$ such that if $\theta$ is invertible, then the induced functor
		\[
		\D(T) \otimes_{\D(S)} \BA_s \to \BA_t
		\]
		is an equivalence.
		\item Certain higher compatibilities. 
	\end{itemize}
\end{df}

\begin{ntn}
	Let 
	\[\mathbf{CrysCat}^S(Y),\; \mathbf{CrysCat}^S(Y)^\str\] 
	be the $(\infty,2)$-category of crystals of categories on $Y$, with 1-morphisms given by morphisms (resp. strict morphisms) between crystals of categories, and 2-morphisms given by natural transformations.
\end{ntn}

\begin{rem}
		Replacing $$\mathbf{CrysCat} \text{ }(\text{resp. } \mathbf{CrysCat}(Y)^\str)$$ 
		by   $$\mathbf{CrysCat}^S \text{ }(\text{resp. } \mathbf{CrysCat}^S(Y)^\str)$$ everywhere we obtain the analogues of the definitions and statements presented in this section. We discuss the proofs and definitions in more detail in Appendix \ref{paramfact}.
\end{rem}

\begin{ex}
	Given $\BA \in \mathbf{FactAlgCat}$ we can construct
	$$\BA^S := \BA \otimes  \QCoh(S) \in  \mathbf{FactAlgCat}^S.$$
\end{ex}

\subsection{Commutative factorization algebra objects}

The identity morphism of $\Ranp$ upgrades to a map of commutative algebras given by the correspondence (\ref{corr ran}) and the map $\widetilde{\add}$. 

\begin{df}\label{multcryscat}
	A multiplicative $S$-linear crystal of categories $\BA$ on $\Ranp$ consists of the following data:
	\begin{itemize}
		\item An object $$\BA \in \mathbf{CrysCat}^S(\Ranp),$$
		\item An isomorphism
		\[\on{unit}:\Vect^S_{\on{pt}} \to \emptyset_{\on{pt}}^\bullet(\BA)\]
		in $\mathbf{CrysCat}^S(\on{pt})$;
		
		\item An isomorphism
		\[\on{mult}:\BA \boxtimes \BA \to {\widetilde{\add}}^\bullet(\BA)\] 
		in $\mathbf{CrysCat}^S((\Ranp\times \Ranp))$;
		
		\item Higher compatibilities.
	\end{itemize}
\end{df}

\begin{rem}
	Note that the data in Definition \ref{multcryscat} has underlying data of lax $S$-linear factorization algebra category.
\end{rem}

\begin{df}
	A commutative $S$-linear factorization algebra category is a multiplicative $S$-linear crystal of categories, whose underlying data of lax $S$-linear factorization algebra category is strict. Denote this category by $\mathbf{CFactAlgCat}^S$.
\end{df}

\begin{df}
	Similarly, for $\BA \in \mathbf{CFactAlgCat}^S$ we can define commutative $S$-linear factorization algebras in $\BA$. Denote this category by $\CFactAlg^S(\BA).$
\end{df}

\begin{ex}\label{ex: commutative fact cat}
	An element $\BC \in \CommAlg(\mathbf{DGCat})$ can be upgraded to an element in $\mathbf{CFactAlgCat}^S$, denoted by $\BC^S$.
\end{ex}

\begin{lm}[{\cite[ Lemma 8.10.1]{cpsii}}]
	For $\BC$ as in Example \ref{ex: commutative fact cat}, the restriction functor to $X$ induces an equivalence
$$ \fact: \CommAlg(\BC^S_{X})\xleftarrow{\sim} \CFactAlg^{S}(\BC^S),$$
where $\BC^S_{X}$ is the pullback of $\BC^S$ along $X \rightarrow \Ranp$.
\end{lm}

\subsection{A variant: modules over factorization algebra objects on $\Ranu$}\label{ranu factorization}

Fix $x\in X$, let $U = X \setminus x$. Recall that $\Ranu$ is a commutative algebra in the category of lax prestacks with morphisms given by correspondences. Similar to Construction \ref{constr-Ranu-module} we get that $\Ranp_{x}$ is a $\Ranu$-module. 

\begin{ntn}We set 
	\begin{enumerate}
		
		\item $\mathbf{FactAlgCat}^S(\Ranu)$ is the category of $S$-linear crystals of categories on $\Ranu$ with a structure of factorization algebra,
		\item $\mathbf{FactMod}^S_U$ is the category of pairs $(\BA_U, \BM)$, where $\BA_U \in \mathbf{FactAlgCat}^S(\Ranu)$ and $\BM$ is an $S$-linear crystal of categories on $\Ranp$ with the structure of factorization module category over $\BA_U$ (similar to \ref{defn-fact-mod-cat}),
		
		\item For $\mathbf{A}_U \in \mathbf{FactAlgCat}^S(\Ranu)$ set 
		$$\mathbf{A}_U-\mathbf{FactModCat}^S(\Ranp_{x}^{ch}),$$

		\item For $(\BA, \BM ) \in \mathbf{FactModCat}^S_U$ and $A_U \in \FactAlg^S(\BA)$ set 
		$$A_U\FactMod^S(\BM).$$
	\end{enumerate}
\end{ntn}

\begin{rem}\label{Ranu}
	All of the results from this section and Appendix \ref{paramfact} make sense in this setting, where the factorization algebra category/algebra is defined over $\Ranu$, and all the proofs go through on the nose. 
\end{rem}
	
\section{Statement of the result}\label{section: statement}

Since we are going to work only on the spectral side of the equivalence $\bL$, we will change the notation and replace $\check{G}$ by $G$. Let $X$ be a smooth affine curve, $x \in X$ be a closed point. Let $U = X \setminus x$.

Recall the statement of Conjecture \ref{intro conj}:
\begin{conj*}
	There exists  a fully faithful functor 
	$$\bFact: \QCoh(\LocSys_G(D^{\circ}))\mathbf{-ModCat} \rightarrow \Rep(G)\mathbf{-FactModCat}$$
	that is compatible with the forgetful functors from both sides to $\mathbf{DGCat}$.
	Here 
	\begin{itemize}
		\item[-] On the left-hand side $\QCoh(\LocSys_G(D^{\circ}))$ is viewed as a monoidal category with usual tensor product. The left-hand side is the $(\infty, 2)$-category of modules for this monoidal category.
		\item[-] On the right-hand side $\Rep(G)$ is the DG category of representations of $G$ equipped with factorization algebra category structure coming from the symmetric monoidal structure on $\Rep(G)$. 
	\end{itemize}
\end{conj*}

\subsection{Construction of the functor} In this subsection we present a construction of the functor $\bFact$ from Conjecture \ref{intro conj}.


\subsubsection{} Recall from \cite[9.4]{cpsii} the definition of the factorization lax prestack $\LocSys_G(D)_{\Randr}$: 

\begin{df}
	For any test affine scheme $S$ the category $\Omega^1_X(D)_{\Randr}(S)$ is defined as follows:
	\begin{itemize}
		\item An object is a pair $(s, \omega)$ consisting of an object $s \in \Randr(S)$ and $\omega \in \Gamma(D_{s}, \Omega^1_X \boxtimes \cO_S)$;
		\item A morphism from $(s, \omega)$ to $(s^{\prime}, \omega^{\prime})$ consists of a morphism $s \rightarrow s^{\prime}$ in $\Randr(S)$ and an isomorphism between $\omega$ and $\omega^{\prime}|_{D_s}$. 
	\end{itemize}
\end{df}

\begin{df}
	For any test affine scheme $S$ the category $G(D)_{\Randr}(S)$ is defined as follows:
	\begin{itemize}
		\item An object is a pair $(s, g)$ consisting of an object $s \in \Randr(S)$ and $g \in \Maps(D_s, G)$;
		\item  A morphism from $(s, g)$ to $(s^{\prime}, g^{\prime})$ consists of a morphism $s \rightarrow s^{\prime}$ in $\Randr(S)$ and an isomorphism between $\omega$ and $g^{\prime}|_{D_s}$. 
	\end{itemize}
\end{df}

\begin{df}
	We set $$\LocSys_G(D)_{\Randr}:= (\Lie(G) \otimes \Omega^1_X(D)_{\Randr}) / G(D)_{\Randr}.$$
\end{df}

By \cite[Lemma 9.8.1]{cpsii} we have an equivalence of factorization algebra categories:
\begin{equation}
	\QCoh(\LocSys_G(D)_{\Randr}) \cong \Rep(G).
\end{equation}

\subsubsection{} We will define $\bFact(\QCoh(\LocSys_G(D^{\circ})))$ as the category of quasi-coherent sheaves on the $\LocSys_G(D)_{\Randr}$-module lax prestack to be denoted $\LocSys_G(D^{\circ}_{\Ranp})$.

\begin{df}\label{forms factorizably}
	Define the prestack $\Omega^1(D^{\circ}_{\Ranp})$  degenerating $\Omega^1(D)$ to $\Omega^1(D^{\circ})$ as follows. For any test affine scheme $S$ the category $\Omega^1(D^{\circ}_{\Ranp})(S)$ is defined as follows:
	\begin{itemize}
		\item An object is a pair $(s, \omega)$ consisting of an object $s \in \Randrpt(S)$ and $\omega \in \Gamma(D_{s} \times_X U, \Omega^1_U \boxtimes \cO_S)$;
		\item A morphism from $(s, \omega)$ to $(s^{\prime}, \omega^{\prime})$ consists of a morphism $s \rightarrow s^{\prime}$ in $\Randrpt(S)$ and an isomorphism between $\omega$ and $\omega^{\prime}|_{(D_{s} \times_X U)}$. 
	\end{itemize}
\end{df}

\begin{df}
	Define the prestack  $G(D^{\circ}_{\Ranp})$ degenerating $G(D)$ to $G(D^{\circ})$ as follows.
	For any test affine scheme $S$ the category $G(D^{\circ}_{\Ranp})(S)$ is defined as follows:
	\begin{itemize}
		\item An object is a pair $(s, g)$ consisting of an object $s \in \Randrpt(S)$ and $g \in \Maps(D_{s} \times_X U, G)$;
		\item  A morphism from $(s, g)$ to $(s^{\prime}, g^{\prime})$ consists of a morphism $s \rightarrow s^{\prime}$ in $\Randrpt(S)$ and an isomorphism between $g$ and $g^{\prime}|_{D_{s} \times_X U}$. 
	\end{itemize}
\end{df}

\begin{df}
	Set $$\LocSys_G(D^{\circ}_{\Ranp}):=( \Lie(G) \otimes \Omega^1(D^{\circ}_{\Ranp}) )/ G(D^{\circ}_{\Ranp}).$$
\end{df}
 
 It is clear that $\LocSys_G(D^{\circ}_{\Ranp})$ has a factorization $\LocSys_G(D)_{\Randr}$-module structure. 

\begin{rem}
	Note that 
	\begin{equation}
	\begin{split}
		\LocSys_G(D^{\circ}_{\Ranp})_{x}(S) = \Lie(G) \otimes \Gamma(S \times ( X^{\wedge}_{x})^{\aff} \times_X U, \Omega_U^1 \boxtimes \cO_S) / \Hom(S \times ( X^{\wedge}_{x})^{\aff} \times_X U, G) \cong \\
		\cong \LocSys_G(D^{\circ})(S).
	\end{split}
	\end{equation}
\end{rem}

\begin{rem}
	Consider any $s \in \Randr(\Spec(A))$ and $x_A: \Spec(A) \rightarrow x \rightarrow \Randr$. Note that the map $D_{x_A} \rightarrow D_{s}$ induces a map $$\LocSys_G(D^{\circ}_{\Ranp}) \rightarrow \LocSys_G(D^{\circ}),$$
	and therefore endow $\bFact(\QCoh(\LocSys_G(D^{\circ}))):= \QCoh(\LocSys_G(D^{\circ}_{\Ranp}) )$ with an action of $\QCoh(\LocSys_G(D^{\circ}))$.
\end{rem}
 
 \subsubsection{} The functor in Conjecture \ref{intro conj} is defined by 
 \begin{equation}
 	\bFact(\BM) := \bFact(\QCoh(\LocSys_G(D^{\circ}))) \otimes_{\QCoh(\LocSys_G(D^{\circ}))} \BM.
 \end{equation}

 \subsection{Lisse sheaves} In this subsection we define and study lisse sheaves on $D^{\circ}$. Even though in this paper we work in the de Rham context, in his subsection we give the definition in all three frameworks: Betti, de Rham, and etale. In the etale context we relax the condition $\on{char}(k) = 0.$
 
 \begin{df}
 	\begin{itemize}
 		\item In the Betti context, let $\Lisse(D^{\circ})^{\heartsuit}$ denote the abelian category of local systems of finite rank on $S^1$, i.e. finite-dimensional vector spaces with an automorphism.
 		
 		\item  In the de Rham conext, let $\Lisse(D^{\circ})^{\heartsuit}$ denote the abelian category of finite rank free $k((t))$-modules $V$equipped with a $k$-linear map 
 		$$\nabla:V \rightarrow V dt$$
 		satisfying the Leibniz rule.  Denote by $\Lisse(D^{\circ})$ the bounded derived category of $\Lisse(D^{\circ})^{\heartsuit}$.
 		
 		\item In the etale context, let $\Lisse(D^{\circ})^{\heartsuit}$ denote the abelian category of lisse $\ell$-adic sheaves on $\Spec(k((t)))$. 
 	\end{itemize}
\end{df}

\begin{lm}\label{finite coh dim}
	The category $\Lisse(D^{\circ})^{\heartsuit}$ has finite cohomological dimension.
\end{lm}

\begin{proof}
	Let $\cO$ be the unit. We need to check that the functor $\Hom_{\Lisse(D^{\circ})}(\cO, -)$ has finite cohomological amplitude. But this functor commutes with colimits, hence it suffices show that for $\cE \in \Lisse(D^{\circ})^{\heartsuit}$ the complex $\Hom_{\QLisse(D^{\circ})}(\cO, \cE)$ is bounded.
	\begin{itemize}
	\item In the Betti context, this object obviously lies in $\Vect^{\geq 0, \leq 1}$. 
	\item In the de Rham context since we can compute cohomology via the de Rham complex this object lies in $\Vect^{\geq 0, \leq 1}$. 
	\item In the etale context, since $\Spec k((t))$ is a $K(\pi, 1)$-scheme, we can can compute $$\Hom_{\QLisse(D^{\circ})}(\cO, \cE)$$ as cohomology of $\pi_1^{\et}(\Spec k((t)))$ with coefficients in the corresponding representation $V$. Let us show that these cohomologies lie only in degrees 0 and 1. If $\cha(k) = 0$ the assertion is immediate. Assume that $\cha(k) = p$ for some prime $p \neq \ell$. The group $\pi_1^{\et}(\Spec k((t)))$ is an extension of $\prod_{q \neq p} \bZ_q$ with a pro-$p$ group of wild inertia. In particular, we have an exact sequence
	$$P \rightarrow \pi_1^{\et}(\Spec k((t))) \rightarrow \bZ_{\ell}$$
	where $P$ is a profinite group whose finite quotients have orders coprime to $q$. Then since $P$ does not have cohomologies with coefficients in $\bQ_{\ell}$-modules, we get that $$\Gamma(\pi_1^{\et}(\Spec e((t))), V ) \cong \Gamma(\bZ_{\ell}, V^P ),$$ 
	and the latter is computes by a two-term complex. 
\end{itemize}
\end{proof}
 
 \begin{df}
 	 Define $\IndLisse(D^{\circ}) := \Ind(\Lisse(D^{\circ}))$. The t-structure on $\Lisse(D^{\circ})$ uniquely extends to a t-structure on $\IndLisse(D^{\circ})$  compatible with filtered colimits.
 \end{df}

\begin{ntn}
	for an Abelian category $C^{\heartsuit}$ denote by $\Ind^{\heartsuit}(C^{\heartsuit})$ the Abelian category of functors between $C^{\heartsuit}$ and $\Sets$ that preserve fiber products.
\end{ntn}

\begin{lm}\label{properties of QLisse}
	The category $\IndLisse(D^{\circ})$ has the following properties:
	\begin{enumerate}
		\item $\IndLisse(D^{\circ})^\heartsuit \cong \Ind^{\heartsuit}(\Lisse(D^{\circ})^{\heartsuit})$,
		\item the category $\IndLisse(D^{\circ})$ is left-complete,
		\item the category $\IndLisse(D^{\circ})$ is the derived category of its heart, i.e. 
		$$\IndLisse(D^{\circ}) \xrightarrow{\sim} D(\Ind^{\heartsuit}(\Lisse(D^{\circ})^{\heartsuit})) := \QLisse(D^{\circ}).$$
	\end{enumerate}
\end{lm}

\begin{proof}
	Part (1) is immediate. Property (2) follows from \cite[Theorem E.1.3]{AGKRRV2} and Lemma \ref{finite coh dim}. We now prove (3). By (1) we have $$\IndLisse(D^{\circ})^{\geq -n} \rightarrow \QLisse(D^{\circ})^{\geq -n}$$
	for every $n$, and by (2) the source is left-complete. It suffices to check that the target is left-complete as well. We claim that for an Abelian category $\cA$ of finite cohomological dimension the category $D(\Ind^{\heartsuit}(\cA))$ is left-complete. Property (3) now follows from Lemma \ref{finite coh dim}. 
\end{proof}
 
 \begin{rem}\label{QLisse vs IndLisse}
 	The category $\QLisse(D^{\circ})$ mimics the category $\QLisse(X)$ of \cite{AGKRRV2}, which is defined as 
 	$$\QLisse(X) \cong D(\Ind^{\heartsuit}(\Lisse(X)^{\heartsuit})).$$
 \end{rem}

\subsubsection{}

In Betti context we have $\QLisse(D^{\circ})$ is obviously a Tannakian category, with a fiber functor given by $\ev_x$ for $x \in S^1$. Moreover, $\QLisse(D^{\circ}) \cong \QLisse(\bG_m)$, so all the desired properties of this Tannakian category were checked in \cite[1.7.4]{AGKRRV2}. We now study $\QLisse(D^{\circ})$ in de Rham and etale frameworks.

Recall that in de Rham context $\Lisse^{\heartsuit}(\bG_m)$ can be identified with coherent modules over $k[t, t^{-1}]$ with connection. Hence by restricting modules we obtain a map 
$$\Lisse^{\heartsuit}(\bG_m) \rightarrow \Lisse^{\heartsuit}(D^{\circ}).$$
In the etale context the inverse image functor defines 
$$\Lisse^{\heartsuit}(\bG_m) \rightarrow \Lisse^{\heartsuit}(D^{\circ}).$$
Since in both contexts it the map is  t-exact we get 
\begin{equation}\label{G_m to disc}
	\QLisse(\bG_m) \rightarrow \QLisse(D^{\circ}).
\end{equation}

\begin{pr}\label{section QLisse}
	There exists a symmetric monoidal $t$-exact section $$\QLisse(D^{\circ}) \rightarrow \QLisse(\bG_m)$$ of the map (\ref{G_m to disc}).
\end{pr}
\begin{proof}
	\begin{itemize}
	\item In the de Rham context, by \cite[Theorem 2.4.10]{Katz2} there exists a $t$-exact symmetric monoidal section $\Lisse(D^{\circ})^{\heartsuit} \rightarrow \Lisse(\bG_m)^{\heartsuit}$. This induces the desired section by Remark \ref{QLisse vs IndLisse}.
	\item In the etale context, the same agument works by \cite[Theorem 1.5.6]{Katz3}.
\end{itemize}
\end{proof}

\begin{rem}\label{fiber functor}
	For any point $x \in \bG_m$ the composition $$\QLisse(D^{\circ}) \rightarrow \QLisse(\bG_m) \xrightarrow{\text{ev}_x} \Vect$$ defines a fiber functor, making  $\QLisse(D^{\circ})$ a Tannakian category.
\end{rem}

In order to apply the formalism of \cite{AGKRRV2} in our local context we need to show that $\QLisse(D^{\circ})$ is a well-behaved Tannakian category. Namely, recall the notion of {\it gentle Tannakian category} from \cite[1.7.2]{AGKRRV2}:

Let $\BH$ be a symmetric monoidal category, equipped with a t-structure and a 
conservative t-exact symmetric monoidal functor $\oblv_\BH$ to $\Vect$. Note that the assumptions on $\oblv_\BH$ imply that the t-structure on $\BH$
is compatible with filtered colimits and right-complete. We will additionally assume that $\BH$ is left-complete in its t-structure.  
\begin{df}
Let $\BH$ be as above. We will say that $\BH$ is a {\it gentle Tannakian category}  if the following conditions are satisfied:

\begin{itemize}
	
	\item The following classes of objects in $\BH^\heartsuit$ coincide: 
	
	\smallskip
	
	(i) Objects contained in $\BH^c\cap \BH^\heartsuit$;
	
	\smallskip
	
	(ii) Objects that are sent to compact objects in $\Vect$ by $\oblv_\BH$;
	
	\smallskip
	
	(iii) Dualizable objects.
	
	\medskip
	
	\item The object $\one_\BH$ has the following properties:
	
	\smallskip
	
	(i) The functor $\Hom_\BH(\one_\BH,-)$ has a finite cohomological amplitude;
	
	\smallskip
	
	(ii) For any $\Bh\in \BH^c\cap \BH^\heartsuit$, the cohomologies of $\Hom_\BH(\one_\BH,\Bh)\in \Vect$
	are finite-dimensional.
	
	\medskip
	
	\item The category $\BH^\heartsuit$ is generated under colimits by 
	$\BH^c\cap \BH^\heartsuit$.
	
\end{itemize}
\end{df}

 \begin{lm}\label{gentle}
  The Tannakian category $\QLisse(D^{\circ})$ is gentle.
 \end{lm}

\begin{proof}
	\begin{itemize}
	\item We first treat te de Rham situation. We begin by checking the first condition. Note that $\QLisse(D^{\circ})^{\heartsuit} \cong \Ind(\Lisse(D^{\circ})^{\heartsuit})$.
	The class of objects $\QLisse(D^{\circ})^c \cap \QLisse(D^{\circ})^{\heartsuit}$ coincides with $\Lisse(D^{\circ})^{\heartsuit}$. 
	
	By construction the class of objects in $\QLisse(D^{\circ})^{\heartsuit} \cong \Ind(\Lisse(D^{\circ})^{\heartsuit})$ that are sent to finite-dimensional vector spaces under the fiber functor from Remark \ref{fiber functor} coincides with $\Lisse(D^{\circ})^{\heartsuit}$. 
	
	Every object of $\Lisse(D^{\circ})^{\heartsuit}$ in $\QLisse(D^{\circ})^{\heartsuit}$ is dualizable, and since every dualizable object is compact the class of dualizable objects coincides with $\Lisse(D^{\circ})^{\heartsuit}$. Hence we checked the first condition in the definition of a gentle Tannakian category.
	
	Let us check the second condition. Let $\cO$ be the unit. We need to check that the functor $\Hom_{\QLisse(D^{\circ})}(\cO, -)$ has finite cohomological amplitude. But this functor commutes with colimits, hence it suffices show that for $\cE \in \Lisse(D^{\circ})^{\heartsuit}$ the complex $\Hom_{\QLisse(D^{\circ})}(\cO, \cE)$ is bounded. The latter statement follows from Lemma \ref{finite coh dim}. We also need to show that cohomologies of $\Hom_{\QLisse(D^{\circ})}(\cO, \cE)$ are finite-dimensional. By Levelt-Turritin theorem (\cite[Theorem 2.1]{BBDE}) we can reduce to the case when $\cE$ is one-dimensional. Let us first treat the case when $\cE$ is non-trivial. Then $H^0$ vanishes tautologically, and $H^1$ is the dual space to the zeroth cohomology of the de Rham complex for $\cE^{\vee}$, hence also vanishes. The case of trivial $\cE$ it is a direct computation to show that both $H^0$ and $H^1$ are one-dimensional. 
	
	The third condition that $\QLisse(D^{\circ})^{\heartsuit}$ is generated under colimits by $\QLisse(D^{\circ})^c \cap \QLisse(D^{\circ})^{\heartsuit}$ is immediate from the fact that the former is $\Ind(\Lisse(D^{\circ})^{\heartsuit})$ and the latter is $\Lisse(D^{\circ})^{\heartsuit}$. 
	
	Lastly, recall that we proved in Lemma \ref{properties of QLisse} that $\QLisse(D^{\circ})$ is left-complete.
	\item In the etale situation, the same arguments work for all desired conditions apart from part (ii) of the second one. However, the fact that cohomologies of $\Hom_{\QLisse(D^{\circ})}(\cO, \cE)$ are finite-dimensional follows from our description given in Lemma \ref{finite coh dim}.
\end{itemize}
\end{proof}

\subsection{Definition and geometry of $\LocSys_G^{\restr}(D^{\circ})$}
 In this subsection we define and study the stack $\LocSys_G^{\restr}(D^{\circ})$. Again, in his subsection we work in all three frameworks: Betti, de Rham, and etale. 

\begin{df}
	We define the prestack of local systems with restricted variation on the punctured disk $\LocSys_G^{\restr}(D^{\circ})(S)$ by sending an affine scheme $S$ to the space of right t-exact symmetric monoidal functors $$\Rep(G) \rightarrow \QLisse(D^{\circ}) \otimes \QCoh(S).$$ 
\end{df}

\begin{rem}\label{AGKRRV}
	Recall that for a gentle Tannakian category $\BH$ in \cite[1.8.1]{AGKRRV2}  the stack $\mathbf{Maps}(\Rep(G), \BH)$  is defined such that for $S$  an affine scheme $S$-points of  $\mathbf{Maps}(\Rep(G), \BH)$  are right t-exact symmetric monoidal functors $$\Rep(G) \rightarrow \BH \otimes \QCoh(S).$$ 
	By \ref{gentle} the stack $\LocSys_G^{\restr}(D^{\circ})$ is a particular case of the stack $\mathbf{Maps}(\Rep(G), \BH)$ defined in \cite{AGKRRV2}.
\end{rem}

\begin{lm}\label{section LS^restr}
	There exists a section of the restriction map $\LocSys_G^{\restr}(\bG_m) \rightarrow \LocSys_G^{\restr}(D^{\circ})$.
\end{lm}

\begin{proof}
		Follows from Lemma \ref{section QLisse}.
\end{proof}

	By Remark \ref{AGKRRV} the stack $\LocSys_G^{\restr}(D^{\circ})$ enjoys the properties of the stack $\mathbf{Maps}(\Rep(G), \BH)$ establised in \cite{AGKRRV2}. We list them in the rest of this subsection.

\begin{pr}\label{properties of LS}
The prestack $\LocSys_G^{\restr}(D^{\circ})$ is an etale stack, equal to a disjoint union of etale stacks each of which can be written as an etale-sheafified quotient by $G$ of an etale stack $Y$ with the following properties:
\begin{itemize}
	\item[a)] $Y$ is locally almost of finite type;
	\item[b)] $Y^{red}$ is an affine, classical, reduced scheme;
	\item[c)] $Y$ is an ind-scheme.
	\item[d)] $Y$ is a formal affine scheme in the sence of \cite[Remark 1.4.7]{AGKRRV2}.
\end{itemize}
\end{pr}

\begin{proof}
	Follows from Lemma \ref{gentle} and \cite[Theorem 1.8.3]{AGKRRV2}.
\end{proof}

\begin{cor}\label{locally almost of finite type}
	$\LocSys_G^{\restr}(D^{\circ})$  is a prestack locally almost of finite type.
\end{cor}

\begin{proof}
	
	By [Chapter 2 Corollary 2.5.7]\cite{GR1}, \cite[Chapter 2 Corollary 2.7.5(a)]{GR1} and \cite[Chapter 2 Corollary 2.7.10]{GR1} for every laft prestack $\cX$ such that $^{\leq n}{\cX}$ is truncated $^{\leq n}{\cX^{\et}}$ is locally of finite type, where $\cX^{\et}$ stands for etale sheafification. 
	
	Now since $Y$ is locally almost of finite type the prestack quotient of $Y$ by $G$ is locally almost of finite type. Moreover, since $Y$ is and ind-scheme by \cite[Chapter 2 Lemma 1.2.3]{GR2} $^{\leq n}Y$ is truncated, and hence the prestack quotient of $^{\leq n}Y$ by $^{\leq n}G$ is truncated.
	
	It is left to notice that by \cite[2.1.4] {AGKRRV2} the stack $\LocSys_G^{\restr}(D^{\circ})$  is convergent.
\end{proof}

\begin{cor}[{\cite[Proposition 7.8.2, Corollary 7.8.8]{AGKRRV2}}]\label{1-affinnes of LS^r}
	The stack $\LocSys_G^{\restr}(D^{\circ})$  is 1-affine. 
\end{cor}

\begin{lm}[{\cite[Corollary 7.9.2, Theorem 7.9.6]{AGKRRV2}}]\label{semi-regidity compact generation}
	The category $\QCoh(\LocSys_G^{\restr}(D^{\circ}))$ is a compactly generated semi-rigid category.
\end{lm}

\subsection{$\LocSys_G^{\restr}(D^{\circ})$ vs $\LocSys_G(D^{\circ})$}
The first goal of this subsection is to establish the following statement:
\begin{thm}\label{map between LS}
	There exists a natural morphism $\LocSys_G^{\restr}(D^{\circ}) \rightarrow  \LocSys_G(D^{\circ})^{\et}$,
	where $(-)^{\et}$ stands for sheafification in etale topology.
\end{thm}

\begin{rem}
	Note that there are morphisms
	$$\LocSys_G^{\restr}(D^{\circ}) \rightarrow  \LocSys_G^{\restr}(\bG_m) \rightarrow  \LocSys_G(\bG_m),$$
	where the last prestack classifies all de Rham $G$-local systems on the open curve $\bG_m$. The difficulty to be overcome is that $\LocSys_G(\bG_m)$ does {\it not} map to $\LocSys_G(D^{\circ})^{\et}$. 
\end{rem}

\subsubsection{} Let $\Bun_G(\bG_m)^+ \subset \Bun_G(\bG_m)$ denote the substack whose $S$-points are principal $G$-bundles on $S \times \bG_m$ which admit an extension to $S \times \bA^1$ etale locally on $S$. By \cite[B.7.15]{GLC} we do have a canonical morphism
$$\LocSys_G(\bG_m)^+ \rightarrow \LocSys_G(D^{\circ})^{\et},$$
where $$\LocSys_G(\bG_m)^+  = \Bun_G(\bG_m)^+  \times_{\Bun_G(\bG_m)} \LocSys_G(\bG_m).$$

Then Theorem \ref{map between LS} is equivalent to the following statement:

\begin{thm}\label{map between LS 2}
	The composite
	\begin{equation}\label{map between LS 2: equation}
		\LocSys_G^{\restr}(D^{\circ}) \rightarrow\LocSys_G^{\restr}(\bG_m) \rightarrow \LocSys_G(\bG_m) \rightarrow \Bun_G(\bG_m)
	\end{equation}
	factors through $\Bun_G(\bG_m)^+$.
\end{thm}

\subsubsection{} Let $H$ be the group associated to $\QLisse(D^{\circ})$ and its fiber functor by the Tannakian
formalism. Two homomorphisms out of $H$ of particular interest to us.
\begin{enumerate}
	\item The map $$H \rightarrow \bG_a$$ corresponding to the full subcategory
	$$\Rep(\bG_a) \cong \QLisse(D^{\circ})_{\triv} \hookrightarrow \QLisse(D^{\circ})$$ generated by the unit $\cO$.
	\item By Chevellay’s theorem, the full abelian subcategory $$\QLisse(D^{\circ})^{\heartsuit, ss} \hookrightarrow\QLisse(D^{\circ})^{\heartsuit}$$ of semi-simple objects is a tensor subcategory. Under the Tannakian formalism, it corresponds to the pro-reductive quotient $$H \rightarrow H^{\redd}.$$
	
\end{enumerate}

\begin{lm}\label{diffeential galois group}
	The map $$H \rightarrow \bG_a \times H^{\redd}$$ is an isomorphism.
\end{lm}

\begin{proof}
	We need to show that the functor
	$$\Rep(H^{\redd}\times \bG_a)^{\heartsuit} \rightarrow \Rep(H)^{\heartsuit}$$ is an equivalence of categories. It is essentially surjective by the Levelt–Turritin theorem (\cite[Theorem 2.1]{BBDE}). To see that
it is fully faithful, note that as $H^{\redd}$ is reductive, every indecomposible object of the LHS
takes the form
$$V \boxtimes J_n,$$
where $V$ is an indecomposible object of $\Rep(H^{\redd})^{\heartsuit}$ and $J_n$ is a Jordan block. 
\end{proof}

\begin{proof}[Proof of Theorem \ref{map between LS 2}]
	Let $\mathbf{Maps}(\Rep(G), \Rep(H^{\redd}))$ denote the prestack obtained by replacing $\QLisse(D^{\circ})$ with $\Rep(H^{\redd})$ in the definition of $\LocSys_G^{\restr}(D^{\circ})$. By Lemma \ref{diffeential galois group} we have a group homomorphism $H^{\redd} \rightarrow H$, which induces 
	$$\LocSys_G^{\restr}(D^{\circ}) \rightarrow \mathbf{Maps}(\Rep(G), \Rep(H^{\redd})).$$
	We claim that the map (\ref{map between LS 2: equation}) factors through $\mathbf{Maps}(\Rep(G), \Rep(H^{\redd}))$. For this it is sufficient to show that the symmetric monoidal functor
	$$\Rep(H^{\redd}) \otimes \Rep(\bG_a) \cong \QLisse(D^{\circ}) \rightarrow \QLisse(\bG_m) \rightarrow D(\bG_m) \rightarrow \QCoh(\bG_m)$$
	factors through $\Rep(H^{\redd})$. This is clear because every $\bG_a$-torsor on $\bG_m$ is trivial.
	
	It remains to show that the morphism
	$$\mathbf{Maps}(\Rep(G), \Rep(H^{\redd})) \rightarrow \Bun_G(\bG_m)$$
	factors through $\Bun_G(\bG_m)^+$. The category $\Rep(H^{\redd})$ is gentle Tannakian so we can apply results from \cite{AGKRRV2} to $\mathbf{Maps}(\Rep(G), \Rep(H^{\redd}))$. Since $H^{\redd}$ is reductive, by \cite[Proposition 3.5.4]{AGKRRV2}, connected components of  $\mathbf{Maps}(\Rep(G), \Rep(H^{\redd}))$ have the form $\pt / \Stab_G(\varphi)$ for some $\varphi: Im(H) \rightarrow G$. Thus we only need to check that the composites
	$$\pt \rightarrow \pt / \Stab_G(\varphi) \rightarrow \Bun_G(\bG_m)$$
	factor through $\Bun_G(\bG_m)^+$. This holds because every $G$-bundle on $\bG_m$ is trivial for a connected $G$.
\end{proof}

\subsubsection{} 

One of the structural result that is not known yet is the following
\begin{conj}\label{conj LS}
	The stack $\LocSys_G^{\restr}(D^{\circ})$ is a disjoint union of formal completions of a collection of pairwise non-intersecting closed substacks of $(\LocSys_G(D^{\circ})^{\et})^{\redd}$.
\end{conj}

\subsubsection{} In the rest of this subsection we a result in the direction of Conjecture \ref{conj LS}, which will be useful for us in the future. 

\begin{pr}\label{monomorphism}
		The map $$\LocSys_G^{\restr}(D^{\circ}) \rightarrow \LocSys_G(D^{\circ})^{\et}$$ is a monomorphism (i.e., is a monomorphism of spaces when evaluated
		on any affine scheme).
\end{pr}

\begin{proof}
	Note that by Example \ref{main example discs} it suffices to show that $$\LocSys_G^{\restr}(D^{\circ}) \rightarrow (\hormerxJets(\BG))_{x}$$  is a monomorphism. We first check the condition on classical affine schemes. 	It is enough to show that $\QCoh(\Spec(A)) \otimes \QLisse(D^{\circ})$ embeds fully faithfully into the category of $A((t))$-modules with an $A$-linear connection. To show this notice that $\QCoh(\Spec(A)) \otimes \QLisse(D^{\circ})$  is compactly generated by objects of the form $A \boxtimes (c[i])$ with $c \in \Lisse(D^{\circ})^{\heartsuit}$, $i \in \bZ$. So to show fully faithfulness it suffices to prove that the functor sends $A \boxtimes c$ with $c\in \Lisse(D^{\circ})^{\heartsuit}$ to compacts in $A((t))$-modules with connection, and that the functor is fully faithful on the full subcategory spanned by these objects.
	
	The first claim follows from the fact that the right adjoint to $A((t)) \otimes_{k((t))} (-)$ commutes with infinite direct sums.
	
	
	The second claim says that 
	$$\Hom_{\QCoh(\Spec(A)) \otimes \QLisse(D^{\circ})} (A \boxtimes c_1, A \boxtimes c_2) \cong \Hom_{A((t))\Mmod \text{with} \nabla}(A((t)) \otimes_{k((t))} c_1, A((t)) \otimes_{k((t))} c_2)$$
	with $c_i \in \Lisse(D^{\circ})^{\heartsuit}$. 
	This is equivalent to proving the equivalence
	$$\Hom_{\QLisse(D^{\circ})}(c_1, A\otimes c_2) \cong \Hom_{ {A((t))\Mmod \text{with} \nabla}}(c_1, A((t)) \otimes_{k((t))}c_2).$$
	By \cite[Theorem 2.1] {BBDE}we can reduce to the case when $c_1$ is irreducible, and then it is dualizable. Hence it suffices to prove
	$$\Hom_{k((t))\Mmod \text{with} \nabla}(\cO, A\otimes c_2) \cong \Hom_{k((t))\Mmod \text{with} \nabla}(\cO, A((t)) \otimes_{k((t))}c_2),$$
	where $\cO$ is $k((t))$ with trivial connection. 
	Using \cite[Theorem 2.1]{BBDE} again we reduce to the case when $c_2$ is irreducible. In that case $c_2 \cong \pi_* \cL$ for $\pi$ a finite covering of the disc and $\cL$ one-dimensional local system. Then by adjunction it suffices to prove the equivalence for $c_2$ one-dimensional, which is a direct computation. 
	
	Note that by construction the prestack $(\hormerxJets(\BG))_{x}$ is convergent, hence to prove the proposition it is enough to show that for a truncated affine scheme $\Spec(A)$ the natural map 
	\begin{equation}\label{monomorphism: formal iso}
		\LocSys_G^{\restr}(D^{\circ})(\pi_0(A)) \times_{(\hormerxJets(\BG))_{x}(\pi_0(A)) }(\hormerxJets(\BG))_{x}(A) \leftarrow 	\LocSys_G^{\restr}(D^{\circ})(A) 
	\end{equation}
	is an isomorphism. If $A$ is 0-truncated the claim is immediate. If $A$ is 1-truncated note that for $f \in \LocSys_G^{\restr}(D^{\circ})(\pi_0(A))$ we have 
	$$\{f \}\times_{(\hormerxJets(\BG))_{x}(\pi_0(A)) }(\hormerxJets(\BG))_{x}(A) \cong \Gamma(D_{\pi_0(A), \nabla} \times_{\bA^1_{\dR}} \bG_{m, \dR}, f^* \mathfrak{g}[1] \otimes_{\pi_0(A)} \pi_1(A)[1]).$$
	On the other side of (\ref{monomorphism: formal iso}) we have 
	$$\{f\} \times_{\LocSys_G^{\restr}(D^{\circ})(\pi_0(A))} \LocSys_G^{\restr}(D^{\circ})(A) \cong \Gamma_{dR}(\pi_0(A)((t)), f^*\mathfrak{g}[1] \otimes_{\pi_0(A)} \pi_1(A)[1]),$$
	and the isomorphism (\ref{monomorphism: formal iso}) follows from Proposition \ref{de rham of disc modules with connection}. For $n$-truncated $A$ we conclude the claim by induction.
	
\end{proof}

\subsection{Main result}

\begin{thm}\label{mainthm}
	There exists  a fully faithful functor 
	$$\bFact^{\restr}: \QCoh(\LocSys_G^{\restr}(D^{\circ}))\mathbf{-ModCat} \rightarrow \Rep(G)\mathbf{-FactModCat}$$
	that commutes with the forgetful functors from both sides to $\mathbf{DGCat}$.
	Here 
	\begin{itemize}
		\item[-] On the left-hand side $\QCoh(\LocSys_G^{\restr}(D^{\circ}))$ is viewed as a monoidal category with usual tensor product. The left-hand side is the $(\infty, 2)$-category of modules for this monoidal category.
		\item[-] On the right-hand side $\Rep(G)$ is the DG category of representation of $G$ equipped with factorization algebra category structure coming from symmetric monoidal structure on $\Rep(G)$. 
	\end{itemize}
\end{thm}

\subsubsection{}\label{construction of the main functor} As in Conjecture \ref{intro conj}, the functor $\bFact^{\restr}$ is given by a certain factorization $\Rep(G)$-module category, $\bFact^{\restr}(\QCoh(\LocSys_G^{\restr}(D^{\circ})))$, equipped with a commuting action of $\QCoh(\LocSys_G^{\restr}(D^{\circ}))$. We define $\bFact^{\restr}(\QCoh(\LocSys_G^{\restr}(D^{\circ})))$ as the category of quasi-coherent sheaves on the prestack $$\LocSys_G(D^{\circ}_{\Ranp})^{\et} \times_{\LocSys_G(D^{\circ})^{\et}}\LocSys_G^{\restr}(D^{\circ}).$$

\subsection{Connection with Conjecture 1} The contents of this subsection are dependent on Conjecture \ref{conj LS}.

\subsubsection{} Conjecture \ref{conj LS} implies that the pushforward along the map 
$$\LocSys_G^{\restr}(D^{\circ}) \rightarrow \LocSys_G(D^{\circ})^{\et}$$
admits arbitrary base change, therefore

\begin{pr}\label{product formula}
	For any prestack $\cY$ mapping to $\LocSys_G(D^{\circ})^{\et}$ one has 
	$$\QCoh(\LocSys_G^{\restr}(D^{\circ})) \otimes_{\QCoh((\LocSys_G(D^{\circ}) ^{\et})} \QCoh(\cY) \cong \QCoh(\LocSys_G^{\restr}(D^{\circ}) \times_{\LocSys_G(D^{\circ})^{\et}} \cY).$$
\end{pr}

We are now ready to establish connection with Conjecture \ref{intro conj}.

\begin{lm}\label{Connection lemma}
	The restriction functor $$R: \QCoh(\LocSys_G^{\restr}(D^{\circ}))\mathbf{-ModCat}\rightarrow \QCoh(\LocSys_G(D^{\circ}))\mathbf{-ModCat}$$ is fully faithful. 
\end{lm}

\begin{proof}
	For any $\BC_1, \BC_2 \in \QCoh(\LocSys_G^{\restr}(D^{\circ}))\Mmod$ we need to show that
	$$\Maps_{\QCoh(\LocSys_G^{\restr}(D^{\circ}))\mathbf{-ModCat}}(\BC_1, \BC_2) \cong \Maps_{\QCoh(\LocSys_G(D^{\circ}))\mathbf{-ModCat}}(R(\BC_1), R(\BC_2)).$$
	We can write  $\BC_1 \cong \colim\QCoh(\LocSys_G^{\restr}(D^{\circ})) \otimes \BD_i$ for some plain DG categories $\BD_i$. Since $R$ is continuous we may assume that $\BC_1 \cong \QCoh(\LocSys_G^{\restr}(D^{\circ})) \otimes \BD$ for a plain DG category $\BD$. Then
	$$\Maps_{\QCoh(\LocSys_G(D^{\circ}))\mathbf{-ModCat}}(R(\BC_1), R(\BC_2)) \cong $$
	$$\cong \Maps_{\QCoh(\LocSys_G^{\restr}(D^{\circ}))\mathbf{-ModCat}}(\QCoh(\LocSys_G^{\restr}(D^{\circ})) \otimes_{\QCoh(\LocSys_G(D^{\circ}))} \QCoh(\LocSys_G^{\restr}(D^{\circ})) \otimes \BD, \BC_2).$$
	By Proposition \ref{product formula} the latter is equivalent to 
	$$\Maps_{\QCoh(\LocSys_G^{\restr}(D^{\circ}))\mathbf{-ModCat}}(\QCoh(\LocSys_G^{\restr}(D^{\circ}) \times_{\LocSys_G(D^{\circ})^{\et}} \LocSys_G^{\restr}(D^{\circ})) \otimes \BD, \BC_2).$$
Then Proposition \ref{monomorphism} implies the desired claim.
\end{proof}

\subsubsection{} Combining Proposition \ref{product formula} and Lemma \ref{Connection lemma}, we get 
\begin{lm}
	The composite functor  $$\bFact \circ R: \QCoh(\LocSys_G^{\restr}(D^{\circ}))\mathbf{-ModCat} \rightarrow \Rep(G)\mathbf{-FactModCat}$$
	coincides with $\bFact^{\restr}$.
\end{lm}

\section{Reductions}\label{section: reductions}

\subsection{Reduction to the case of $\bA^1$}

In this subsection we reduce Theorem \ref{mainthm} to the case when $X = \mathbb{A}^1$, $x = 0$. 
 We first record the following corollary of Theorem \ref{main etale pullback}:
 
 \begin{cor}
 	Let $f: X \rightarrow Y$ be a map of curves with  $x \in X, y \in Y$ such that $f(x) = y$ and $f$ is etale at $x$. 
 	Then we have 
 	\begin{equation}\label{pullback etale at a point}
 	\Rep(G)-\mathbf{FactModCat}(\Ranp_{Y, y})  \cong  \Rep(G)-\mathbf{FactModCat}(\Ranp_{X, x}). 
 	\end{equation}
 \end{cor}

\subsubsection{} Note that every affine smooth curve $X$ admits a map $f: X \rightarrow \bA^1$ etale at $x \in X$. So we see that to reduce the problem to the case of  $\bA^1$ it suffices to check that the following diagram commutes: 

\begin{equation}\label{compatibility curves diagram}
	\begin{tikzcd}
	&   	\Rep(G)-\mathbf{FactModCat}(\Ranp_{\bA^1, 0}) \ar[d, dash, "\cong"']  \\
	\QCoh(\LocSys_G^{restr}(D^{\circ}))\mathbf{-ModCat}  \ar[ur, rightarrow, ""']\ar[r, rightarrow, ""']&   \Rep(G)-\mathbf{FactModCat}(\Ranp_{X, x}).
	\end{tikzcd}
\end{equation}

We conclude commutativity of (\ref{compatibility curves diagram}) by the following statement:

\begin{lm}
	Under the isomorphism (\ref{pullback etale at a point}) the object $$\QCoh({\LocSys_G}(D^{\circ}_{\Ranp_{\bA^1}}))$$ maps to 
	$$\QCoh({\LocSys_G}(D^{\circ}_{\Ranp_{X}})).$$
\end{lm}

\begin{proof}
	Without loss of generality we may assume that $f$ is etale and surjective, and $f^{-1}(y)=\{x\}$. Recall the definition of prestacks in question via Example \ref{main example discs}. Let $({X_{\dR}^I \times x})^{\circ}$ be the pullback

		\[
	\begin{tikzcd}
	({X_{\dR}^I \times x})^{\circ}\ar[r, rightarrow, ""']\ar[d, rightarrow, ""']&   	X_{\dR}^I \times x \ar[d, rightarrow, ""']  \\
	\stackrel{\circ}{{\Ranp}_{X_{\dR}, x}} \ar[r, rightarrow, ""']&   \Ranp_{X_{\dR}, x}.
	\end{tikzcd}
	\]
	Our goal is to show that 
	$$\hormerxJets(\BG^{\naive})_{X_{\dR}^I \times x} \times_{(X_{\dR}^I \times x)} ({X_{\dR}^I \times x})^{\circ}$$
	is isomorphic to 
	$$\hormerxJets(\BG^{\naive})_{Y_{\dR}^I \times y} \times_{(Y_{\dR}^I \times y)} ({X_{\dR}^I \times x})^{\circ}.$$
	Let $x^I: S \rightarrow ({X_{\dR}^I \times x})^{\circ} \subset X_{\dR}^I \times x$ and denote by $y^I$ composition $S \rightarrow Y_{\dR}^I \times y$. We will show that 
	\begin{equation}
		D_{x^I, \nabla} \times_{X_{\dR}} (X - x)_{\dR} \cong D_{y^I, \nabla}  \times_{Y_{\dR}} (Y - y)_{\dR}.
	\end{equation}

To do that it suffices to show that there exists 
\begin{equation}
	D_{x^I, \nabla}  \xrightarrow{\cong}D_{y^I, \nabla} 
\end{equation} that makes the diagram 

	\[
\begin{tikzcd}
D_{x^I, \nabla} \ar[r, rightarrow, ""']\ar[d, rightarrow, "\cong"']&   	X_{\dR} \ar[d, rightarrow, ""']  \\
D_{y^I, \nabla}  \ar[r, rightarrow, ""']&   Y_{\dR}
\end{tikzcd}
\]
commutative. From the Definition \ref{fake} it follows that we need to show that $\Gamma_f: \Gamma_{x^{I, \redu}} \xrightarrow{\cong} \Gamma_{y^{I, \redu}}$. We claim that this map is finite. Indeed, we have 
	\[
\begin{tikzcd}
\bigsqcup_{i \in (I+1)} S^{\redu} \ar[r, rightarrow, "\pi_x"']\ar[dr, rightarrow, "\pi_y"']&   	\Gamma_{x^{I, \redu}} \ar[d, rightarrow, ""']  \\
 &   \Gamma_{y^{I, \redu}}
\end{tikzcd}
\]
with the obvious maps $\pi_x$ and $\pi_y$. Now we note that all schemes are affine, $\cO_{\Gamma_{y^{I, \redu}}}$ is Noetherian, $\pi_y$ is finite, and $\cO_{\Gamma_{x^{I, \redu}}} \subset \cO_{\bigsqcup_{i \in (I+1)} S^{\redu}}$, which gives that $\Gamma_f$ is finite. It is left to check that $\Gamma_f$ is bijection of $k$-points, but this follows from the fact that $x$ mapped to $({X_{\dR}^I \times x})^{\circ} \subset X_{\dR}^I \times x$.
\end{proof}

\begin{cor}
	Under the isomorphism (\ref{pullback etale at a point}) the object $$\QCoh(\LocSys_G^{\restr}(D^{\circ}) \times_{\LocSys_G(D^{\circ})} {\LocSys_G}(D^{\circ}_{\Ranp_{\bA^1}}))$$ maps to 
	$$\QCoh(\LocSys_G^{\restr}(D^{\circ}) \times_{\LocSys_G(D^{\circ})} {\LocSys_G}(D^{\circ}_{\Ranp_{X}})).$$
\end{cor}

\begin{rem}
	This argument reduced Conjecture \ref{intro conj} to the case of $X = \bA^1, x= 0$.
\end{rem}
For the duration of the main body of the paper we fix a notation $X = \bA^1$, $x= 0$, $U = \bA^1 \setminus x$.

\subsection{Reduction to the affine case I}
Let $\BC_1$ and $\BC_2$ be DG categories acted on by $\QCoh(\LocSys_G^{\restr}(D^{\circ})$.
Theorem \ref{mainthm} asserts that the functor 
\begin{equation}\label{mainassertion}
\on{Fun}_{\QCoh(\LocSys_G^{\restr}(D^{\circ}))}(\BC^1, \BC^2 ) \rightarrow  \on{FactFun}_{\Rep(G)}(\bFact^{\restr}(\BC^1), \bFact^{\restr}(\BC^2))
\end{equation}
is an equivalence.

\subsubsection{} In this subsection we will show that it suffices to prove (\ref{mainassertion}) when $\BC^1 \cong \QCoh(R_1)$ with $R_1$ an affine scheme, where the action of $\QCoh(\LocSys_G^{\restr}(D^{\circ})$ comes from a pullback along some map $R_1 \rightarrow \LocSys_G^{\restr}(D^{\circ})$. We begin with the following immediate property:
\begin{lm}
	The functor $\bFact^{\restr}$ commutes with small colimits.
\end{lm}

\subsubsection{} Thus we may assume that $\BC^1 \cong \QCoh( \LocSys_G^{\restr}(D^{\circ})) \otimes \BC_0^1$, where $\BC_0^1$ is a plain DG category. 
Suppose $\BM^1:=\BN\otimes \BM_0$ such that the action of $\QCoh( \LocSys_G^{\restr}(D^{\circ})) $ comes from the first factor. Then we have
$$\on{Fun}_{\QCoh(\LocSys_G^{\restr}(D^{\circ}))}(\BM^1,\BM^2)  \simeq \on{Fun}(\BM_0,\on{Fun}_{\QCoh(\LocSys_G^{\restr}(D^{\circ}))}(\BN,\BM^2))$$
and
\[
\on{FactFun}_{\Rep(G)}( \bFact^{\restr}(\BM^1),\bFact^{\restr}(\BM^2) ) \simeq \on{Fun}(\BM_0,\on{FactFun}_{\Rep(G)}( \bFact^{\restr}(\BN),\bFact^{\restr}(\BM^2) )).
\]

\subsubsection{} Thus we may assume that $\BC_0^1$ is $\Vect$. 
By  Proposition \ref{properties of LS} we have $\LocSys_G^{\restr}(D^{\circ}) \cong ((\sqcup Y_i)/G)^{\et}$, where $Y_i$ are formal affine schemes, so we can write $$ \QCoh( \LocSys_G^{\restr}(D^{\circ})) \cong   \QCoh(\sqcup Y_i)^{G, w},$$
where superscript $w$ stands for weak invariants in the sense of \cite{B}.

\begin{lm}\label{inv vs coinv}
	Let $\BA$ be a symmetric monoidal category, $\BC \in \BA\mathbf{-ModCat}$. Assume also that $\BC$ admits a weak action of a group $G$ of finite type. Then $$\BC^{G, w} \cong \BC_{G, w}$$ as $\BA$-modules.
\end{lm}

\begin{proof}
	Follows directly from the case $\BA = \Vect$ (see \cite[Corollary 2.3.5]{B}). 
\end{proof}

\subsubsection{} Hence by Lemma \ref{inv vs coinv} we can reduce to the case $\BC^1 \cong \QCoh(\sqcup Y_i)_{G, w}$. Then we can write $\QCoh(\sqcup Y_i)_{G, w}$ as a colimit of the bar resolution

\begin{equation}\label{bar}
	\begin{tikzcd}\ldots \arrow[r, shift right=1.80ex]\arrow[r, shift left=1.80ex]  \arrow[r, shift left=0.65ex] \arrow[r, shift right=0.65ex] & \QCoh(G) \otimes \QCoh(G) \otimes \QCoh(\sqcup Y_i)\arrow[r, shift left=1.3ex] \arrow[r, shift right=1.3ex] \arrow[r] &\QCoh(G) \otimes \QCoh(\sqcup Y_i) \arrow[r, shift left=0.65ex] \arrow[r, shift right=0.65ex] & \QCoh(\sqcup Y_i), \end{tikzcd}
\end{equation}
where the arrows are given by action, multiplication, and trivial action. Notice that since  $\QCoh( \LocSys_G^{\restr}(D^{\circ}))$ is semi-rigid all arrows in the diagram (\ref{bar}) are $ \QCoh( \LocSys_G^{\restr}(D^{\circ}))$-linear. Therefore we may assume that $$\BC^1 \cong \QCoh(G^n) \otimes\QCoh(\sqcup Y_i),$$ and thus that $$\BC^1 \cong \QCoh(G^{n} \times Y_i).$$

\subsubsection{} Now each $Y_i \times G^n$ is a formal affine scheme, so by \cite[Corollary 7.3.4]{AGKRRV2} we can write$$ \QCoh(G^{n} \times Y_i) \cong  \colim_{\AffSch_{/Y_i \times G^n}} \QCoh(S_i)$$ with affine $S_i$, where the maps are given by pushforwards. Using semi-rigidity of $\QCoh( \LocSys_G^{\restr}(D^{\circ}))$ we see that these maps are $ \QCoh( \LocSys_G^{\restr}(D^{\circ})$-linear, hence we obtain the desired reduction.

\subsection{Reduction to the affine case II} In this subsection we will further reduce to the case when $\BC^2 \cong \QCoh(R_2) \otimes \BC_0^2$. Here $R_2$ is an affine scheme, the action of $\QCoh(\LocSys_G^{\restr}(D^{\circ}))$ comes from a map $$R_2 \rightarrow \LocSys_G^{\restr}(D^{\circ}),$$ and $\BC_0^2$ is a plain DG category.

\begin{pr}\label{functor commutes with limits}
	The functor $\bFact^{\restr}$ commutes with small limits.
\end{pr}

\begin{lm}\label{retsriction and dualizability}
	Let $S$ be a prestack such that $\QCoh(S)$ is semi-rigid. 
		Let 
	\[
	\Phi: \BA \rightleftarrows \BA': \Psi
	\]
	be an adjunction in $\mathbf{FactAlgCat}^S$. Let $\BC' \in \BA'\mathbf{-FactMod}^S$. Assume that restrictions of $\BA$ and $\BA'$ to $X_{\dR}^I $ for every $I$ and restriction of $\BC'$ to every $X_{\dR}^I \times \{x\}$ for every $I$ are dualizable as plain DG categories. 
	Then restriction of $\mathbf{Res}_{\Phi}(\BC')$ to $X_{\dR}^I \times \{x\} $ is dualizable as a plain DG category for every $I$. 
\end{lm}

\begin{proof}
	By Remark \ref{restrictino as a limit} and the conditions on  $\BC'$, $\BA$ and $\BA'$, and \cite[Lemma A.3.4]{C} we know that $\mathbf{Res}_{\Phi}(\BC')_{X_{\dR}^I \times \{x\}}$ is given as a limit of dualizable categories. Then the desired assertion follows from Corollary \ref{factres linear}.
\end{proof}

\begin{proof}[Proof of Proposition \ref{functor commutes with limits}]
	Recall (see \ref{construction of the main functor}) that $$\bFact^{\restr}(\BM):= \BM \otimes_{\QCoh(\LocSys_G^{\restr}(D^{\circ}))} \bFact^{\restr}(\QCoh(\LocSys_G^{\restr}(D^{\circ}))).$$
	Since by Lemma \ref{semi-regidity compact generation}  $\QCoh(\LocSys_G^{\restr}(D^{\circ}))$ is semi-rigid  it suffices to check that $$\bFact^{\restr}(\QCoh(\LocSys_G^{\restr}(D^{\circ})))_{X_{\dR}^I \times \{x\}}$$ is dualizable as a plain DG category for every $I$. 
	
	We claim that by Proposition \ref{restriction} we are in position to use Lemma \ref{retsriction and dualizability}. Indeed, the category $\Vect^{\LocSys_G^{\restr}(D^{\circ})}$ is dualizable on every stratum since by Lemma \ref{semi-regidity compact generation} the category $\QCoh(\LocSys_G^{\restr}(D^{\circ}))$ is dualizable. Dualizability of $$\Rep(G) ^{\LocSys_G^{\restr}(D^{\circ})}$$ on every strata follows from \cite[Corollary 6.18.2]{cpsi}and \cite[Lemma A.3.4]{C}.
	\end{proof}

\subsubsection{} By Barr-Beck-Lurie  theorem we can write $\BC^2 $ as a colimit of the diagram 

$${\begin{tikzcd}\ldots \arrow[r, shift left=1.3ex] \arrow[r, shift right=1.3ex] \arrow[r] &\QCoh( \LocSys_G^{\restr}(D^{\circ})) \otimes\QCoh( \LocSys_G^{\restr}(D^{\circ})) \otimes \BC^2\arrow[r, shift left=0.65ex] \arrow[r, shift right=0.65ex] & \QCoh( \LocSys_G^{\restr}(D^{\circ})) \otimes \BC^2, \end{tikzcd}}$$
where the maps are given by action and multiplication on $\QCoh( \LocSys_G^{\restr}(D^{\circ})) $. By definition of semi-rigidity and by  \cite[Lemma C.5.3]{AGKRRV2} these maps admit continuous $\QCoh( \LocSys_G^{\restr}(D^{\circ})) $-linear right adjoints. 

\begin{lm}\label{lims and colims in modules}
	Let $\BA$ be a semi-rigid symmetric monoidal category. Let $I$ be an $\infty$-category, let $i \mapsto \BC_i$ be a functor $I \rightarrow \BA\mathbf{-ModCat}$. For $\alpha: i \rightarrow j \in I$ denote the corresponding functor between $\BC_i$ and $\BC_j$ by $F_{\alpha}$. Assume that for each $\alpha$ the functor $$\Oblv(F_{\alpha}): \Oblv(\BC_i) \rightarrow \Oblv(\BC_j)$$ admits continuous right adjoint $\Oblv(G_{\alpha})$, which is then upgrades to an $\BA$-linear functor $G_{\alpha}$ by semi-rigidity. Then
	$$\underset{I, F}{\colim \text{ }}  \BC_i \cong \lim_{I^{\op}, G} \BC_i$$
	in $\BA\Mmod$. 
\end{lm}

\begin{proof}
	Let us first check that the underlying DG categories are equivalent. Since $$\Oblv:\BA\mathbf{-ModCat}\rightarrow \DGCat_{\text{cont}}$$ commutes with colimits we have $$\Oblv(\underset{I, F}{\colim \text{ }}  \BC_i ) \cong \underset{I, \Oblv(F)}{\colim \text{ }} \Oblv(\BC_i).$$
	On the other hand since $\Oblv$ also commutes with limits we have that 
	$$\Oblv(\lim_{I^{\op}, G} \BC_i) \cong \lim_{I^{\op}, \Oblv(G)} \Oblv(\BC_i).$$
	But by \cite[Lemma 1.3.3]{G}we get 
	\begin{equation}\label{equation lemma on limis and colimits}
		\underset{I, \Oblv(F)}{\colim \text{ }} \Oblv(\BC_i) \cong \lim_{I^{\op}, \Oblv(G)} \Oblv(\BC_i),
	\end{equation}
	where the functor from left-hand side to right-hand side corresponds to the system of functors 
	$$\prescript{\prime}{}{\ins_{j, \DGCat_{\text{cont}}}}: \Oblv(\BC_j) \rightarrow \lim_{I^{\op}, \Oblv(G)} \Oblv(\BC_i),$$
	which are left adjoints to tautological functors $$\ev_{j, \DGCat_{\text{cont}}}: \lim_{I^{\op}, \Oblv(G)} \Oblv(\BC_i)  \rightarrow \Oblv(\BC_j).$$
	The latter functors are images under $\Oblv: \BA\mathbf{-ModCat} \rightarrow \DGCat_{\text{cont}}$ of 
	$$\ev_{j, \BA\mathbf{-ModCat}}: \lim_{I^{\op}, G} \BC_i  \rightarrow \BC_j,$$
	hence by semi-rigidity left adjoints $\prescript{\prime}{}{\ins_i}$ upgrade to $\BA$-linear functors left adjoint to $\ev_{j, \BA\mathbf{-ModCat}}$.
	
	Therefore there exists an $\BA$-linear map $$\underset{I, F}{\colim \text{ }}  \BC_i  \rightarrow \lim_{I^{\op}, G} \BC_i$$ such that its image under $\Oblv$ is (\ref{equation lemma on limis and colimits}). But since (\ref{equation lemma on limis and colimits}) is an equivalence and $\Oblv$ is conservative we are done. 
	
\end{proof}

\subsubsection{} By Lemma \ref{lims and colims in modules} we can write $\BC^2$ as a small limit in $\QCoh(\LocSys_G^{\restr}(D^{\circ}))\mathbf{-ModCat}$ of categories of the form $$\QCoh( \LocSys_G^{\restr}(D^{\circ})) \otimes \BD_i,$$ where each $\BD_i$ is a plain DG category. Hence by Lemma \ref{functor commutes with limits} we can reduce to the case $$\BC^2 \cong \QCoh( \LocSys_G^{\restr}(D^{\circ})) \otimes \BC^2_0,$$
were $\BC^2_0$ is a plain DG category. 
Note that we have 
\begin{equation*}
	\begin{split}
	\QCoh( \LocSys_G^{\restr}(D^{\circ})) \otimes \BC^2_0 &\cong  \\
	\cong &\on{Fun}_{\QCoh( \LocSys_G^{\restr}(D^{\circ}))}(\QCoh( \LocSys_G^{\restr}(D^{\circ})), \BC^2_0 \otimes \QCoh( \LocSys_G^{\restr}(D^{\circ})))
	\end{split}
\end{equation*}
as $\QCoh( \LocSys_G^{\restr}(D^{\circ}))$-modules. As in the previous step, we can rewrite  $\QCoh( \LocSys_G^{\restr}(D^{\circ}))$ as $$\colim  \QCoh(S_i)$$ with affine $S_i$. Hence
\begin{equation*}
	\begin{split}
	\QCoh( \LocSys_G^{\restr}(D^{\circ})) \otimes \BC^2_0 & \cong \\
	\cong &\on{Fun}_{\QCoh( \LocSys_G^{\restr}(D^{\circ}))}(\underset{}{\colim} \QCoh(S_i), \BC^2_0 \otimes \QCoh( \LocSys_G^{\restr}(D^{\circ})))
	\end{split}
\end{equation*}
 as modules over $\QCoh( \LocSys_G^{\restr}(D^{\circ}))$ , which can be rewritten as 
 $$\lim\on{Fun}_{\QCoh( \LocSys_G^{\restr}(D^{\circ}))}(\QCoh(S_i), \BC^2_0 \otimes \QCoh( \LocSys_G^{\restr}(D^{\circ}))) \cong \lim \QCoh(S_i) \otimes \BC^2_0.$$
 Here the last equality follows from the fact that $\QCoh(S_i)$ is self-dual as $\QCoh( \LocSys_G^{\restr}(D^{\circ}))$ -module (\cite[C.2.7]{AGKRRV2}). Therefore by Lemma \ref{functor commutes with limits} we obtained the desired reduction.

 \subsubsection{} So equation (\ref{mainassertion}) becomes
 
\begin{equation}\label{main assertion step 2}
\begin{split} 
& \QCoh(R_1) \otimes_{\QCoh(\LocSys_G^{\restr}(D^{\circ}))}  \QCoh(R_2) \otimes \BC_0^2  \cong \\
& \cong \on{FactFun}_{\Rep(G)}(\bFact^{\restr}(\QCoh(R_1) ), \bFact^{\restr}(\QCoh(R_2) \otimes \BC_0^2)).
\end{split}
\end{equation}
Since $\LocSys_G^{\restr}(D^{\circ})$ is 1-affine (Lemma \ref{1-affinnes of LS^r}) and $\QCoh(\LocSys_G^{\restr}(D^{\circ}))$ is semi-rigid, we can use \cite[Lemma 7.8.11] {AGKRRV2}to rewrite left-hand side of the equation above as 
$$\QCoh(R_1 \times_{\LocSys_G^{\restr}(D^{\circ}) }R_2) \otimes  \BC_0^2.$$

\subsection{Factorization modules for the twisted regular representation} 
The rest of the section will be dealing with rewriting right-hand side of the  equation (\ref{main assertion step 2}). To simplify the notation we assume that $\BC_0^2 = \Vect$, and denote by $T$ the affine scheme $\Spec(R_1 \times R_2)$. 

\subsubsection{}Then by Corollary \ref{ambidextrous adjunction S-linear} and Corollary \ref{ambidextrous adjunction S-linear 2} we can rewrite the right-hand side of  (\ref{main assertion step 2}) as 
\begin{equation}\label{right-hand side of the main assertion step 3}
	\on{FactFun}_{\Rep(G)}^T(\bFact^{\restr}(R_1) \otimes \QCoh(R_2), 
	\bFact^{\restr}(R_2 )\otimes \QCoh(R_1)). 
\end{equation}
We will use Theorem-Definition \ref{Basic Adjunction} to simplify (\ref{right-hand side of the main assertion step 3}). We will realize the source and the target as parameterized factorization restriction.

\subsubsection{} Recall from Lemma \ref{section LS^restr} we have a map $$\LocSys_G^{\restr}(D^{\circ}) \rightarrow \LocSys_G^{\restr}(U),$$
which gives a map 
$$U_{\dR} \times \LocSys_G^{\restr}(D^{\circ}) \rightarrow \BG \times U_{\dR}\times \LocSys_G^{\restr}(D^{\circ}),$$
where the map to the third component is projection.
This induces  a map 
$$\Jets_{\nabla,  \LocSys_G^{\restr}(D^{\circ}) }(U_{\dR} \times \LocSys_G^{\restr}(D^{\circ}) ) \rightarrow \Jets_{\nabla,  \LocSys_G^{\restr}(D^{\circ}) }( \BG \times U_{\dR}\times \LocSys_G^{\restr}(D^{\circ})).$$
Pulling back we obtain a map of $\LocSys_G^{\restr}(D^{\circ})$-linear factorization algebra  categories
$$\Phi^{\univ}: \Rep(G)^{\LocSys_G^{\restr}(D^{\circ})} \rightarrow \Vect^{\LocSys_G^{\restr}(D^{\circ})}.$$
Let $f_{\sigma_1}$ denote the morphism $R_1\rightarrow \LocSys_G^{\restr}(D^{\circ})$. Restricting $\Phi^{\univ}$ to $R_1$ we get 
$$\Phi^{\univ}_{1}: \Rep(G)^{R_1} \rightarrow \Vect^{R_1}.$$
let $$\Phi_1:= \Phi^{\univ}_{1} \otimes \QCoh(R_2):\Rep(G)^{T} \rightarrow \Vect^{T}.$$

\begin{rem}
	Note that $\Phi^{\univ}$ (and thus $\Phi_1$) admit continuous right adjoints given by pushforward of quasi-coherent sheaves.
\end{rem}

\begin{pr}\label{restriction}
	There exists an equivalence 
	$$\mathbf{Res}_{\Phi^{\univ}}^{\LocSys_G^{\restr}(D^{\circ})}(\Vect^{\LocSys_G^{\restr}(D^{\circ})}) \cong \mathbf{Fact}^{\restr}(\QCoh(\LocSys_G^{\restr}(D^{\circ})))$$
as elements in $\Rep(G)^{\LocSys_G^{\restr}(D^{\circ})}\mathbf{-FactModCat}^{\LocSys_G^{\restr}(D^{\circ})}$.
\end{pr}

\begin{cor}\label{restriction R}
	There exists an equivalence 
	$$\mathbf{Res}_{\Phi_1}^{T}(\Vect^{T}) \cong \bFact^{\restr}(\QCoh(R_1) ) \otimes \QCoh(R_2)$$
	as elements in $\Rep(G)^{T}\mathbf{-FactModCat}^{T}$.
	
\end{cor}

\subsubsection{} We postpone the proof of the Proposition \ref{restriction} to the Subsection \ref{proof of geometric prop} and finish the description of (\ref{right-hand side of the main assertion step 3}).
It follows that we have
\begin{equation*}
\begin{split}
	 &  \on{FactFun}_{\Rep(G)}^T(\bFact^{\restr}(R_1) \otimes \QCoh(R_2), 
	\bFact^{\restr}(R_2 )\otimes \QCoh(R_1)). \\
	\cong  &\on{FactFun}_{\Rep(G)}^T(\mathbf{Res}_{\Phi_1}(\Vect^T), 
	\bFact^{\restr}(R_2)\otimes \QCoh(R_1)), \\
	 \cong  &\on{FactFun}_{\Vect}^T(\Vect^T, 
	\mathbf{Res}_{\Phi_1^R}(\bFact^{\restr}(R_2)\otimes \QCoh(R_1)) )\\
	 \cong	&\mathbf{Res}_{\Phi_1^R}(\bFact^{\restr}(R_2)\otimes \QCoh(R_1))\\
	 \cong &\Phi_{1}^R( \omega_{U} \boxtimes \cO_{T})\FactMod^T (\bFact^{\restr}(R_2)\otimes \QCoh(R_1)),
	\end{split}
\end{equation*}
where
\begin{itemize}
	\item
	The first equivalence is due to Proposition \ref{restriction R}.
	\item
	The second equivalence is due to Theorem-Definition \ref{Basic Adjunction}.
	\item
	The third equivalence is due to \cite[Lemma B.6.6]{text}.
	\item
	The last equivalence is due to Remark \ref{Basic Adjunction Rem}.
\end{itemize}
Finally, note that by Proposition \ref{restriction R} we have 
$$\bFact^{\restr}(R_2)\otimes \QCoh(R_1) \cong \mathbf{Res}_{\Phi_2}(\Vect^T),$$
so by Corollary \ref{Basic Adjunction Cor} the expression (\ref{right-hand side of the main assertion step 3}) becomes

$$\Phi_2\circ \Phi_{1}^R( \omega_{U} \boxtimes \cO_{T})\FactMod^T (\Vect^T).$$

\subsubsection{}Now we give an explicit description of $\Phi_2 \circ \Phi_{1}^R (\omega_{U} \boxtimes \cO_{T})$. 
Then chasing the diagram 

\[
\begin{tikzcd}
\CommAlg(\D(U) \otimes \QCoh(T) ) \arrow[d, "{F_{\sigma_1}}_*"] \ar[r, "\sim", "\fact"']&   \CFactAlg^{T}(\Vect^T)\arrow[d, "\Phi_{1}^R"]    \\
\CommAlg(\Rep(G) \otimes \D(U) \otimes \QCoh(T) )  \ar[r, "\sim", "\fact"']\arrow[d, "{F_{\sigma_2}}_*"]  &  \CFactAlg^{T}(\Rep(G)^T) \arrow[d, "\Phi_{2}"]\\
\CommAlg(\D(U) \otimes \QCoh(T))  \ar[r, "\sim", "\fact"']&   \CFactAlg^{T}(\Vect^T)
\end{tikzcd}
\]
we obtain 

$$\Phi_2 \circ \Phi_{1}^R (\omega_{\Ranu} \boxtimes \cO_{T}) \cong \fact ({F_{\sigma_2}}^*{F_{\sigma_1}}_* (\omega_{U} \boxtimes \cO_{T})).$$

\subsubsection{} So we reduced the statement (\ref{mainassertion}) to 
\begin{equation}\label{main assertion after reductions}
\QCoh(R_1 \times_{\LocSys_G^{\restr}(D^{\circ}) }R_2)
\cong{F_{\sigma_2}}^*{F_{\sigma_1}}_* (\omega_{U} \boxtimes \cO_{T})\FactMod^T(\Vect^T).
\end{equation}

\subsection{Proof of Proposition 4.4.4}\label{proof of geometric prop}

\begin{pr}\label{extension + torsor locally trivial on A}
	Let $S$ be an affine scheme with a map $\sigma: S \rightarrow \LocSys_G^{\restr}(D^{\circ})$. Consider a post-composition of this map with the section  $\LocSys_G^{\restr}(D^{\circ}) \rightarrow  \LocSys_G^{\restr}(U)$ and the corresponding morphism $$f_{\sigma}: (U)_{\dR} \times S \rightarrow \BG.$$
	Then etale locally on $S$ the morphism $f_{\sigma}$ factors through $\BG^{\naive}$. 
\end{pr}

\begin{proof}
	Since $\BG^{\naive} \rightarrow \BG$ is formally etale it suffices to show that the underlying $G$-torsor on $U \times S$ is trivial etale locally on $S$. Recall from the proof of Theorem \ref{map between LS 2} that the map 
	$$\LocSys_G^{\restr}(D^{\circ}) \rightarrow\LocSys_G^{\restr}(\bG_m) \rightarrow \Bun_G(\bG_m)$$ factors through $$\mathbf{Maps}(\Rep(G), \Rep(H^{\redd})) \rightarrow \Bun_G(\bG_m).$$
	Then, as in the proof of Theorem \ref{map between LS 2}, we notice that connected components of  $$\mathbf{Maps}(\Rep(G), \Rep(H^{\redd}))$$ have the form $\pt / \Stab_G(\varphi)$ for some $\varphi: Im(H) \rightarrow G$. Thus we only need to check that the composites
	$$\pt \rightarrow \pt / \Stab_G(\varphi) \rightarrow \Bun_G(\bG_m)$$
	define trivial $G$-bundles. This holds because every $G$-bundle on $\bG_m$ is trivial for a connected $G$.
\end{proof}
\subsubsection{} We will construct an adjoint pair of functors between $\Vect^{\LocSys_G^{\restr}(D^{\circ}) }$ and $$\bFact^{\restr}(\QCoh(\LocSys_G^{\restr}(D^{\circ}))$$ given by pullback and pushforward along a certain map
\begin{equation}\label{main section}
\Psi_s: 	\LocSys_G^{\restr}(D^{\circ})  \times (\Randrpt)^{\et} \rightarrow 	\LocSys_G^{\restr}(D^{\circ}) \times_{\LocSys_G(D^{\circ})^{\et}} \LocSys_G(D^{\circ}_{\Ranp})^{\et}.
\end{equation}
Let us now construct $\Psi_s$. 

\begin{cnstr}
	Consider an affine scheme $S =\Spec(A)$ with maps $\sigma: S \rightarrow \LocSys_G^{\restr}(D^{\circ})$ and $t: S \rightarrow  (\Randrpt)^{\et}$. By Proposition \ref{extension + torsor locally trivial on A} we can find an etale covering $S^{\prime}$ of $S$ such that the corresponding morphism  $$f_{\sigma}: (U)_{\dR} \times S^{\prime} \rightarrow \BG$$ factors through $\BG^{\naive}$. We then get an $S^{\prime}$-point of $\LocSys_G(D^{\circ}_{\Ranp})^{\et}$ by taking 
	\[
	\begin{tikzcd}
	D_{t^{\prime}, \nabla} \times_{X_{\dR}} U_{\dR}  \ar[r, rightarrow, ""'] &   (U)_{\dR} \times S^{\prime} \ar[r, rightarrow, "f_{\sigma}"']& \BG^{\naive},
	\end{tikzcd}
	\]
	where $t^{\prime}: S^{\prime}\rightarrow S \rightarrow (\Randrpt)^{\et}$ 
	
	This upgrades to a map from the Chech nerve of $S^{\prime} \rightarrow S$ to $\LocSys_G(D^{\circ}_{\Ranp})^{\et}$. Hence we obtain an $S$-point of $\LocSys_G(D^{\circ}_{\Ranp})^{\et}$, and the fact that the composition $$S \rightarrow \LocSys_G(D^{\circ}_{\Ranp})^{\et} \rightarrow \LocSys_G(D^{\circ})^{\et}$$ agrees with the composition
	$$S \rightarrow \LocSys_G^{\restr}(D^{\circ}) \rightarrow \LocSys_G(D^{\circ})^{\et}$$ is immediate. 
\end{cnstr}

\begin{pr}\label{affine diagonal}
	The morphism 
	\begin{equation}\label{main equation}
	\LocSys_G^{\restr}(D^{\circ}) \times_{\LocSys_G(D^{\circ})^{\et}} \LocSys_G(D^{\circ}_{\Ranp})^{\et} \rightarrow \LocSys_G^{\restr}(D^{\circ})  \times (\Randrpt)^{\et}
	\end{equation} has affine diagonal. 
\end{pr}

\begin{cor}\label{G affine}
	The morphism $\Psi_s$ is affine.
\end{cor}
The rest of the subsection is dedicated to the proof of Proposition \ref{affine diagonal}. Proposition \ref{restriction} will follow immediately, as will be explained at the end of the subsection. Consider the following general setup:
\begin{lm}\label{XYZ}
	Let $\cX$, $\cY$, $\cZ$ be prestacks over a base prestack $B$ with $\cX \rightarrow \cY \leftarrow \cZ$. Let $G_X$, $G_Y$, $G_Z$ be group prestacks over $B$ with $G_X \rightarrow G_Y \leftarrow G_Z$ acting compatibly on $\cX$, $\cY$, $\cZ$ respectively. Assume that $$(G_Z)^{\et} \rightarrow (G_Y / G_X)^{\et}$$ is surjective in etale topology. 
	Then 
	$$(\cX / G_X)^{\et} \times_{(\cY / G_Y)^{\et}} (\cZ / G_Z)^{\et} \cong ((\cX \times_{\cY} \cZ) / (G_X \times_{G_Y} G_Z))^{\et} $$ as prestacks over $B^{\et}$. 
\end{lm}

\begin{proof}
	Since sheafification commutes with finite limits \cite[Proposition 6.2.2.7]{HTT} $$(\cX \times_{\cY} \cZ )^{\et} \cong \cX^{\et} \times_{\cY^{\et}} \cZ^{\et}$$
	and there is a tautological map to $(\cX / G_X)^{\et} \times_{(\cY / G_Y)^{\et}} (\cZ / G_Z)^{\et}$ such that 
	$$(\cX^{\et} \times_{\cY^{\et}} \cZ^{\et}) \times_{((\cX / G_X)^{\et} \times_{(\cY / G_Y)^{\et}} (\cZ / G_Z)^{\et})} (\cX^{\et} \times_{\cY^{\et}} \cZ^{\et})  \cong (\cX^{\et} \times_{\cY^{\et}} \cZ^{\et}) \times (G_X \times_{G_Y} G_Z)^{\et}.$$
	Therefore it suffices to show that $$\cX^{\et} \times_{\cY^{\et}} \cZ^{\et} \rightarrow (\cX / G_X)^{\et} \times_{(\cY / G_Y)^{\et}} (\cZ / G_Z)^{\et} $$ is surjective in etale topology. Equivalently, we need to show that 
	$$(\cX \times_{\cY} \cZ)^{\et} \rightarrow (\cX / G_X \times_{\cY / G_Y} \cZ / G_Z)^{\et}$$ is surjective in etale topology. Let $T$ be a test affine scheme mapping to the target. Passing to a covering $T^{\prime}$ we obtain a $T^{\prime}$-point of $\cX / G_X \times_{\cY / G_Y} \cZ / G_Z$, and since $(G_Z)^{\et} \rightarrow (G_Y / G_X)^{\et}$ is surjective, passing to a cover $T^{\prime \prime}$ we can lift the $T^{\prime \prime}$-point of $(\cX / G_X \times_{\cY / G_Y} \cZ / G_Z)^{\et}$ to $(\cX \times_{\cY} \cZ)^{\et}$.

\end{proof}

\begin{pr}\label{forms mod gauge}
	Recall from Definition \ref{forms factorizably} that $\Omega^1(D^{\circ}_{\Ranp})$ is the prestack degenerating $\Omega^1(D)$ to $\Omega^1(D^{\circ})$ and  $G(D^{\circ}_{\Ranp})$ is the prestack degenerating $G(D)$ to $G(D^{\circ})$. Then 
	\begin{equation}
	(\Omega^1(D^{\circ}) / G(D))^{\et} \times_{\LocSys_G(D^{\circ})^{\et}} \LocSys_G(D^{\circ}_{\Ranp})^{\et} \cong (\Omega^1(D^{\circ}_{\Ranp}) / G(D)_{\Randrpt})^{\et}.
	\end{equation}
\end{pr}

\begin{proof}
	First, we will prove that 
	\begin{equation}\label{intermediate}
	(\Omega^1(D^{\circ}) / G(D))^{\et} \times_{\LocSys_G(D^{\circ})^{\et}} \LocSys_G(D^{\circ}_{\Ranp})^{\et} \cong (\Omega^1(D^{\circ}_{\Ranp}) / G(D) \times_{G(D^{\circ})} G(D^{\circ}_{\Ranp}))^{\et}.
	\end{equation}
	By Lemma \ref{XYZ} it is enough to show that the natural map $$(G(D^{\circ}_{\Ranp})^{\et} \rightarrow (G(D^{\circ}) \times \Randrpt) / (G(D) \times \Randrpt))^{\et}$$ is surjective. 
	
	Consider a prestack $G[t, t^{-1}]$ defined by $$G[t, t^{-1}](\Spec(A)) = \Maps(\Spec(A[t, t^{-1}]), G).$$ Note that $G[t, t^{-1}]\times \Randrpt $  maps to $G(D^{\circ}) \times \Randrpt $ and to $G(D^{\circ}_{\Ranp})$ such that the diagram 
	\[
	\begin{tikzcd}
	&   	G(D^{\circ}_{\Ranp})\ar[d, rightarrow, ""']  \\
	G[t, t^{-1}]\times \Randrpt \ar[ur, rightarrow, ""']\ar[r, rightarrow, ""']&   G(D^{\circ}) \times \Randrpt 
	\end{tikzcd}
	\]
	is commutative. So it suffices to check that the map 
	$$G[t, t^{-1}]\rightarrow (G(D^{\circ})/ G(D) )^{\et}$$ is surjective in etale topology. We claim that $$(G[t, t^{-1}] / G[t])^{\et} \cong (G(D^{\circ})/ G(D) )^{\et}, $$
	where $G[t](\Spec(A)) =\Maps(\Spec(A[t]), G).$  Indeed, by the Beauville-Laszlo theorem $(G(D^{\circ})/ G(D) )^{\et}$ is equivalent to the stack parameterizing $G$-torsors on $\bA^1$ with a trivialization on $\bG_m$. But all such torsors extend to $\bP^1$, and hence by \cite[Appendix]{DS} are trivializable. Therefore we obtain the desired isomorphism by choosing a trivialization.
	
	Next, we claim that $$G(D) \times_{G(D^{\circ})} G(D^{\circ}_{\Ranp}) \cong G(D)_{\Randrpt}.$$
	This follows by definition from the fact that for $f: \Spec(A) \rightarrow \Randrpt$ the diagram 
	\[
	\begin{tikzcd}
	\Spec(A((t)))\ar[r, rightarrow, ""']\ar[d, rightarrow, ""']&   	\Spec(A[t]^{\wedge}_{(I_{\Gamma_{f^{\redu}}})}[t^{-1}])\ar[d, rightarrow, ""']  \\
	\Spec(A[[t]]) \ar[r, rightarrow, ""']&   \Spec(A[t]^{\wedge}_{(I_{\Gamma_{f^{\redu}}})}) 
	\end{tikzcd}
	\]
	is a pushout square. To see this (or, equivalently, that the corresponding diagram of rings is a pullback square), it suffices to show that 
	\begin{equation}\label{R[t^{-1}]/R equation}
	A((t)) / A[[t]] \cong A[t]^{\wedge}_{(I_{\Gamma_{f^{\redu}}})}[t^{-1}] / A[t]^{\wedge}_{(I_{\Gamma_{f^{\redu}}})}.
	\end{equation}
	Note that $A[[t]] / (t) \cong A$, and since $ (I_{\Gamma_{f^{\redu}}}) \subset (t)$ we also have $A[t]^{\wedge}_{(I_{\Gamma_{f^{\redu}}})} / (t) \cong A$. Hence  $$A[[t]] / (t^n) \cong  A[t]^{\wedge}_{(I_{\Gamma_{f^{\redu}}})} / (t^n).$$
	Then writing left-hand side of (\ref{R[t^{-1}]/R equation}) as 
	\[
	\begin{tikzcd}
	\colim(0 \ar[r, rightarrow, ""'] &   A[[t]]/(t)\ar[r, rightarrow, "t"']& A[[t]]/(t^2) \ar[r, rightarrow, "t"'] & ...),
	\end{tikzcd}
	\]
	and right-hand side of (\ref{R[t^{-1}]/R equation}) as 
	\[
	\begin{tikzcd}
	\colim(0 \ar[r, rightarrow, ""'] &  A[t]^{\wedge}_{(I_{\Gamma_{f^{\redu}}})}/(t)\ar[r, rightarrow, "t"']& A[t]^{\wedge}_{(I_{\Gamma_{f^{\redu}}})}/(t^2) \ar[r, rightarrow, "t"'] & ...)
	\end{tikzcd}
	\]
	we get the result. So we can rewrite the right-hand side of (\ref{intermediate}) as
	$$(\Omega^1(D^{\circ}_{\Ranp}) /  G(D)_{\Randrpt})^{\et}.$$
\end{proof}

\begin{lm}\label{forms mod gauge affine diagonal}
	The morphism $$\Omega^1(D^{\circ}_{\Ranp}) / G(D)_{\Randrpt}\rightarrow \Randrpt$$ has an affine diagonal.
\end{lm}

\begin{proof}
	Note that since $G(D)_{\Randrpt}$ is affine over $\Randrpt$ it suffices to check that the map  $$\Omega^1(D^{\circ}_{\Ranp})\rightarrow \Randrpt$$ has an affine diagonal.
	Indeed, consider a prestack $\Omega^{1, \leq n}(D^{\circ}_{\Ranp})$ defined as follws.
	For any test affine scheme $S$ the category $\Omega^{1, \leq n}(D^{\circ}_{\Ranp})(S)$ is defined as:
	\begin{itemize}
		\item An object is a pair $(s, \omega)$ consisting of an object $s \in \Ranp(S)$ and $\omega \in \Gamma(D_{s}, \Omega^{1, \leq n}_U \boxtimes \cO_S)$;
		\item A morphism from $(s, \omega)$ to $(s^{\prime}, \omega^{\prime})$ consists of a morphism $s \rightarrow s^{\prime}$ in $\Ranp(S)$ and an isomorphism between $\omega$ and $\omega^{\prime}|_{D_{s}}$. 
	\end{itemize}
	Here $\Omega^{1, \leq n}_U $ stands for the sheaf of 1-forms with a pole of order $\leq n$. 
	Since $j_*(\Omega^1_U) \cong \colim_n \Omega^{1, \leq n}_U$ we have 
	$$\Gamma(D_{s}\otimes_X U, \Omega^{1}_U \boxtimes \cO_S) \cong \colim_n \Gamma(D_{s}, \Omega^{1, \leq n}_U \boxtimes \cO_S).$$
	Therefore we have that for any affine $S$ a map $S \rightarrow \Omega^1(D^{\circ}_{\Ranp})$ factors through $\Omega^{1, \leq n}(D^{\circ}_{\Ranp})$  for some $n$. Thus it suffices to check that $$\Omega^{1, \leq n}(D^{\circ}_{\Ranp}))\rightarrow \Randrpt$$ has an affine diagonal, which is immediate. 
\end{proof}

\begin{lm}\label{etale covering}
	Let $f:\cX \rightarrow \cY$ be a morphism of prestacks which are sheaved in etale topology. Suppose that there exists a moprhism $\cY^{\prime} \rightarrow \cY$ which is affine and an etale covering. Then if the moprhism $$f^{\prime}: \cX^{\prime} = \cX \times_{\cY} \cY^{\prime} \rightarrow \cY^{\prime}$$ has affine diagonal then so does $f$. 
\end{lm}
\begin{proof}
	Note that $$\cX^{\prime} \times_{\cY^{\prime}} \cX^{\prime} \rightarrow \cX \times_{\cY} \cX$$ is an affine etale covering, and pullback of $\Delta_f$ along this map gives affine morphism $\Delta_{f^{\prime}}$. 
	For an affine scheme $T$ mapping to $\cX \times_{\cY} \cX$ consider the pullback along this map of the Cartesian diagram 
	\[
	\begin{tikzcd}
	\cX^{\prime}\ar[r, rightarrow, ""']\ar[d, rightarrow, ""']&   	\cX^{\prime} \times_{\cY^{\prime}} \cX^{\prime}\ar[d, rightarrow, ""']  \\
	\cX \ar[r, rightarrow, ""']&   \cX \times_{\cY} \cX.
	\end{tikzcd}
	\]
	We are in position to apply \cite[Chapter 2 Proposition 2.4.2]{GR1}, which gives the result. 
\end{proof}

\begin{lm}\label{affine diagonal after pullback to affines}
	Let $f: \cX \rightarrow \cY$ be a morphism of prestacks. Assume that for every affine $S$ mapping to $\cY$ the morphism $f_S: \cX \times_{\cY} S \rightarrow S$ has an affine diagonal. Then $f$ has an affine diagonal.
\end{lm}

\begin{proof}[Proof of Proposition \ref{affine diagonal}]
	Recall that $\LocSys_G^{\restr}(D^{\circ})$ is isomorphic to an etale sheafified quotient by $G$ of a certain etale stack $\cY$. By Lemma \ref{etale covering} then it suffices to check that 
	\begin{equation}\label{cY morphism}
	\cY \times_{\LocSys_G(D^{\circ})^{\et}} \LocSys_G(D^{\circ}_{\Ranp})^{\et} \rightarrow \cY \times (\Randrpt)^{\et}
	\end{equation}
	has an affine diagonal. 
	
	By Lemma \ref{affine diagonal after pullback to affines}, to see this  it suffices to show that for every affine $S$ mapping to $\cY \times (\Randrpt)^{\et}$ pullback of  (\ref{cY morphism}) to $S$ has an affine diagonal. 
	
	But recall that $\cY$ can be written as a filtered colimit $\colim_i Y_i$, where $Y_i$ are affine schemes. Therefore a map from an affine scheme $S$ to $\cY \times (\Randrpt)^{\et}$ factors through $Y_i  \times (\Randrpt)^{\et}$ for some $i$. So it suffices to show that 
	\begin{equation}
	Y_i \times_{\LocSys_G(D^{\circ})^{\et}} \LocSys_G(D^{\circ}_{\Ranp})^{\et} \rightarrow Y_i \times (\Randrpt)^{\et}
	\end{equation}
	has an affine diagonal. Consider an etale covering $Y_i^{\prime}$ of $Y_i$ such that a map $Y_i^{\prime} \rightarrow \LocSys_G(D^{\circ})^{\et}$ lifts to $(\Omega^1(D^{\circ}) / G(D))^{\et}$. 
	By Lemma \ref{etale covering} it is enough to show that 
	\begin{equation}
	Y_i^{\prime}\times_{\LocSys_G(D^{\circ})^{\et}} \LocSys_G(D^{\circ}_{\Ranp})^{\et} \rightarrow Y_i^{\prime}\times (\Randrpt)^{\et}
	\end{equation}
	has an affine diagonal. But by Proposition \ref{forms mod gauge} 
	$$Y_i^{\prime}\times_{\LocSys_G(D^{\circ})^{\et}} \LocSys_G(D^{\circ}_{\Ranp})^{\et} \cong Y_i^{\prime} \times_{(\Omega^1(D^{\circ}) / G(D))^{\et}} (\Omega^1(D^{\circ}_{\Ranp}) / G(D)_{\Randrpt})^{\et},$$
	and since $(\Omega^1(D^{\circ}) / G(D))^{\et}$ has an afffine diagonal and $(\Omega^1(D^{\circ}_{\Ranp}) / G(D)_{\Randrpt})^{\et} \rightarrow \Randrpt$ has an affine diagonal (Lemma \ref{forms mod gauge affine diagonal}) the result follows.
\end{proof}

\begin{cnstr}
	
	Note that we have an adjunction in $\mathbf{FactAlgCat}^T$:
	\begin{equation} \label{eqn-Phi-adj}
	\Phi^{\univ} : \Rep(G)^T \rightleftarrows \Vect^T: \Phi^{\univ, R}.
	\end{equation}

	\medskip
By Corollary \ref{G affine} we have an adjunction in $\mathbf{FactModCat}^T$:
	\[
	(\Phi^{\univ} , \Psi_s^*): (\Rep(G)^T, \bFact^{\restr}(R_1) \otimes \QCoh(R_2)) \rightleftarrows (\Vect^T, \Vect^T): (\Phi^{\univ, R} , {\Psi_s}_*).
	\]
\end{cnstr}

\begin{proof}[Proof of Proposition \ref{restriction}]
	It is clear that $\Psi_s^*$ and ${\Psi_s}_*$ induce the equivalences between the cores. Then the claim follows from  Corollary \ref{B.8.4}.
\end{proof}

\section{Proof of the main theorem}\label{section: main proof}
Consider the following diagram:

	\[
\begin{tikzcd}
\QCoh(T)\ar[r, "\sim", "\Id"']&   \QCoh(T)    \\
\QCoh(R_1 \times_{\LocSys_G^{\restr}(D^{\circ}) }R_2) \arrow[u, " "]  \ar[r, "", ""']  &  {F_{\sigma_2}}^*{F_{\sigma_1}}_* (\omega_{U} \boxtimes \cO_{T})\FactMod^T  ,\arrow[u, ""]
\end{tikzcd}
\]
where the left arrow is $*$-pushforward along $p:R_1 \times_{\LocSys_G^{\restr}(D^{\circ}) }R_2 \rightarrow T$  and the right arrow is taking the fiber at $x$.  Commutativity follows from the fact that $(-)_{\Rep(G)}$ and all reduction steps preserved forgetful functors.

Passing to left adjoints, we obtain a diagram that commutes up to a natural transformation 

	\[
\begin{tikzcd}
\QCoh(T) \arrow[d, " "]\ar[r, "\sim", "\Id"']&   \QCoh(T) \arrow[d, " "]\arrow[Rightarrow, dl, " "]    \\
\QCoh(R_1 \times_{\LocSys_G^{\restr}(D^{\circ}) }R_2)  \ar[r, "", ""']  &  {F_{\sigma_2}}^*{F_{\sigma_1}}_* (\omega_{U} \boxtimes \cO_{T})\FactMod^T.
\end{tikzcd}
\]

Let $M_1$ and $M_2$ denote the monads acting on $\QCoh(T)$ coming from left- and right-hand sides of the diagrams above. We have a map $$M_2 \rightarrow M_1.$$
We will show that this is an equivalence, which will imply the equation (\ref{main assertion after reductions}).

\subsection{Description of $M_1$}
\begin{lm}
	The diagonal $\Delta: \LocSys_G^{\restr}(D^{\circ})  \rightarrow \LocSys_G^{\restr}(D^{\circ})  \times \LocSys_G^{\restr}(D^{\circ}) $ is affine schematic.
\end{lm}

\begin{proof}
	We will prove that for any prestack $\cY$ with an affine diagonal and an affine group $G$ acting on it, the prestack $\cY / G$ has an affine diagonal. The statement of the lemma will follow then from Proposition \ref{properties of LS}. 
	
	Indeed, consider a cover of $\cY / G \times \cY / G$ given by $\cY \times \cY$. By \cite[Chapter 2 Proposition 2.4.1(b)]{GR1}it suffices to check that the pullback of $\Delta$ to this cover is an affine morphism. We have 
	$$(\cY \times \cY) \times_{\cY / G \times \cY / G} \cY / G \cong G \times \cY,$$
	and the pullback of the diagonal map of $\cY / G$ is then the composition 
	 					\[
	\begin{tikzcd}
	G \times \cY \ar[r, rightarrow,  "\Delta"]
	& (G \times \cY) \times (G \times \cY)\arrow[r,  " \proje \times \action"]
	&  \cY \times \cY,
	\end{tikzcd}
	\] 
	where both maps are affine. 
\end{proof}

\begin{rem}
	In particular this shows that $p$ is affine schematic and $p_*$ is continuous. 
\end{rem}
Hence by \cite[Chapter 3 Proposition 3.3.3]{GR1} we get that $$M_1 \cong p_*(\cO_{R_1 \times_{\LocSys_G^{\restr}(D^{\circ}) }R_2}) \otimes -.$$

\subsection{Description of $M_2$}

Note that ${F_{\sigma_2}}^*{F_{\sigma_1}}_* (\omega_{U} \boxtimes \cO_{T})$ lies in $$\QLisse(U) \otimes \QCoh(T) \subset \D_{\indhol}(U) \otimes \QCoh(T)$$ since $\sigma_1$ and $\sigma_2$ are restricted families of local systems. Therefore we can apply Theorem \ref{factmod commutative} and Remark \ref{factmod commutative U}. 

\begin{lm}
	The commutative algebra ${F_{\sigma_2}}^*{F_{\sigma_1}}_* (\omega_{U} \boxtimes \cO_{T})$ comes from a prestack $$({U}_{\dR} \times T) \times_{\BG \times {U}_{\dR} \times T} ({U}_{\dR} \times T)$$ affine over ${U}_{\dR} \times T$, where the maps in the fiber product are given by $F_{\sigma_1}$ and $F_{\sigma_2}$.
\end{lm}

\begin{proof}
	Since ${F_{\sigma_1}}$ is affine by base change $${F_{\sigma_2}}^*{F_{\sigma_1}}_* (\omega_{U} \boxtimes \cO_{T}) \cong \pi_* \cO_{({U}_{\dR} \times T) \times_{\BG \times {U}_{\dR} \times T} ({U}_{\dR} \times T)}$$ for $\pi: ({U}_{\dR} \times T) \times_{\BG \times {U}_{\dR} \times T} ({U}_{\dR} \times T) \rightarrow {U}_{\dR} \times T.$
	
\end{proof}
So it is left to compute the fiber at $x$ of  $\hormerxJetsF_{\nabla, T \times {U}_{\dR}}^{\mer, x}(({U}_{\dR} \times T) \times_{\BG \times {U}_{\dR} \times T} ({U}_{\dR} \times T))$.

\begin{lm}
	The fiber of $$\hormerxJetsF_{\nabla, T \times {U}_{\dR}}^{\mer, x}(({U}_{\dR} \times T) \times_{\BG \times {U}_{\dR} \times T} ({U}_{\dR} \times T))$$ at $x$ is isomorphic to $R_1 \times_{\LocSys_G^{\restr}(D^{\circ}) }R_2$. 
\end{lm}

\begin{proof}
	It is immediate that the fiber in question is equivalent to $$T \times_{T \times (\hormerxJets(\BG))_{x}} T \cong R_1 \times_{ (\hormerxJets(\BG))_{x}} R_2.$$
	By construction both maps factor through $\LocSys_G^{\restr}(D^{\circ})$, so the claim follows from Proposition \ref{monomorphism}.
\end{proof}

\subsubsection{} Therefore we get that $M_2 \cong p_*(\cO_{R_1 \times_{\LocSys_G^{\restr}(D^{\circ}) }R_2}) \otimes -$ as well, which concludes the proof of Theorem \ref{mainthm}.

\appendix
\section{Multidiscs and jets}\label{appendix: multidiscs}

\subsection{Discs}
For an affine scheme $S$ fix a map $x^I: S \rightarrow X_{\dR}^I$. 
\begin{df}
	Define the formal disc $\widehat{D}_{x^I}$ as $(S \times X)^{\wedge}_{\Gamma_{f^{\redu}}},$
	where $\Gamma_{f^{\redu}}$ is the pullback of the universal divisor along 
	$$S^{\redu} \rightarrow X^I \rightarrow \Div^{|I|}(X).$$
\end{df}

This prestack admits a canonical descent along $X \rightarrow X_{\dR}$ given by ${{\widehat{D}}_{x^I, \nabla}} := S ^{\wedge}_{\Gamma_{f^{\redu}}}$.

\begin{df}
	Define the adic disc $D_{x^I}$ as the affinization of $\widehat{D}_{x^I}$.
\end{df}
\begin{rem}
	Note that affinization here is defined since $\widehat{D}_{x^I}$ is ind-affine.
\end{rem}

This prestack also admits a canonical descent along $X \rightarrow X_{\dR}$, denoted by $D_{x^I, \nabla}$. We will explain the name and give an explicit construction below. Let us note that the various formal discs can be assembled in the following diagram: 

\[
\begin{tikzcd}
\widehat{D}_{x^I} \ar[d, rightarrow, ""']\ar[r, rightarrow, ""'] &   D_{x^I}\ar[d, rightarrow, ""'] \ar[d, rightarrow, ""']\ar[r, rightarrow, ""']  & X \ar[d, rightarrow, ""']\\
{\widehat{D}}_{x^I, \nabla} \ar[r, rightarrow, ""']& D_{x^I, \nabla} \ar[r, rightarrow, ""']& X_{\dR} .
\end{tikzcd}
\]

Now we need to define  $D_{x^I, \nabla}$ for ${x^I}: S \rightarrow X_{\dR}^I$. Notice that naive definition $$D_{x^I, \nabla}:= (D_{x^I})_{\dR}$$ will fail to produce the desired answer since even for $x: \pt \rightarrow X_{\dR}$ the diagram 

\[
\begin{tikzcd}
D_x \ar[d, rightarrow, ""']\ar[r, rightarrow, ""'] &   X \ar[d, rightarrow, ""']  \\
(D_x)_{\dR} \ar[r, rightarrow, ""']&  X_{\dR}
\end{tikzcd}
\]
will not be Cartesian. Instead, we will produce the definition via writing down an action of the infinitesimal groupoid on $D_{x^I}$. 

\begin{cnstr}\label{fake}
	For ${x^I}: S \rightarrow X_{\dR}^I$ define $D_{x^I, \nabla}$ as follows. Consider the simplicial prestack $$\widehat{D}_{x^I}^{\bullet} := {\widehat{D}}_{x^I, \nabla} \times_{X_{\dR}} \hat{\Delta}_{\bullet} = S ^{\wedge}_{\Gamma_{f^{\redu}}}\times_{X_{\dR}} \hat{\Delta}_{\bullet},$$
	where $\hat{\Delta}_{\bullet}$ is the Cech nerve of $X \rightarrow X_{\dR}$. For each $n \in \bN$, restrict $\hat{D}_{x^I}^{\bullet}$ to the simplicial scheme $\Delta^{(n)}_{\bullet}$, where $\Delta^{(n)}_{k}$ is the $n$-th neighborhood of the main diagonal in $X^{k+1}$, then take the affinization term-wise. Denote the result by ${D}_x^{(n), \bullet}$.
	
	Put $D_{x^I, \nabla}= \colim_n |{D}_{x^I}^{(n), \bullet}|.$
	
\end{cnstr}

\begin{rem}
	Since $S ^{\wedge}_{\Gamma_{f^{\redu}}}$ is naturally a prestack over $X_{\dR}$ we get that $D_{x^I, \nabla}$ also lives over $X_{\dR}$.
\end{rem}




\begin{lm}
	For any ${x^I}: S \rightarrow X_{\dR}^I$ we have $$D_{x^I, \nabla}\times_{X_{\dR}} X \cong D_{x^I}.$$
\end{lm}

\begin{proof}
	
	Consider the following framework. Let $S_{\bullet}$ be a simplicial prestack mapping to $\hat{\Delta}_{\bullet}$. Call $S_{\bullet}$ Cartesian over $\hat{\Delta}_{\bullet}$ if $$S_n \times_{\hat{\Delta}_n } \hat{\Delta}_{n+1} \cong S_{n+1}$$
	for any choice of  projection $\hat{\Delta}_{n+1} \rightarrow \hat{\Delta}_n$. We claim that in this case $|S_{\bullet}| \times_{X_{\dR}} X \cong S_0$ as prestacks over $X$. Indeed, if we view $X$ as a constant simplicial prestack over $X_{\dR}$ we get 
	$$|S_{\bullet}| \times_{X_{\dR}} X  \cong |S_{\bullet} \times_{X_{\dR}} X |.$$
	
	We claim that
	\begin{equation}\label{keyiso}
	S_{\bullet} \times_{X_{\dR}} X  \cong S_0 \times_{X_{\dR}} \hat{\Delta}_{\bullet}
	\end{equation}
	as simplicial prestacks over $\hat{\Delta}_{\bullet} \times_{X_{\dR}} X$. Here $S_{\bullet} \times_{X_{\dR}} X$ maps to $\hat{\Delta}_{\bullet}$ via projection to $S_{\bullet}$ and to $X$ via second projection, and $S_0 \times_{X_{\dR}} \hat{\Delta}_{\bullet}$ maps to $\hat{\Delta}_{\bullet}$ via second projection and to $X$ via projection to $S_0$. 
	
	Let us introduce some notation. Let $p_{\hat{k}}$ be maps in the simplicial complex $\hat{\Delta}_{\bullet}$, let $d_{\hat{k}}$ be maps in the simplicial complex $S_{\bullet}$. Finally, let $\pi_{\bullet}$ denote the map $S_{\bullet} \rightarrow \hat{\Delta}_{\bullet}$. 
	
	We proceed by constructing the isomorphism in three steps. First, Consider $S_n \times_{\hat{\Delta}_n } \hat{\Delta}_{n+1}$, where  $\hat{\Delta}_{n+1}$ maps to $\hat{\Delta}_n $ via  $p_{\hat{1}}$. This is a prestack over $\hat{\Delta}_{n} \times_{X_{\dR}} X$ (it maps to $\hat{\Delta}_{n}$ via $\pi_n$ precomposed with projection and to $X$ via composition of projection to $\hat{\Delta}_{n+1}$ and projection of the latter to the first coordinate). Identity on $S_n$ and projection to the first coordinate on $\hat{\Delta}_{n+1}$ define an isomorphism
	$$S_n \times_{\hat{\Delta}_n } \hat{\Delta}_{n+1} \xrightarrow[\sim]{\text{}} S_n \times_{X_{\dR}} X.$$
	Now using that $S_{\bullet}$ is Cartesian over $\hat{\Delta}_{\bullet}$ we get an isomorphism over $\hat{\Delta}_{n} \times_{X_{\dR}} X$ 
	$$(d_{\hat{1}}, \pi_{n+1}): S_{n+1} \xrightarrow[\sim]{\text{}} S_n \times_{\hat{\Delta}_n } \hat{\Delta}_{n+1}.$$
	Here $S_{n+1}$ si viewed as a prestack over $\hat{\Delta}_{n} \times_{X_{\dR}} X$ via $(p_{\hat{1}} \circ \pi_{n+1}, \pi_0 \circ d_1 )$, where $d_1: S_{n+1} \rightarrow S_0$ is the obvious map.
	Using again that $S_{\bullet}$ is Cartesian over $\hat{\Delta}_{\bullet}$ we get an isomorphism over $\hat{\Delta}_{n} \times_{X_{\dR}} X$ 
	$$(d_1, p_{\hat{1}} \circ \pi_{n+1}): S_{n+1}  \xrightarrow[\sim]{\text{}} S_0 \times_{X_{\dR}} \hat{\Delta}_n.$$
	
	Composition of the above three isomorphisms defines the desired (\ref{keyiso}), since it is  compatible with the simplicial structure.

	It is left to notice that $(\colim_n {D}_{x^I}^{(n)})^{\bullet}$ is Cartesian over $\hat{\Delta}_{\bullet}$ by construction.

\end{proof}

\begin{cor}\label{fakedR_affine}
	The prestack $D_{x^I, \nabla}$ is affine over $X_{\dR}$ for any ${x^I}: S \rightarrow X_{\dR}^I$ with affine $S$. 
\end{cor}

\begin{proof}
	Since $X$ is smooth every map $T \rightarrow X_{\dR}$ for affine $T$ factors through $T \rightarrow X$, we have 
	$$T \times_{X_{\dR}}D_{x^I, \nabla} \cong T \times_X X \times_{X_{\dR}} D_{x^I, \nabla} \cong T \times_X D_x.$$
	The latter is fiber product of affine schemes so we are done. 
\end{proof}

Let us justify the name of the prestack $D_{x^I, \nabla}$. For simplicity we will work with $D_{x_A, \nabla}$ for $$x_A: \Spec(A) \rightarrow x \rightarrow X_{\dR}.$$
\begin{pr}\label{de rham of disc modules with connection}
	For classical $A$ the category $\QCoh(D_{x_A, \nabla})$ is equivalent to the category of $A[[t]]$-modules with connection. Moreover, for an object $(M, \nabla)$ in this category we have $$\Gamma(D_{x_A, \nabla}, M) \cong \Gamma_{\dR}(\Spec(A[[t]]), M).$$
\end{pr}

\begin{proof}
	Proof of the first claim follows the standard argument for de Rham prestack of a scheme. Let us show the second assertion. We may for simplicity assume that $X = \Spec k[t].$ Consider the diagram 
	\begin{equation}\label{de rham cohomology discs}
		\begin{tikzcd}
		 &   \lim_n M \otimes_{A[[t]]]} \Omega^2_{D_{x}^{(n), 1}/A}  \arrow[r, shift left=1.3ex] \arrow[r, shift right=1.3ex] \arrow[r] & \lim_n M \otimes_{A[[t]]]} \Omega^2_{D_{x}^{(n), 2}/A}&\ldots \\
		M dt\arrow[r, shift left=0.65ex] \arrow[r, shift right=0.65ex] & \lim_n M \otimes_{A[[t]]]} \Omega^1_{D_{x}^{(n), 1}/A}\ar[u, rightarrow, "\nabla"'] \arrow[r, shift left=1.3ex] \arrow[r, shift right=1.3ex] \arrow[r] & \lim_n M \otimes_{A[[t]]]} \Omega^1_{D_{x}^{(n), 2}/A} \ar[u, rightarrow, "\nabla"']&\ldots \\ 
		M \arrow[r, shift left=0.65ex] \arrow[r, shift right=0.65ex] \ar[u, rightarrow, "\nabla"']& \lim_n M \otimes_{A[[t]]]} \cO_{D_{x}^{(n), 1}}\ar[u, rightarrow, "\nabla"']\arrow[r, shift left=1.3ex] \arrow[r, shift right=1.3ex] \arrow[r] & \lim_n M \otimes_{A[[t]]]} \cO_{D_{x}^{(n), 2}}\ar[u, rightarrow, "\nabla"']&\ldots .
		\end{tikzcd}
	\end{equation}
	The desired assertion follows from the fact that each column in (\ref{de rham cohomology discs}) is quasi-isomorphic to $$M \xrightarrow{\nabla} M dt,$$
	and each of the two top rows are acyclic, so let us check these claims. Indeed, the first claim follows from the naive Poincare lemma. As for the second claim, notice that these complexes can be obtain via base change of cosimplicial modules
	\begin{equation}\label{cosimplicial module}
		\begin{tikzcd}
	\Omega^i_{A[t]/A} \arrow[r, shift left=0.65ex] \arrow[r, shift right=0.65ex]& \Omega^i_{A[t_1, t_2]/A} \arrow[r, shift left=1.3ex] \arrow[r, shift right=1.3ex] \arrow[r] & \ldots
			\end{tikzcd}
	\end{equation}
	along the cosimplicial ring map 
		\begin{equation}
	\begin{tikzcd}
	A[t] \ar[d, rightarrow, ""']\arrow[r, shift left=0.65ex] \arrow[r, shift right=0.65ex]& A[t_1, t_2] \ar[d, rightarrow, ""']\arrow[r, shift left=1.3ex] \arrow[r, shift right=1.3ex] \arrow[r] & \ldots\\
		A[[t]] \arrow[r, shift left=0.65ex] \arrow[r, shift right=0.65ex]& A[[t_1]][t_2]^{\wedge}_{(t_1 - t_2)}\arrow[r, shift left=1.3ex] \arrow[r, shift right=1.3ex] \arrow[r] & \ldots.
	\end{tikzcd}
	\end{equation}
	Finally, the complex (\ref{cosimplicial module}) is acyclic by \cite[Lemma 2.15]{BdJ}.
	
\end{proof}

\subsection{Horizontal jets meromorphic at one point}

\begin{df}
Let $\cY$ be a prestack over $X_{\dR}$. For an affine $S$ with a map $$(x \subseteq x^{\prime}) : S \rightarrow \Arr(\Randr)$$ let $\Jets_{\nabla}^{\mer \leadsto \reg}(\cY)_{\Arr(\Randr)}(S)$ be the space of maps 
$$(D_{x^{\prime}, \nabla} - x) \rightarrow \cY$$
over $X_{\dR}$.
\end{df}

\begin{ntn}
		Let $$\hormerxJets(\cY)_{\Randrpt} \cong \Jets_{\nabla}^{\mer \leadsto \reg}(\cY)_{\Arr(\Randr)} \times_{\Arr(\Randr)}\Randrpt.$$
\end{ntn}

\begin{df}
	For an affine $S$ with a map $x: S \rightarrow X_{\dR}$ set $\hormerxJets(\cY)(S) $ to be maps $$D_{x, \nabla}\times_{X_{\dR}} U_{\dR} \rightarrow \cY$$ over $U_{\dR}$. in other words, 
	$$\hormerxJets(\cY) \cong \Jets_{\nabla}^{\mer \leadsto \reg}(\cY)_{\Randrpt} \times_{\Randrpt}(X_{\dR} \times \{x\}).$$
\end{df}
\begin{rem}
	Note that in this definition $\cY$ could be taken as a prestack over $U_{\dR}$ as well. 
\end{rem}

\begin{ex}\label{main example discs}
	For $\cY = \BG^{\naive} \times U_{\dR}$ we have $$\hormerxJets(\cY)(S)_{x} \cong \LocSys_G(D^{\circ}),$$
	$$\hormerxJets(\cY)_{\Randrpt} \cong {\LocSys_G(D^{\circ}_x)}_{\Randrpt}.$$
	Here $ \BG^{\naive}$ is the \emph{prestack} quotient of a point by the trivial action of $G$.
\end{ex}

Below assume that $\cY$ is affine over $U_{\dR}$.

\begin{lm}
	We have 
	$$\hormerxJets(\cY)  \times_{X_{\dR}} U_{\dR} \cong \cY.$$
	
\end{lm}

\begin{proof}
	We need to show that for any affine $S$ mapping to $U_{\dR}$
	$$\Maps_{U_{\dR}}(S, \cY) \cong \Maps_{X_{\dR}}(S, \hormerxJets(\cY)).$$
	
	Notice that $S \cong | \colim_n S \times_{U_{\dR}} \Delta^{(n)}_{U, \bullet}|$, where $\Delta^{(n)}_{U, k}$ is the $n$-th neighborhood of the diagonal in $U^{k+1}$. Hence 
	$$\Maps_{U_{\dR}}(S, \cY) \cong \lim_k \lim_n \Maps_{U_{\dR}}(S \times_{U_{\dR}}\Delta^{(n)}_{U, k}, \cY ) \cong \lim_k \lim_n \Maps_{\Delta^{(n)}_{U, k}}(S \times_{U_{\dR}}\Delta^{(n)}_{U, k}, \cY  \times_{U_{\dR}} \Delta^{(n)}_{U, k}).$$
	
	Since $\Delta^{(n)}_{U, k}$ and $\cY  \times_{U_{\dR}} \Delta^{(n)}_{U, k}$ are affine schemes the latter is isomorphic to 
	$$ \lim_k \lim_n \Maps_{\Delta^{(n)}_{U, k}}((S \times_{U_{\dR}}\Delta^{(n)}_{U, k})^{\aff}, \cY  \times_{U_{\dR}} \Delta^{(n)}_{U, k}) \cong \Maps_{U_{\dR}}(| \colim_n (S \times_{U_{\dR}}\Delta^{(n)}_{U, k})^{\aff}|, \cY).$$
	
	On the other hand, 
	$$\Maps_{X_{\dR}}(S, \hormerxJets(\cY)) = \Maps_{U_{\dR}} (| \colim_n (S \times_{X_{\dR}}\Delta^{(n)}_{\bullet})^{\aff}| \times_{X_{\dR}} U_{\dR}, \cY) \cong$$
	$$ \cong \Maps_{U_{\dR}}(| \colim_n (S \times_{U_{\dR}}\Delta^{(n)}_{U, k})^{\aff}|, \cY).$$
\end{proof}

\subsubsection{} Consider categories of prestacks affine over $U_{\dR}$ and $X_{\dR}$ and the functor $$R: \cT \rightarrow \cT \times_{X_{\dR}} U_{\dR}$$ between them. 

\begin{lm}\label{description of R^R}
	Assume that partially defined right adjoint $R^R$ to $R$ is defined on $\cY$. Then $$R^R(\cY) \cong \hormerxJets(\cY) .$$
\end{lm}

\begin{proof}
	For $S$ affine with a map $x: S \rightarrow X_{\dR}$ we have 
	$$\Maps_{X_{\dR}}(S, R^R(\cY)) \cong \Maps_{X_{\dR}} (| \colim_n S \times_{X_{\dR}} \Delta_{\bullet}^{(n)}|, R^R(\cY)) \cong $$
	$$\cong \lim_k \lim_n\Maps_{\Delta_{\bullet}^{(n)}}(S \times_{X_{\dR}} \Delta_{\bullet}^{(n)}, R^R(\cY) \times_{X_{\dR}} \Delta_{\bullet}^{(n)}).$$ 
	
	Since $\Delta_{\bullet}^{(n)}$ and $R^R(\cY) \times_{X_{\dR}} \Delta_{\bullet}^{(n)}$ are affine this is the same as 
	$$\lim_k \lim_n\Maps_{\Delta_{\bullet}^{(n)}}((S \times_{X_{\dR}} \Delta_{\bullet}^{(n)})^{\aff}, R^R(\cY) \times_{X_{\dR}} \Delta_{\bullet}^{(n)}) \cong \Maps_{X_{\dR}} (D_{x, \nabla}, R^R(\cY)).$$
	Now using Lemma \ref{fakedR_affine} we get that this is the same as
	$$\Maps_{U_{\dR}}(D_{x, \nabla} \times_{X_{\dR}} U_{\dR}, \cY).$$
\end{proof}

\begin{rem}
	The proof of Lemma \ref{description of R^R} implies that $D_{x, \nabla}$ can be thought of as the relative affinization of $S \rightarrow X_{\dR}$. This is specific to the first power of $X_{\dR}$.
\end{rem}





\subsection{Horizontal jets meromorphic at one point: parametrized version} Let $B$ be a prestack affine over $X_{\dR}$. Consider categories of prestacks affine over $B \times_{X_{\dR}} U_{\dR}$ and $B$ and the functor $$R_B: \cT \rightarrow \cT \times_{X_{\dR}} U_{\dR}$$ between them. Let $\cY$ be a prestack affine over $B \times_{X_{\dR}} U_{\dR}$. 

\begin{df}
	For affine $S$ with a map $x: S \rightarrow B$ set $\hormerxJetsF_{\nabla, B}^{\mer, x_0}(\cY)(S) $ to be maps $$D_{x, \nabla} \times_{X_{\dR}} U_{\dR} \rightarrow \cY$$ over $B$. 
\end{df}

\begin{lm}\label{description of R_B^R}
	Assume that partially defined right adjoint $R_B^R$ to $R_B$ is defined on $\cY$. Then for $x: S \rightarrow B$ with affine $S$ we have 
	$$\Maps_{B}(S, R_B^R(\cY)) \cong \Maps_{B \times_{X_{\dR}} U_{\dR}}(D_{x, \nabla}\times_{X_{\dR}} U_{\dR}, \cY).$$
	
	In other words, $R_B^R(\cY) \cong\hormerxJetsF_{\nabla, B}^{\mer, x_0}(\cY)(S)$. 
\end{lm}

\begin{proof}
	As in the proof of the Lemma \ref{description of R^R} we have 
	$$\Maps_{B}(S, R_B^R(\cY)) \cong  \Maps_{B}(D_{x, \nabla}, R_B^R(\cY)).$$
	Since $D_{x, \nabla}$ is affine over $X_{\dR}$ (Lemma \ref{fakedR_affine}) and $B$ is affine over $X_{\dR}$ we see that $D_{x, \nabla}$ is affine over  $B$. 
	Hence 
	$$ \Maps_{B}(D_{x, \nabla}, R_B^R(\cY)) \cong \Maps_{B \times_{X_{\dR}} U_{\dR}}(D_{x, \nabla}\times_{X_{\dR}} U_{\dR}, \cY).$$
\end{proof}

\section{Etale pullback for factorization patterns}\label{appendix: etale factorization}

Consider the following setup. Let $f: X \rightarrow Y$ be an etale map of curves. The goal of this subsection is to study the relationship between factorization algebra and module categories for $X$ and $Y$.

\subsection{Etale pullback of factorization algebra categories}

\begin{dfprop}
	There exists an open lax subprestack $$p: \stackrel{\circ}{{\Ranp}_{X}} \hookrightarrow \Ranp_{X},$$ such that  geometric points of its underlying prestack are finite sets of points of $X$ with distinct images under $f$. 
\end{dfprop}

\begin{proof}
	For an affine scheme $S$ define $\stackrel{\circ}{{\Ranp}_{X}}(S)$ to be the full subcategory of $\Ranp_{X}(S)$ such that its objects, i.e. a finite collections $(x_i)_{i \in I} $ of $x_i: S \rightarrow X$, have the following property. For any $i, j, \in I$ we have $(x_i, x_j) \in V(S) \subset X^2(S)$, where $$V := (X\times X) \setminus (X \times_{Y} X \setminus \Delta)$$
	is an open in $X^2$. 
\end{proof}

\begin{cnstr}\label{circ Ran}
	One can define a commutative algebra structure 
	$$\stackrel{\circ}{{\Ranp}_{X}} \in \CommAlg(\laxPreSt_{\Corr})$$
	such that $n$-ary multiplication is given by 
	$$(\stackrel{\circ}{{\Ranp}_{X}} \times \stackrel{\circ}{{\Ranp}_{X}} ) \leftarrow (\stackrel{\circ}{{\Ranp}_{X}}  \times \stackrel{\circ}{{\Ranp}_{X}} )_{Y-\on{disj}} \rightarrow\stackrel{\circ}{{\Ranp}_{X}} ,$$
	where 
	\[
	\begin{tikzcd}
	({\stackrel{\circ}{{\Ranp}_{X}} } \times \ldots \times \stackrel{\circ}{{\Ranp}_{X}} )_{Y-\on{disj}}\ar[d, rightarrow, ""']\ar[r, rightarrow, ""'] &   \stackrel{\circ}{{\Ranp}_{X}}  \times  \ldots \times \stackrel{\circ}{{\Ranp}_{X}} \ar[d, rightarrow, ""']  \\
	(\Ranp_{Y} \times \ldots \times \Ranp_{Y})_{\on{disj}}\ar[r, rightarrow, ""']&  \Ranp_{Y} \times \ldots \times \Ranp_{Y}
	\end{tikzcd}
	\]
	is a fiber square.
	
\end{cnstr}

The following claim is immediate from the Construction \ref{circ Ran}
\begin{lm}\label{q is etale}
	The natural map $$q: \stackrel{\circ}{{\Ranp}_{X}}   \rightarrow {\Ranp_{Y}}$$
	is etale, schematic and of finite type. 
\end{lm}

\begin{df}
	Similar to Definition \ref{defn-laxunital-factalg-cat} and Definition \ref{defn-laxunital-laxfactalg-cat} we can define
	$$\mathbf{FactAlgCat}( \stackrel{\circ}{{\Ranp}_{X}} ) $$
	and $$\mathbf{LaxFactAlgCat}( \stackrel{\circ}{{\Ranp}_{X}}).$$
\end{df}

\subsubsection{} Since $p$ is a genuine strict morphism in  $\CommAlg(\laxPreSt_{\Corr})$ the map $$p^{\bullet}: \mathbf{CrysCat}({\Ranp}_{X}) \rightarrow \mathbf{CrysCat}(\stackrel{\circ}{{\Ranp}_{X}} )$$ upgrades to 

\[
\begin{tikzcd}
\mathbf{FactAlgCat}( {\Ranp}_{X})\ar[d, hookrightarrow]\ar[r, rightarrow, "p^{\bullet}"] &   \mathbf{FactAlgCat}( \stackrel{\circ}{{\Ranp}_{X}} )\ar[d, hookrightarrow]\\
\mathbf{LaxFactAlgCat}( {\Ranp}_{X})\ar[r, rightarrow, "p^{\bullet}"] &   \mathbf{LaxFactAlgCat}( \stackrel{\circ}{{\Ranp}_{X}} ).
\end{tikzcd}
\]
Moreover, the map $$p_{\bullet}: \mathbf{CrysCat}(\stackrel{\circ}{{\Ranp}_{X}} ) \rightarrow \mathbf{CrysCat}({\Ranp}_{X})$$ upgrades to a functor 
$$p_{\bullet}: \mathbf{LaxFactAlgCat}( \stackrel{\circ}{{\Ranp}_{X}} ) \rightarrow \mathbf{LaxFactAlgCat}( {\Ranp}_{X}),$$
which is right adjoint to $p^{\bullet}$.

\begin{pr}\label{equivalence of fact alg objects for open Ran}
	The functor 
	$$p^{\bullet}: \mathbf{FactAlgCat}(\Ranp) \rightarrow \mathbf{FactAlgCat}( \stackrel{\circ}{{\Ranp}_{X}} )$$
	is an equivalence. 
\end{pr}

\begin{proof}
	We have an adjoint pair 
	\[
	p^{\bullet}: \begin{tikzcd}
	\mathbf{FactAlgCat}({\Ranp_{X}} )\ar[r, rightarrow,  shift left, ""]
	& \arrow[l,  shift left, ""]   \mathbf{FactAlgCat}(\stackrel{\circ}{{\Ranp}_{X}}  ) : \Str \circ p_{\bullet},
	\end{tikzcd}
	\] 
	where $$\Str: \mathbf{LaxFactAlgCat}({\Ranp_{X}} ) \rightarrow \mathbf{FactAlgCat}({\Ranp_{X}} )$$ is defined in \cite[Proposition-Definition 6.0.1]{CF}.
	We claim that $p^{\bullet}$ as above is conservative. It suffices to check this after an etale cover, and the natural maps from prestacks
	$$(	\stackrel{\circ}{{\Ranp}_{X}}  \times \ldots 	\stackrel{\circ}{{\Ranp}_{X}} )_{\on{disj}} := (	\stackrel{\circ}{{\Ranp}_{X}} \times \ldots 	\stackrel{\circ}{{\Ranp}_{X}} ) \times_{{\Ranp_{X}} \times \ldots {\Ranp_{X}}}({\Ranp_{X}} \times \ldots {\Ranp_{X}})_{\on{disj}}$$
	are jointly surjective. So it is enough to show that the pullback of $p^{\bullet}$ along 
	$$m^{\circ}: (	\stackrel{\circ}{{\Ranp}_{X}}  \times \ldots 	\stackrel{\circ}{{\Ranp}_{X}} )_{\on{disj}}  \rightarrow \Ranp_{X}$$
	is conservative. In other words, we need to show that for $$\BA_1, \BA_2 \in \mathbf{FactAlgCat}({\Ranp_{X}} )$$  with $f: \BA_1\rightarrow \BA_2$ such that $$p^{\bullet} \circ f:  \BA_1^{\circ} \rightarrow \BA_2^{\circ}$$ is an equivalence, the map $(m^{\circ})^{\bullet} \circ f$ is also an equivalence. But that is immediate.

	It suffices to check that $\Str \circ p_{\bullet}$ is fully faithful, but this follows from the formula for $\Str$ (\cite[4.4]{CF}) and the fact that $p_{\bullet}$ is fully faithful.
\end{proof}

\begin{df}
	The composition 
	\[
	\begin{tikzcd}
	f^\bullet: \mathbf{FactAlgCat}(\Ranp_Y) \ar[r, rightarrow,  "q^{\bullet}"]
	& \mathbf{FactAlgCat}(\stackrel{\circ}{{\Ranp}_{X}})\arrow[r,  " (p^{\bullet})^{-1}"]
	&  \mathbf{FactAlgCat}(\Ranp),
	\end{tikzcd}
	\] 
	is called \emph{etale pullback of factorization algebra categories}. 
\end{df}

\subsection{Etale pullback of factorization module categories}

\begin{cnstr}
	Let $Z$ be a lax prestack and $x: Z \rightarrow \stackrel{\circ}{{\Ranp}_{X}}$ be a morphism. We will define a $\stackrel{\circ}{{\Ranp}_{X}}$-module $\stackrel{\circ}{{\Ranp}_{X, Z}}$  in  lax prestacks. We have projection morphisms
	$$\prjctn_s, \prjctn_t: \Arr(\stackrel{\circ}{{\Ranp}_{X}}) \rightarrow \stackrel{\circ}{{\Ranp}_{X}}$$
	remembering the source or the target of the arrow. We have a commutative diagram
	\[
	\begin{tikzcd}
	\stackrel{\circ}{{\Ranp}_{X}} \times \Arr(\stackrel{\circ}{{\Ranp}_{X}}) \ar[d, rightarrow, ""'] &   (\stackrel{\circ}{{\Ranp}_{X}}\times \Arr(\stackrel{\circ}{{\Ranp}_{X}}))_{Y-{\on{disj}}} \ar[d, rightarrow, ""']\ar[r, rightarrow, ""']\ar[l, rightarrow, ""'] & \Arr(\stackrel{\circ}{{\Ranp}_{X}})\ar[d, rightarrow, ""'] \\
	\stackrel{\circ}{{\Ranp}_{X}}\times \stackrel{\circ}{{\Ranp}_{X}}&  (\stackrel{\circ}{{\Ranp}_{X}} \times \stackrel{\circ}{{\Ranp}_{X}})_{Y-{\on{disj}}}\ar[r, rightarrow, ""']\ar[l, rightarrow, ""']  & \stackrel{\circ}{{\Ranp}_{X}}
	\end{tikzcd}
	\]
	such that the left square is Cartesian. The top row defines a $\stackrel{\circ}{{\Ranp}_{X}}$-module $\Arr(\stackrel{\circ}{{\Ranp}_{X}})$. Set ${\Ranp_{X_{\dR}, Z}}^{\circ}$ to be the pullback 
	\[
	\begin{tikzcd}
	\stackrel{\circ}{{\Ranp}_{X, Z}}\ar[d, rightarrow, ""']\ar[r, rightarrow, ""'] &   Z \ar[d, rightarrow, "x"']  \\
	\Arr(\stackrel{\circ}{{\Ranp}_{X}}) \ar[r, rightarrow, "\prjctn_s"']&  \stackrel{\circ}{{\Ranp}_{X}},
	\end{tikzcd}
	\]
	which inherits a ${\Ranp_{X_{\dR}}}^{\circ}$-module structure from $\Arr({\Ranp_{X_{\dR}}}^{\circ})$.
	
	The composition 
	\[
	\begin{tikzcd}
	\pi: \stackrel{\circ}{{\Ranp}_{X, Z}}\ar[r, rightarrow, ""'] &   \Arr(\stackrel{\circ}{{\Ranp}_{X}}) \ar[r, rightarrow, "\prjctn_t"']& \stackrel{\circ}{{\Ranp}_{X}}
	\end{tikzcd}
	\]
	is $\stackrel{\circ}{{\Ranp}_{X}}$-linear. We view $\stackrel{\circ}{{\Ranp}_{X, Z}}$ as a lax prestack over $\stackrel{\circ}{{\Ranp}_{X}}$ via $\pi$. 
\end{cnstr}

\begin{lm}
	For $x\in X$, $y \in Y$ such that $f(x) = y$ the natural map $$\tilde{q}: \stackrel{\circ}{{\Ranp}_{X, x}}  \rightarrow {{\Ranp}_{Y, y}}$$
	is etale, schematic and of finite type.
\end{lm}

\begin{df}
	Similar to Notation \ref{factmodcat} we can  define $$\mathbf{FactModCat}(\stackrel{\circ}{{\Ranp}_{X}} ,  \stackrel{\circ}{{\Ranp}_{X, Z}}  ).$$
\end{df}

We have the following statement analogous to Proposition \ref{equivalence of fact alg objects for open Ran}:

\begin{pr}
	Pullback $(p^{\bullet}, \tilde{p}^{\bullet})$ for $$ \tilde{p}: \stackrel{\circ}{{\Ranp}_{X, Z}}  \rightarrow \Ranp_{X, Z}$$ induces an equivalence
	$$\mathbf{FactModCat}(\stackrel{\circ}{{\Ranp}_{X}} ,  \stackrel{\circ}{{\Ranp}_{X, Z}}  ) \cong \mathbf{FactModCat}({\Ranp_{X}},  {\Ranp_{X, Z}} ).$$
\end{pr}

\begin{df}
	For $x \in X, y \in Y$ such that $f(x) = y$ the composition 
	\[
	\begin{tikzcd}
	\mathbf{FactModCat}(\Ranp_Y, \Ranp_{Y, y} ) \ar[r, rightarrow,  "(q^{\bullet}{,}\tilde{q}^{\bullet})"]
	& \mathbf{FactModCat}(\stackrel{\circ}{{\Ranp}_{X}} ,  \stackrel{\circ}{{\Ranp}_{X, x}}  ) \\
	\arrow[r,  "(p^{\bullet}{,}\tilde{p}^{\bullet})"]&  \mathbf{FactModCat}(\Ranp_X, {\Ranp_{X, x}}),
	\end{tikzcd}
	\] 

	is called \emph{etale pullback of factorization algebra and module categories}. 
\end{df}

For a fixed $\mathbf{A} \in \mathbf{FactAlgCat}(\Ranp_Y)$ we thus have 
\begin{equation}\label{etale pullback of fact modules}
f^{\bullet}: \mathbf{A}-\mathbf{FactModCat}(\Ranp_{Y, y}) \rightarrow f^{\bullet}(\mathbf{A})-\mathbf{FactModCat}(\Ranp_{X, x}). 
\end{equation}

\begin{lm}\label{pullback vs restriction}
	The etale pullback of factorization modules functor (\ref{etale pullback of fact modules}) is compatible with factorization restriction defined in \cite[Theorem-Definition 6.0.2]{CF}, i.e. for $$N: \mathbf{C} \rightarrow \mathbf{A}$$ with $\mathbf{C}, \mathbf{A} \in \mathbf{FactAlg}$ we have 
	\[
	\begin{tikzcd}
	\mathbf{A}-\mathbf{FactModCat}(\Ranp_{Y, y})  \ar[d, rightarrow, ""']\ar[r, rightarrow, "\Res_N"'] &   \mathbf{C}-\mathbf{FactModCat}(\Ranp_{Y, y})  \ar[d, rightarrow, ""']  \\
	f^{\bullet}(\mathbf{A})-\mathbf{FactModCat}(\Ranp_{X, x}) \ar[r, rightarrow, "\Res_{f^{\bullet}N}"']&  f^{\bullet}(\mathbf{C})-\mathbf{FactModCat}(\Ranp_{X, x}).
	\end{tikzcd}
	\]
\end{lm}

\begin{proof}

	It suffices to show that the following diagram commutes:
	
	\begin{equation}\label{pullback vs res diagram}
	\begin{tikzcd}
	\mathbf{A}-\mathbf{FactModCat}(\Ranp_{Y, y})  \ar[d, rightarrow, "q^{\bullet}"']\ar[r, rightarrow, "\Res_N"'] &   \mathbf{C}-\mathbf{FactModCat}(\Ranp_{Y, y})  \ar[d, rightarrow, "q^{\bullet}"']  \\
	q^{\bullet}(\mathbf{A})-\mathbf{FactModCat}(\stackrel{\circ}{{\Ranp}_{X, x}}) \ar[r, rightarrow, "\Res_{q^{\bullet}N}"']&  q^{\bullet}(\mathbf{C})-\mathbf{FactModCat}(\stackrel{\circ}{{\Ranp}_{X, x}}).
	\end{tikzcd}
	\end{equation}

	But this follows from the fact that $q^{\bullet}$ commutes with limits and the formula for factorization restriction (\cite[Lemma 6.3.2]{CF} and \cite[4.4]{CF}). 
\end{proof}

\begin{thm}\label{main etale pullback}
	The map (\ref{etale pullback of fact modules}) is an equivalence.
\end{thm}

\begin{proof}
	It suffices to check that 
	$$\tilde{q}^{\bullet}: \mathbf{A}-\mathbf{FactModCat}(\Ranp_{Y, y})  \rightarrow q^{\bullet}(\mathbf{A})-\mathbf{FactModCat}(\stackrel{\circ}{{\Ranp}_{X, x}})$$
	is an equivalence. Denote the factorization algebra object $q^{\bullet}(\mathbf{A})$ by $\mathbf{B}$. 
	
	Note that from Construction \ref{circ Ran} it follows that we have an adjunction
	\[
	q^{\bullet}: \begin{tikzcd}
	\mathbf{FactAlgCat}({\Ranp_{Y}} )\ar[r, rightarrow,  shift left, ""]
	& \arrow[l,  shift left, ""]   \mathbf{FactAlgCat}(\stackrel{\circ}{{\Ranp}_{X}}) : q_{\bullet}.
	\end{tikzcd}
	\] 
	and similarly
	\[
	\tilde{q}^{\bullet}: \begin{tikzcd}
	\mathbf{FactModCat}(\Ranp_Y, \Ranp_{Y, y} ) \ar[r, rightarrow,  shift left, ""]
	& \arrow[l,  shift left, ""]   \mathbf{FactModCat}(\stackrel{\circ}{{\Ranp}_{X}}, \stackrel{\circ}{{\Ranp}_{X, x}})  :  \tilde{q}_{\bullet}.
	\end{tikzcd}
	\] 
	
	By definition of factorization restriction as the 2-Cartesian lift (see \cite[Theorem-Definition 6.0.2]{CF}) we get that the functor 
	$$\tilde{q}^{\bullet}: \mathbf{A}-\mathbf{FactModCat}(\Ranp_{Y, y})  \rightarrow \mathbf{B}-\mathbf{FactModCat}(\stackrel{\circ}{{\Ranp}_{X, x}})$$
	admits a right adjoint described as 
	$$ \Res_{\mu}\circ \tilde{q}_*,$$
	where $\mu: \mathbf{A} \rightarrow q_{\bullet}\mathbf{B} \in \mathbf{FactAlg}({\Ranp_{Y}} )$ comes from the unit of adjunction.
	
	So we need to check that 
	\begin{equation}\label{counit etale factorization}
	\tilde{q}^{\bullet}\circ \Res_{\mu}\circ \tilde{q}_{\bullet} \xrightarrow{\cong} \Id
	\end{equation}
	and 
	\begin{equation}\label{unit etale factorization}
	\Id \xrightarrow{\cong}  \Res_{\mu} \circ \tilde{q}_{\bullet} \circ \tilde{q}^{\bullet}.
	\end{equation}
	
	We first claim that $$\tilde{q}^{\bullet}: \mathbf{A}-\mathbf{FactModCat}(\Ranp_{Y, y})  \rightarrow \mathbf{B}-\mathbf{FactModCat}(\stackrel{\circ}{{\Ranp}_{X, x}})$$ is conservative. 
	
	Indeed, factor $f: X \rightarrow Y$ as a composition of a surjective map $f_1: X \twoheadrightarrow Y_1$ and an open $j: Y_1\hookrightarrow Y$. Then $$q: \stackrel{\circ}{{\Ranp}_{X}}  \rightarrow {\Ranp_{Y}}$$ factors as $$\stackrel{\circ}{{\Ranp}_{X}} \xtwoheadrightarrow{q_1} \Ranp_{Y_1} \xhookrightarrow{j} \Ranp_{Y}.$$
	Then $$\widetilde{q_1}^{\bullet}: \mathbf{A}-\mathbf{FactModCat}(\Ranp_{Y_1, y})  \rightarrow \mathbf{B}-\mathbf{FactModCat}(\stackrel{\circ}{{\Ranp}_{X, x}})$$ is conservative by surjectivity, and $$\tilde{j}^{\bullet}: \mathbf{A}-\mathbf{FactModCat}(\Ranp_{Y, y})  \rightarrow \mathbf{A}-\mathbf{FactModCat}(\Ranp_{Y_1, y})$$ is an equivalence by descent. 
	
	Therefore it suffices to check only that (\ref{counit etale factorization}) is an equivalence. 
	Using (\ref{pullback vs res diagram}) we can rewrite (\ref{counit etale factorization}) as 
	
	\begin{equation}\label{counit etale factorization new}
	\Res_{q^{\bullet}(\mu)} \circ \tilde{q}^{\bullet}\circ \tilde{q}_{\bullet}\xrightarrow{\cong} \Id.
	\end{equation}
	
	Note that the composition 
	$$q^{\bullet}(\mathbf{A}) \xrightarrow{q^{\bullet}(\mu)} q^{\bullet}q_{\bullet}(\mathbf{B}) \xrightarrow{\epsilon} \mathbf{B},$$
	where the latter map comes from count of the adjunction, is the identity. Therefore to show that (\ref{counit etale factorization new}) is identity it suffices to show that 
	\begin{equation}\label{res(e)}
	\tilde{q}^{\bullet}\circ \tilde{q}_{\bullet} \cong \Res_{\epsilon}.
	\end{equation}
	
	For this statement we may assume that $\mathbf{B}$ is any element of $\mathbf{FactAlgCat}(\stackrel{\circ}{{\Ranp}_{X}} )$. Indeed, since $q$ and $\widetilde{q}$ are etale for any $$\mathbf{N} \in \mathbf{B}-\mathbf{FactModCat}(\stackrel{\circ}{{\Ranp}_{X, x}})$$ we get the adjunction 
	\[
	(\epsilon, \widetilde{\epsilon}): \begin{tikzcd}
	(	q^{\bullet}\circ q_{\bullet}(\mathbf{B}), 	\widetilde{q}^{\bullet}\circ \widetilde{q}_{\bullet}(\mathbf{N})) \ar[r, rightarrow,  shift left, ""]
	& \arrow[l,  shift left, ""]   (\mathbf{B}, \mathbf{N})  :  (\epsilon^R, \widetilde{\epsilon}^R).
	\end{tikzcd}
	\] 
	Note that $\widetilde{\epsilon}$ and $\widetilde{\epsilon}^R$ are equivalences on cores, so we conclude by \cite[Corollary B.8.4]{text}.
\end{proof}

\begin{rem}
	Everything in this subsection prior to Theorem \ref{main etale pullback} works as is for factorization algebras and modules in $\D$ as opposed to $\mathbf{CrysCat}$. The proof of Theorem \ref{main etale pullback} in fact can be simplified in that case. Indeed, the morphisms (\ref{counit etale factorization}) and (\ref{unit etale factorization}) obviously become equivalences after taking cores, but the functor of taking cores is conservative.
\end{rem}

\section{Parameterized factorization patterns}\label{paramfact}

\subsection{Parameterized crystals of categories}\label{paramcryscat}

Let $S$ be a 1-affine prestack such that $\QCoh(S)$ is semi-rigid. 
The goal of this subsection is to set up a theory of crystals of categories parameterized by $S$ and derive properties of such needed for the proof of Theorem \ref{Basic Adjunction}.

\begin{df} Let $Y$ be a laft lax prestack. An \emph{$S$-linear crystal of categories} $\BA$ on $Y$ consists of the following data:
	\begin{itemize}
		\item For any $s:R\to Y$ with $R$ a finite type affine scheme, there is a $\D(R)\otimes \QCoh(S)$-module DG category:
		\[ \BA_s \in \D(R) \otimes \QCoh(S)-\mathbf{ModCat} \cong \QCoh(S)-\mathbf{ModCat}(D(R)-\mathbf{ModCat}). \]
		
		\item For any 2-cell
		\[
		\xymatrix{
			& R \ar[rd]^-{s} \ar@{=>}[d]^-\theta \\
			T \ar[ru]^-v \ar[rr]_{t} & & Y,
		}
		\]
		there is a functor
		\[
		\theta_\dagger: \BA_s \to \BA_t
		\]
		intertwining the symmetric monoidal functor $$u^! \otimes \Id:\D(R)  \otimes \QCoh(S)\to \D(T) \otimes \QCoh(S),$$ such that if $\theta$ is invertible, then the induced functor
		\[
		\D(T) \otimes_{\D(S)} \BA_s \to \BA_t
		\]
		is an equivalence.
		\item Certain higher compatibilities. 
	\end{itemize}
\end{df}

\begin{ex}
	We have the constant $S$-linear crystal $\mathbf{Vect}_Y^S$. 
\end{ex}

As usual, there are two notions of morphisms between crystals of categories on a lax prestack: lax ones and strict ones defined as in Definition \ref{defn-lax-morphism-cryscat}.
 Denote by $\Maps(\mathbf{C}, \mathbf{D})$ ($\Maps_{\str}(\mathbf{C}, \mathbf{D})$) the category of (strict) functors from $\BC$ to $\BD$. 
 
 \begin{ntn}
 	Let 
 	\[\mathbf{CrysCat}^S(Y),\; \mathbf{CrysCat}^{S}(Y)^{\str}\] 
 	be the $(\infty,2)$-category of crystals of categories on $Y$, with 1-morphisms given by morphisms (resp. strict morphisms) between crystals of categories, and 2-morphisms given by natural transformations.
 \end{ntn}
	
\medskip

We now state the $S$-linear avatars of claims in \cite[Subsections A.2-A.4]{CF} and note that the proofs are completely similar. 

\begin{lm}\label{pointwise maps}
	Colimits and finite limits in $\Maps_{\str}(\mathbf{C}, \mathbf{D})$ and $\Maps(\mathbf{C}, \mathbf{D})$ are calculated point-wise, i.e. for any affine $S$ with $s: S \rightarrow Y$ the functors $$\Maps_{\str}(\mathbf{C}, \mathbf{D}) \rightarrow \on{Fun}_{\D(S) \otimes \QCoh(S)}(\mathbf{C}(s), \mathbf{D}(s))$$ and $$\Maps(\mathbf{C}, \mathbf{D}) \rightarrow \on{Fun}_{\D(S) \otimes \QCoh(S)}(\mathbf{C}(s), \mathbf{D}(s))$$
	commute with colimits and finite limits.
\end{lm}

\begin{cor}
	The categories $\Maps(\mathbf{C}, \mathbf{D})$ and $\Maps_{\str}(\mathbf{C}, \mathbf{D})$ are presentable and stable.
\end{cor}

\begin{df}
	Define (lax) global sections of $\mathbf{C}$ to be $$\Gamma_{\str}^S(Y, \mathbf{C}):= \Maps_{\str}(\mathbf{Vect}_Y^S, \mathbf{C}),$$
	$$\Gamma^S(Y, \mathbf{C}):= \Maps(\mathbf{Vect}_Y^S, \mathbf{C}),$$
\end{df}

\begin{rem}
	For $Y \in \PreSt$ we have $\Gamma^S(Y, \mathbf{C}) \cong \Gamma_{\str}^S(Y, \mathbf{C})$.
\end{rem}

\begin{ex}
	We have 
	$$\D(Y) \otimes \QCoh(S) \cong \Maps(\mathbf{Vect}_Y^S, \mathbf{Vect}_Y^S) \cong \Gamma^S(Y, \mathbf{Vect}_Y^S).$$
\end{ex}

\begin{cnstr}
	For a morphism $f: \mathbf{Y_1} \rightarrow \mathbf{Y_2}$ in $\laxPreSt$ we have natural functors 
	$$f^{\bullet}: \mathbf{CrysCat}^S(\mathbf{Y_2}) \rightarrow \mathbf{CrysCat}^S(\mathbf{Y_1}),$$
	and 
	$$f^{\bullet}: \mathbf{CrysCat}^S(\mathbf{Y_2})^{\str} \rightarrow \mathbf{CrysCat}^S(\mathbf{Y_1})^{\str}.$$
	We will call them the pullback functors.
\end{cnstr}

\begin{lm}\label{adjoints-base change strict}
	Let $p: Y \rightarrow Z$ be a schematic morphism of prestacks. Then the functor 
	$$p^{\bullet}: \mathbf{CrysCat}^S(Z)^{\str} \rightarrow \mathbf{CrysCat}^S(Y)^{\str}$$
	admits a right adjoint 
	$$p_{\bullet}: \mathbf{CrysCat}^S(Y)^{\str}\rightarrow\mathbf{CrysCat}^S(Z)^{\str}$$
	satisfying the base change property.
\end{lm}

\begin{proof}
	Note that $\mathbf{CrysCat}$ is tensored over $\mathbf{DGCat}$, and since $$p^{\bullet}:  \mathbf{CrysCat}(Z)^{\str} \rightarrow \mathbf{CrysCat}(Y)^{\str}$$ is $\mathbf{DGCat}$-linear it upgrades to a functor between $S$-linear crystals of categories, which coincides with the pullback 
	$$p^{\bullet}: \mathbf{CrysCat}^S(Z)^{\str} \rightarrow \mathbf{CrysCat}^S(Y)^{\str}.$$
	
	Note that  in general $$p_{\bullet}: \mathbf{CrysCat}(Y)^{\str}\rightarrow\mathbf{CrysCat}(Z)^{\str}$$ is lax  $\mathbf{DGCat}$-linear. However, for dualizable categories it is strict   $\mathbf{DGCat}$-linear, therefore using semi-rigidity of $\QCoh(S)$ the assertion follows from \cite[Lemma A.1.6]{CF}.
\end{proof}

\begin{cnstr}
	Since $\mathbf{Vect}_Y^S$ is the symmetric monoidal unit in $\mathbf{CrysCat}^S(Y)$ and by Lemma \ref{pointwise maps} we get that $$\Gamma^S(Y, \mathbf{C}), \text{  }\Maps(\mathbf{C}, \mathbf{D}) \in \D(Y) \otimes\QCoh(S)-\mathbf{ModCat}.$$
\end{cnstr}

\begin{dfprop}\label{adjoints+base change lax}
	Let $p: Y \rightarrow Z$ be a finite type schematic morphism in $\laxPreSt$ such that for any $S \in \AffSch$ the map $Y(S) \rightarrow Z(S)$ is a Cartesian fibration fibered in $\infty$-groupoids. Then:
	
	\begin{enumerate}
		\item The functor 	$$p^{\bullet}: \mathbf{CrysCat}^S(Z) \rightarrow \mathbf{CrysCat}^S(Y)$$
		admits a right adjoint $$p_{\bullet}: \mathbf{CrysCat}^S(Y)\rightarrow\mathbf{CrysCat}^S(Z)$$
		satisfying the base change property;
		
		\item The adjoint pair $(p^{\bullet}, p_{\bullet})$ in $(1)$ extends one in Lemma \ref{adjoints-base change strict}. More precisly, both functors in $(1)$ preserve strict morphisms, and for any 
		$$\mathbf{C} \in \mathbf{CrysCat}^S(Y), \text{  } \mathbf{B} \in \mathbf{CrysCat}^S(Z)$$
		the counit and unit
		$$p^{\bullet} \circ p_{\bullet} (\mathbf{C}) \xrightarrow{\rho(\mathbf{C})} \mathbf{C}, \text{  }  \mathbf{B}\xrightarrow{\mu(\mathbf{B})}p_{\bullet} \circ p^{\bullet} (\mathbf{B})$$
		are contained respectively in $\mathbf{CrysCat}^S(Y )$ and $\mathbf{CrysCat}^S(Z)$;
		
		\item If $p$ is etale (see \cite[Definition A.3.1]{CF}), then $ \mathbf{B}\xrightarrow{\mu(\mathbf{B})}p_{\bullet} \circ p^{\bullet} (\mathbf{B})$ has a right adjoint $$p^{\bullet} \circ p_{\bullet} (\mathbf{B}) \xrightarrow{\varrho(\mathbf{B})} \mathbf{B}$$ contained in $\mathbf{CrysCat}^S(Z)$, which is functorial in $\mathbf{B}$.
	\end{enumerate}
\end{dfprop}

\begin{proof}
	Follows from \cite[Proposition-Definition A.3.3]{CF} by the same argument as in Lemma \ref{adjoints-base change strict}.
\end{proof}

\subsection{Parameterized factorization patterns}\label{Parameterized factorization patterns}

\begin{ntn}
	We can define, parallel to notions described in Section \ref{prelims},  the following evident categories:
	\begin{enumerate}
		\item Set $\mathbf{FactAlgCat}^S$, and call its objects $S$-linear factorization categories over $\Ranp^{ch}$,
		
		\item For $\BA \in \mathbf{FactAlgCat}^S$ set $\FactAlg^S(\BA)$,
		and call its objects $S$-linear factorization algebras over $\Ranp^{ch}$,
		
		\item Set $\mathbf{FactModCat}^S,$
		
		\item For $\mathbf{A} \in \mathbf{FactAlgCat}^S$ set 
		$$\mathbf{A}-\mathbf{FactModCat}^S,$$
		
		\item For $(\BA, \BM ) \in \mathbf{FactModCat}^S$ and $A \in \FactAlg^S(\BA)$ set 
		$$A\FactMod^S(\BM).$$
	\end{enumerate}
\end{ntn}

\begin{rem}
	To get rigorous definitions of the above objects, note that $\mathbf{CrysCat}^S$ is an example of a $\mathbf{Type}$ (as in \cite[Subsection 1.2]{CF}), so the formalism \cite[Section 1-3]{CF} applies. 
\end{rem}

Having established results about $\mathbf{CrysCat}^S$ is subsection \ref{paramcryscat} we get that the main theorem of \cite{CF} upgrades to the following statement.
\begin{thmdefn}[{\cite{CF}}]\label{Basic Adjunction}
	\label{thmdefn-factresS}
	The forgetful functor
	\begin{equation} \label{eqn-basic-proj}
	\mathbf{FactModCat}^S \to \mathbf{FactAlgCat}^S
	\end{equation}
	is a (1,2)-Cartesian fibration. Given a morphism $\Phi: \BA_1\to \BA_2$ in the target, we denote the contravariant transport functor along $\Phi$ by
	\[
	\mathbf{Res}^S_\Phi: \BA_2\mathbf{-FactModCat}^S \to \BA_1\mathbf{-FactModCat}^S,
	\]
	and call it the \emph{$S$-linear factorization restriction functor along $\Phi$}. Moreover, for $\BC_2\in \BA_2\mathbf{-FactModCat}^S$, there is a canonical equivalence
	\begin{equation}
	\label{eqn-thmdefn-factresS}
	\mathbf{Res}^S_\Phi(\BC_2)_x \simeq \Phi(\on{unit}_{\BA_1})\on{-FactMod}^S(\BC_2).
	\end{equation}
	In particular, if $\Phi$ is strictly unital, then $\mathbf{Res}^S_\Phi(\BC_2)$ and $\BC_2$ have equivalent cores.
\end{thmdefn}

\begin{rem}\label{restrictino as a limit}
	Exactly as in \cite[Lemma 6.3.2 and 4.4]{CF} we have that $\mathbf{Res}^S_\Phi$ can be computed as a limit. 
\end{rem}

As in \cite{text} we get 

\begin{cor}\label{Basic Adjunction Cor}
	\label{cor-factmod-in-factresS}
	Let $\Phi:\BA\to \BA'$ be a morphism in $\mathbf{FactAlgCat}^S$ and $\BC' \in \BA'\mathbf{-FactModCat}^S$. Let $\cA\in \FactAlg^S(\BA)$. Then there is a canonical equivalence
	\[
	\cA\on{-FactMod}( \mathbf{Res}^S_\Phi(\BC') ) \simeq \Phi(\cA)\on{-FactMod}(\BC').
	\]
\end{cor}

\begin{df}
	Let $\BA$ be an $S$-linear factorization algebra category and $\BC$ be an $S$-linear factorization $\BA$-module category. Let $\phi:\cA_1\to \cA_2$ be a morphism in $\FactAlg^S(\BA)$, viewed as a 2-morphism in $\mathbf{FactAlgCat}$. Then the contravariant transport for (\ref{eqn-basic-proj}) along this 2-morphism defines a factorization $(\Vect \otimes \QCoh(S))_{\Ranp}$-linear functor 
	\[
	\mathbf{Res}^S_{\cA_2}((\Vect \otimes \QCoh(S))_{\Ranp}) \to \mathbf{Res}^S_{\cA_1}((\Vect \otimes \QCoh(S))_{\Ranp}).
	\]
	The core of it is a functor
	\[
	\cA_2\on{-FactMod}^S(\BC) \to  \cA_1\on{-FactMod}^S(\BC)
	\]
	which we denote by $\on{Res}^S_\phi$ and call it the \emph{factorization restriction functor along $\phi$}.
	
\end{df}

As in \cite[Lemma B.8.1]{text} we get 
\begin{lm}
	Let 
	\[
	\Phi: \BA \rightleftarrows \BA': \Psi
	\]
	be an adjunction in the 2-category $\mathbf{FactAlgCat}^S$. Then $\Phi$ is strictly unital.
\end{lm}

The following result follows from Theorem-Definition \ref{thmdefn-factresS} and straightening, and is a direct generalization of \cite[Theorem B.8.2]{text}. We call it the \emph{basic $S$-linear adjunction for factorization restrictions}.

\begin{thm}
	\label{thm-basic-adjS}
	Let $\Phi: \BA \rightleftarrows \BA': \Psi$ be an adjunction in the 2-category $\mathbf{FactAlgCat}^S$. Then there is a canonical adjunction in $\mathbf{BiCat}_\infty$:
	\[
	\mathbf{Res}^S_\Phi: \BA'\mathbf{-FactModCat}^S \rightleftarrows \BA\mathbf{-FactModCat}^S: \mathbf{Res}^S_\Psi.
	\]
\end{thm}

\begin{rem}\label{Basic Adjunction Rem}
	Let $\Phi: \BA \rightleftarrows \BA': \Psi$ be an adjunction in the 2-category $\mathbf{FactAlgCat}^S$. Note that
	\begin{itemize}
		\item[(1)]
		For $\BC'\in \BA'\mathbf{-FactModCat}^S$, the core of $\mathbf{Res}^S_\Phi(\BC')$ is equivalent to that of $\BC'$ because $\Phi$ is strictly unital.
		\item[(2)]
		For $\BC\in \BA\mathbf{-FactModCat}^S$, the core of $\mathbf{Res}^S_\Psi(\BC)$ is equivalent to $\Psi(\on{unit}_{\BA})\on{-FactMod}^S(\BC)$.
	\end{itemize}
	Hence we can formulate the basic adjunction as the following equivalence:
	\[
	\on{FactFun}_{\BA}( \BC'_x, \BC_x   ) \simeq \on{FactFun}_{\BA'}( \BC'_x, \Psi(\on{unit}_{\BA})\on{-FactMod}^S_{\BC_x}),
	\]
	where we use the point of view of factorization module categories \emph{concentrated} at $x$ (rather than \emph{supported} on $\Ranp$).
	
\end{rem}

\begin{cor}\label{B.8.4}
	\label{cor-crit-factres}
	Let
	\[
	(\Phi,\Xi): (\BA,\BC) \rightleftarrows (\BA',\BC'): (\Psi,\Upsilon)
	\]
	be an adjunction in $\mathbf{FactModCat}^S$ such that $\Xi$ and $\Upsilon$ induce equivalences between the cores of $\BC$ and $\BC'$. Then the factorization $\BA$-linear functor
	\[
	\Xi^\enh: \BC \to \mathbf{Res}^S_\Phi(\BC')
	\]
	induced by $\Xi$ is an equivalence.
	
\end{cor}

\begin{proof}
	Generalizes directly from \cite[Corollary B.8.4] {text} and results in Subsection \ref{paramcryscat}.
\end{proof}

A converse to Corollary \ref{B.8.4} is also true, namely
\begin{cor}\label{Restriction adjoint}
	Let 
	\[
	\Phi: \BA \rightleftarrows \BA': \Psi
	\]
	be an adjunction in $\mathbf{FactAlgCat}^S$. Let $\BC' \in \BA'\mathbf{-FactModCat}^S$. Then there the natural map 
	$$(\Phi,\Xi): (\BA,\mathbf{Res}^S_\Phi(\BC')) \rightarrow (\BA',\BC')$$
	admits a right adjoint, denoted by $(\Psi,\Upsilon)$.
\end{cor}
We deduce the claim from Theorem-Definition \ref{Basic Adjunction} and the following construction. 
\begin{cnstr}
	Let $\pi: \BE \to \BB$ be a $(1,2)$-Cartesian fibration and $f: x \rightleftarrows y: g$ be the adjunction in $\BB$. Let $v$ be an element over $y$. We will construct a right adjoint in $\BE$ to the map $$L: f^{\dagger}(v) \rightarrow v.$$
	
	 Consider the 2-morphism $e: fg \rightarrow \Id_y$ and the 1-morphism $\Id_v$ over $\Id_y$. Pulling back $\Id_v$ along $e$ we obtain $e^{\dagger}(\Id_v):v \rightarrow v$ over $fg$. This morphism factors through $f^{\dagger}(v)$, so we obtain a morphism $R: v \rightarrow f^{\dagger}(v)$. We claim that $R$ is right adjoint to $L$.
	 
	 Indeed, the morphism $e^{\dagger}(\Id_v) \rightarrow \Id_v$ gives $LR \rightarrow \Id_v$, this is the counit of the desired adjunction. Now consider the 2-morphism $f \rightarrow fgf \rightarrow f$. Note that the composition is $\Id_f$. Pulling back $L$, we obtain $L \rightarrow e^{\dagger}(\Id_v) \circ L \rightarrow L$. But since $L$ is a Cartesian arrow, the first morphism $$L \rightarrow e^{\dagger}(\Id_v) \circ L\cong LRL$$ must be the horizontal composition of a 2-morphism $\Id_{f^{\dagger}(v)} \rightarrow RL$  with $\Id_L$. This 2-morphism $$\Id_{f^{\dagger}(v)} \rightarrow RL$$ is the desired unit of the adjunction. The axioms of adjunction follow by unwinding the definitions. 
\end{cnstr}

\begin{cor}\label{factres linear}
	The functor 	\[
	\mathbf{Res}^S_\Phi: \BA'\mathbf{-FactModCat}^S \rightarrow \BA\mathbf{-FactModCat}^S
	\]
	is $\mathbf{DGCat}$-linear. 
\end{cor}
\begin{proof}
	Let $\BD \in \mathbf{DGCat}$. Then by Corollary \ref{Restriction adjoint} we get an adjunction 
		\[
	(\Phi,\Xi \otimes \Id_{\BD}): (\BA,\mathbf{Res}^S_\Phi(\BC') \otimes \BD) \rightleftarrows (\BA',\BC' \otimes \BD): (\Psi,\Upsilon\otimes \Id_{\BD}).
	\]
	But now Corollary \ref{B.8.4} implies that $\mathbf{Res}^S_\Phi(\BC') \otimes \BD \cong \mathbf{Res}^S_{\Upsilon\otimes \Id_{\BD}}(\BC' \otimes \BD)$, as desired.
\end{proof}

\subsubsection{}
Let $\mathbf{A}$ be an $S$-linear factorization category, and $\mathbf{C}$ be an $S$-linear factorization $A$-module category. Let $\mathbf{C}_{x }$ denote the fiber of $\mathbf{C}$. Denote by $i_{\Ranp_{X, x}}^! \otimes \Id_{\QCoh(S)}$ 
the restriction functor $\mathbf{C}\rightarrow \mathbf{C}_{x}$. Let $U = X \setminus x$. Let $j_{\Ranp}^! \otimes \Id_{\QCoh(S)}$ denote the restriction functor  $$ \mathbf{C} \rightarrow \mathbf{C}|_{\Ranu  \times  x }.$$

Let $A \in \mathbf{A}$ be an $S$-linear factorization algebra.
\begin{pr}\label{proposition 5.3.3}
	Let $c \in \mathbf{C}_{x \times S}$ be such that the partially defined left adjoint ${j_{\Ranp}}_! \otimes \Id_{\QCoh(S)}$ is defined on $$A \boxtimes_S c \in \mathbf{C}|_{\Ranu  \times  x}.$$
	Then the partially defined left adjoint $\ind_{A}^S$ to forgetful functor 
	$$ A\FactMod^S(\mathbf{C}) \rightarrow \mathbf{C}_{x_0 \times S}$$
	is defined  on $c$. We have $$\ind_{A}^S(c) \cong  ({j_{\Ranp}}_! \otimes \Id_{\QCoh(S)})(A \boxtimes_S c).$$
\end{pr}

\begin{proof}
	Analogous to \cite[Lemma 5.3.3]{text}.
\end{proof}

\subsection{Parameterized vs non-parameterized factorization patterns}

In this subsection we study the relationship between the usual and $S$-linear factorization algebra and module categories. 

We first note that the natural transformation
\begin{equation}
(\QCoh(S) \otimes -):\mathbf{CrysCat} \rightarrow \mathbf{CrysCat}^S
\end{equation}
of functors 
$$\AffSch_{\on{ft}} \rightarrow \BiCat$$
induces
\begin{equation}
(\QCoh(S) \otimes -):\mathbf{FactAlgCat} \rightarrow \mathbf{FactAlgCat}^S.
\end{equation}

There is an evident pair of adjoint functors 
\begin{equation}\label{adjoint ambi}
\QCoh(S) \otimes -: 
\begin{tikzcd}
\mathbf{CrysCat}(\Ranp)\ar[r, rightarrow,  shift left, ""]
& \arrow[l,  shift left, ""]   \mathbf{CrysCat}^S(\Ranp) : \Oblv.
\end{tikzcd}
\end{equation}

\begin{lm}
	The adjunction (\ref{adjoint ambi}) is ambidextrous.
\end{lm}

\begin{proof}
	Follows from the fact that $\QCoh(S)$ is semi-rigid and thus self-dual.
\end{proof}

\begin{cor}\label{ambidextrous adjunction S-linear}
	The induced functor
	\begin{equation}
	\QCoh(S) \otimes -: 
	\begin{tikzcd}
	\BA \mathbf{-FactModCat}\ar[r, rightarrow,  shift left, ""]
	& \arrow[l,  shift left, ""]   \BA\otimes \QCoh(S) \mathbf{-FactModCat^S}
	\end{tikzcd}
	\end{equation}
	admits a right adjoint $\Oblv$ given by $$(\BA\otimes \QCoh(S), \BM^S) \mapsto (\BA, \BM^S).$$ We claim that this adjunction is ambidextrous.
\end{cor}

By the same argument, we get 
\begin{cor}\label{ambidextrous adjunction S-linear 2}
	Let $S^{\prime}$ a 1-affine prestack such that $\QCoh(S^{\prime})$ is semi-rigid. Then for $\BA \in \mathbf{FactAlgCat}^S$ we have an ambidextrous adjunction
	\begin{equation}
	\QCoh(S^{\prime}) \otimes -: 
	\begin{tikzcd}
	\BA \mathbf{-FactMod}^S\ar[r, rightarrow,  shift left, ""]
	& \arrow[l,  shift left, ""]   \BA\otimes \QCoh( S^{\prime}) \mathbf{-FactMod}^{S \times S^{\prime}}: \Oblv.
	\end{tikzcd}
	\end{equation}
\end{cor}

\section{Factorization modules over commutative algebras}\label{appendix: factmodules over commutative alg}

Let $\cY$ be a prestack affine over $X_{\dR} \times T$, i.e. $\cY = \Spec_{X_{\dR} \times T}(A)$, where $$A \in \CommAlg(\D(X) \otimes \QCoh(T)).$$ The goal of this section is to give a rather elementary proof of the following folklore statement.

\begin{thm}\label{factmod commutative}
	Assume that $A \in \D_{\indhol}(X) \otimes \QCoh(T)$. Then we have an equivalence 
	$$\QCoh(\hormerxJetsF_{\nabla, T \times U_{\dR}}^{\mer, x}(\cY)_x) \cong A\FactMod^T.$$
\end{thm}

\begin{rem}
	Since the first draft of this paper was circulated, the proof of this statement appeared in \cite{GLC} in a greater generality (without the ind-holonomicity condition). However, we still keep this section since the present proof is less technically heavy. 
\end{rem}

\begin{df}
	Define lax prestack $\D_{\on{indhol}}$ to be the left Kan extension of the functor 
	$$\D_{\on{indhol}}: \AffSch_{\on{ft}}^{\op} \rightarrow \on{Cat}_\infty$$
	along the embedding $ \AffSch_{\on{ft}}^{\op} \hookrightarrow \AffSch^{\op}$. 
	
	For any lax prestack $\cY$ define the category of ind-holonomic $D$-modules on $\cY$ as $\Maps(\cY, \D_{\on{indhol}})$. 
\end{df}
\begin{rem}
Note that for an affine scheme $S$ we get $$\D_{\on{indhol}}(S) \cong \underset{S_i \in \AffSch_{\on{ft}} S \rightarrow S_i}{\colim}\D_{\on{indhol}}(S_i).$$
Hence for an affine scheme of finite type we recover the usual definition of ind-holonomic $D$-modules.
\end{rem}

\begin{rem}
	Recall from \cite[2.3.1]{G2} that for $\cY_1, \cY_2 \in \laxPreSt$ and $f: \cY_1 \rightarrow \cY_2$ the map $$f^!: \D_{\on{indhol}}(\cY_2) \rightarrow \D_{\on{indhol}}(\cY_1)$$ admits a left adjoint $f_!$. Hence for a map of curves $$X \rightarrow Y$$ and the induces map $$f: \Ranp \rightarrow \Ranp_Y$$ the left adjoint $f_!$ to the functor $f^!: \D( \Ranp) \rightarrow \D(\Ranp_Y)$ is defined on the subcategory $\D_{\on{indhol}}(\Ranp_Y) \subset \D(\Ranp_Y)$.
	\end{rem}

\begin{lm}
	The object $$A \in \D(\Ranp) \otimes  \QCoh(T)$$ lies in the full subcategory $\D_{\indhol}(\Ranp) \otimes \QCoh(T)$, therefore the functor ${j_{\Ranp}}_! \otimes \Id_{\QCoh(T)}$ is defined on $A$. 
\end{lm}

\begin{proof}
	
	It suffices to show that restriction of $A$ to every $X^K $ lies in $\D_{\indhol}(X^K) \otimes \QCoh(T)$.  Restriction to $X $ lies in $\D_{\indhol}(X)\otimes \QCoh(T)$ by assumption.
	
	For general $K$ denote by $S_K$ the $(1,1)$-category indexing data $K \twoheadrightarrow I \twoheadrightarrow J$, where we allow morphisms of
	diagrams that are contravariant in $I$ and covariant in $J$, and surjective term-wise. Then one has
	$$(A)_{X^K} \cong \underset{K \twoheadrightarrow I \twoheadrightarrow J \in S_K}{\colim} ({\Delta_{q \circ p}}_! \boxtimes \Id_T) \circ (\Delta_q^!\boxtimes \Id_T) (A^{\boxtimes_T I}) .$$
	In particular, since  ${j_{\Ranp}}_! \otimes \Id_{\QCoh(T)}$ commutes with colimits and all three operations preserve ind-holonomicity on the first factor we get the result.

\end{proof}

Writing any $\cF \in \QCoh(T)$ as a colimit of $\cO_{T}$ we see that ${j_{\Ranp}}_! \otimes \Id_{\QCoh(T)}$ is defined on $$A \boxtimes_T \cF$$ and 
$$(i_{\Ranp_{X, x}}^! \otimes \Id_{T})({j_{\Ranp}}_! \otimes \Id_{T})(A \boxtimes_T \cF) \cong  (i_{\Ranp_{X, x}}^! \otimes \Id_{T})({j_{\Ranp}}_! \otimes \Id_{T}) (A) \otimes \cF .$$

Note that $i_{\Ranp_{X, x}}^! \otimes \Id_{\QCoh(T)}$ is continuous and conservative, hence by Barr-Beck Theorem and Proposition \ref{proposition 5.3.3} we get that 
\begin{equation}
	A\FactMod^T \cong (i_{\Ranp_{X, x}}^! \otimes \Id_{T})({j_{\Ranp}}_! \otimes \Id_{T}) (A) \Mmod(\QCoh(T)).
\end{equation}

We now calculate $M:= (i_{\Ranp_{X, x}}^! \otimes \Id_{T})({j_{\Ranp}}_! \otimes \Id_{T}) (A)$. 
\subsubsection{}
We begin by noticing that the functor $j_{\Ranp}^! \times \Id_{\QCoh(T)}$ induces a functor 
$$(j_{\Ranp}^! \otimes \Id_{\QCoh(T)})^{\text{enh}}:   \CFactAlg^T(\D(\Ranp_{X}) \otimes \QCoh(T)) \rightarrow \CFactAlg^T(\D(\Ranu) \otimes  \QCoh(T)).$$

Tautologically, the diagram 

\[
\begin{tikzcd}
\CFactAlg^T(\D(\Ranp_{X}) \otimes \QCoh(T)) \arrow[d, "(j_{\Ranp}^! \otimes \Id_{\QCoh(T)})^{\text{enh}}"] \ar[r, "", "\oblv"']&   \D(\Ranp_{X}) \otimes \QCoh(T)\arrow[d, "j_{\Ranp}^! \times \Id_{\QCoh(T)}"]    \\
\CFactAlg^T(\D(\Ranu) \otimes  \QCoh(T))  \ar[r, "", "\oblv"']  &  \D(\Ranu) \otimes  \QCoh(T)
\end{tikzcd}
\]
commutes.

\begin{pr}\label{left adjoint}
	Assume that $$\cF \in \CFactAlg^T(\D(\Ranu) \otimes  \QCoh(T))$$ and on the underlying sheaf the partial left adjoint ${j_{\Ranp}}_! \otimes \Id_{\QCoh(T)}$ is defined. Then partially defined left adjoint $$((j_{\Ranp}^! \otimes \Id_{\QCoh(T)})^{\text{enh}})^L$$ is defined on $\cF$ and 
	$$\oblv((j_{\Ranp}^! \otimes \Id_{\QCoh(T)})^{\text{enh}})^L(\cF) 
	\cong ({j_{\Ranp}}_! \otimes \Id_{\QCoh(T)})(\oblv(\cF)).$$

\end{pr}

Consider the following generalization of \cite[Proposition 5.6.10]{GL}:

\begin{lm}\label{partial left adjoint CALG}
	let $C$ and $D$ be presentable symmetric monoidal $\infty$-categories with unit objects $\mathbf{1}_C$ and $\mathbf{1}_D$. Assume that tensor product on $C$ and $D$ preserves colimits separately in each variable. Let $$G: D \rightarrow C$$ be a lax symmetric monoidal functor. 
	
	Assume that partially defined left adjoint $F$ is defined on the full subcategory of $C$ generated by $\mathbf{1}_C$ and powers of $\oblv(c)$ for some $c \in \CommAlg(C)$ and $\oblv: \CommAlg(C) \rightarrow C$. Assume in addition that $F$ is symmetric monoidal and that the unit map $\mathbf{1}_C \rightarrow G(\mathbf{1}_D)$ induces an equivalence $F(\mathbf{1}_C ) \rightarrow \mathbf{1}_D$.  
	
	Then the left adjoint $F_+$ to the functor $$G: \CommAlg(D) \rightarrow \CommAlg(C)$$ induced by $G$ is defined on $c$ and 
	$$\oblv(F_+(c)) \cong F(\oblv(c)).$$
\end{lm}

\begin{proof}
	Denote by $C^{\prime}$ the full subcategory of $C$ generated by $\mathbf{1}_C$ and powers of $\oblv(c)$ and by $\iota$ the embedding. By \cite[Proposition 5.6.6]{GL} and \cite[Proposition 5.6.10]{GL} the pair of adjoint functors 
	\[
	\begin{tikzcd}
	C^{\prime}  \arrow[r,  hook, shift left, "\iota"]
	& \arrow[l, shift left, "\iota^R"]  C  
	\end{tikzcd}
	\]
	induces the pair of adjoint functors between unital commutative algebra objects
	\[
	\begin{tikzcd}
	\CommAlg(C^{\prime})  \arrow[r, hook, shift left, "\iota_+"]
	& \arrow[l, shift left, "\iota^R_+"]  \CommAlg(C).
	\end{tikzcd}
	\]
	
	Moreover, the pair of adjoint functors
	\[
	\begin{tikzcd}
	C^{\prime}  \arrow[r,  shift left, "F \circ \iota"]
	& \arrow[l, shift left, "\iota^R \circ G"]  D 
	\end{tikzcd}
	\]
	induces the pair of adjoint functors between unital commutative algebra objects
	\[
	\begin{tikzcd}
	\CommAlg(C^{\prime})  \arrow[r, shift left, "{(F \circ \iota)}_+"]
	& \arrow[l, shift left, "{(\iota^R \circ G)}_+"]  \CommAlg(D).
	\end{tikzcd}
	\]
	Since $(\iota^R \circ G)_+ \cong \iota^R_+ \circ G_+$ we see that partial left adjoint $F_+$ is defined on the subcategory $\CommAlg(C^{\prime})$ and is given by $(F \circ \iota)_+$. Moreover, since by construction $c$ lies in the full subcategory $\CommAlg(C^{\prime})$ we get that 
	$$ F(\oblv(c)) \cong (F \circ \iota)(\oblv(c)) \cong \oblv((F \circ \iota)_+(c)).$$
\end{proof}

\begin{proof}[Proof of Proposition \ref{left adjoint}] 
	Recall (e.g. from \cite[Definition 5.6.1]{GL}) that the structure of a $T$-linear commutative factorization algebra in $\D(\Ranp ) \otimes \QCoh(T)$ is equivalent to the structure of a (unital) commutative algebra in $\D(\Ranp) \otimes \QCoh(T)$ under convolution with the property that the morphism $$\cF \boxtimes_T \cF \rightarrow (\widetilde{\add}^! \otimes \Id_{\QCoh(T)})\cF $$ restricted to the disjoint locus is an isomorphism. In other words, the  symmetic monoidal structure is given by $$\cF, \cG \in \D(\Ranp) \otimes \QCoh(T) \mapsto (\widetilde{\add}_! \otimes \Id_{\QCoh(T)})(\cF \boxtimes_T \cG).$$
	(We postpone the discussion of the fact that $\add_!$ if well-defined to Lemma \ref{left adjoint to addition on unital Ran} and finish the proof first.) 
	
	Denote by $$\CommAlg^{\text{conv}}( \D(\Ranp) \otimes \QCoh(T))$$ the category of (unital) commutative algebras in $\D(\Ranp) \otimes \QCoh(T)$  with respect to the convolution multiplication. We will first prove that for $$\cF \in \CommAlg^{\text{conv}}(\D(\Ranu) \otimes \QCoh(T))$$ such that ${j_{\Ranp}}_! \otimes \Id_{\QCoh(T)}$ is defined on $\oblv(\cF)$, the partial left adjoint $((j_{\Ranp}^! \otimes \Id_{\QCoh(T)})^{\text{enh}})^L$ is defined on $\cF$, and $$\oblv(((j_{\Ranp}^! \otimes \Id_{\QCoh(T)})^{\text{enh}})^L(\cF)) \cong ({j_{\Ranp}}_! \otimes \Id_{\QCoh(T)})(\oblv(\cF)).$$
	This follows from Lemma \ref{partial left adjoint CALG}. Indeed, if ${j_{\Ranp}}_! \otimes \Id_{\QCoh(T)}$ is defined on $c \in  \D(\Ranu) \otimes \QCoh(T)$ then it is also defined on all powers of $c$ for convolution multiplication, and if ${j_{\Ranp}}_! \otimes \Id_{\QCoh(T)}$ is defined on $c, d \in  \D(\Ranu) \otimes \QCoh(T)$ then it is also defined on $c \otimes d$ and is equal to $$({j_{\Ranp}}_! \otimes \Id_{\QCoh(T)})(c) \otimes ({j_{\Ranp}}_! \otimes \Id_{\QCoh(T)})(d).$$
	
	It is left to show that if for $$\cF \in \CFactAlg^T(\D(\Ranu) \otimes  \QCoh(T))$$ the partial left adjoint ${j_{\Ranp}}_! \otimes \Id_{\QCoh(T)}$ is defined on $\oblv(\cF)$, then $$((j_{\Ranp}^! \otimes \Id_{\QCoh(T)})^{\text{enh}})^L(\cF) \in \CFactAlg^T(\D(\Ranp_{X}) \otimes \QCoh(T)) .$$  In other words, we need to check that the morphism 
	$$({j_{\Ranp}}_! \otimes \Id_{\QCoh(T)})(\cF)  \boxtimes_T ({j_{\Ranp}}_! \otimes \Id_{\QCoh(T)})(\cF)  \rightarrow (\widetilde{\add}^! \otimes \Id_{\QCoh(T)})({j_{\Ranp}}_! \otimes \Id_{\QCoh(T)})(\cF)$$
	restricted to the disjoint locus is an isomorphism. Let $$\pi_X \times \Id_T: [\Ranp \times \Ranp]_{\on{disj}} \times T \rightarrow \Ranp \times \Ranp \times T$$ denote the embedding of the disjoint locus. 
	Then
	$$((\pi_{(X)_{\dR}})^! \boxtimes \Id_{\QCoh(T)})((j_{\Ranp} \times j_{\Ranp})_! \boxtimes\Id_{\QCoh(T)})(\cF \boxtimes_T \cF) .$$
	by base change is equivalent to 
	$$((j \times j)_{\on{disj}, !} \boxtimes \Id_{\QCoh(T)})(\pi_{(U)_{\dR}}^! \boxtimes \Id_{\QCoh(T)})(\cF \boxtimes_T \cF),$$
	where $$(j \times j)_{disj}: [\Ranu \times \Ranu]_{\on{disj}} \rightarrow [\Ranp\times \Ranp].$$
	Since  $\cF \in \CFactAlg^T(\D(\Ranu) \otimes  \QCoh(T))$  this is equivalent to 
	$$((j \times j)_{\on{disj}, !} \boxtimes \Id_{\QCoh(T)})(\add^!\boxtimes \Id_{\QCoh(T)})(\cF).$$
	Since $\add$ is etale (\cite{text}) this is equivalent to 
	$$((\add_{\on{disj}}^!\otimes \Id_{\QCoh(T)})({j_{\Ranp}}_! \otimes \Id_{\QCoh(T)})(\cF).$$
\end{proof}
We now prove that $\add_!$ is well-defined.
\begin{lm}
	Restriction functor $f_I^!: \D(\Ranp) \rightarrow \D(X^I)$ along the tautological morphism admits a left adjoint.
\end{lm}

\begin{proof}
	Consider a functor $$D(-): \Fin \rightarrow \DGCat_{\text{cont}}$$ sending $I \in \Fin$ to $\D(X^I)$ and $p: I \rightarrow J$ to $\Delta_p^!: \D(X^I) \rightarrow \D(X^J)$ and the Grothendieck construction $$\Groth(D(-)) \rightarrow \Fin$$ associated to it. 
	
	Recall (e.g. from \cite[6.2.2]{G2}) that the category  $\D(\Ranp)$ is equivalent to the category of sections of the above functor such that surjections in $\Fin$ are sent to coCartesian arrows. We will show that $f_I^!$ commutes with limits. To do this it suffices to show that a limit of the above sections is computed fiberwise, i.e. we need to show that the naive limit of above sections sends  surjections in $\Fin$ to coCartesian arrows. But this follows from the fact that for $p: I \twoheadrightarrow J$ the functor $\Delta_p^!$ commutes with limits.
\end{proof}

\begin{cor}
	The category $\D(\Ranp)$ is compactly generated by the images of compact objects in $\D(X^I)$ under ${f_I}_!$.
\end{cor}

\begin{proof}
	Follows from the fact that $f_I^!$ are jointly conservative. 
\end{proof}

\begin{cor}\label{left adjoint to addition on unital Ran}
	The restriction functor $$\widetilde{\add}^!:  \D(\Ranp) \rightarrow  \D(\Ranp \times \Ranp)$$
	admits a left adjoint.
\end{cor}

\subsubsection{}
Hence we can reduce the computation of $M$ to the first power of the curve. We get that 
\begin{equation}
M \cong (i^! \otimes \Id_{\QCoh(T)}) ((j^! \otimes \Id_{\QCoh(T)})^{\text{enh}})^L (A),
\end{equation}
where 
\[
\begin{tikzcd}
\CommAlg(\D(X) \otimes \QCoh(T)) \arrow[d, "(j^! \otimes \Id_{\QCoh(T)})^{\text{enh}}"] \ar[r, "\sim", "\fact"']&   \CFactAlg^T(\D(\Ranp_{X}) \otimes \QCoh(T)) \arrow[d, "(j_{\Ranp}^! \otimes \Id_{\QCoh(T)})^{\text{enh}}"]    \\
\CommAlg(\D(U) \otimes \QCoh(T))) \ar[r, "\sim", "\fact"']  &  \CFactAlg^T(\D(\Ranu) \otimes  \QCoh(T)).
\end{tikzcd}
\]

\subsubsection{} Consider the full subcategory $$\PreSt_{\Aff_{/X_{\dR} \times T}}^{\op} \subset\CommAlg(\D(X) \otimes \QCoh(T)),$$ consisting of prestacks affine over $X_{\dR} \times T$ and, analogously, the full subcategory $$\PreSt_{\Aff_{/{U}_{\dR} \times T}}^{\op} \subset\CommAlg(\D(U) \otimes \QCoh(T)) ).$$

Restriction defines a functor $R: \PreSt_{\Aff_{/X_{\dR} \times T}} \rightarrow \PreSt_{\Aff_{/{U}_{\dR} \times T}}$. Tautologically, the diagram 
\[
\begin{tikzcd}
\CommAlg(\D(X) \otimes \QCoh(T)) \arrow[d, "(j^! \otimes \Id_{\QCoh(T)})^{\text{enh}}"] &   \PreSt_{\Aff_{/X_{\dR} \times T}}^{\op}\ar[l, hookrightarrow, ""'] \arrow[d, "R"]    \\
\CommAlg(\D(U) \otimes \QCoh(T)) &  \PreSt_{\Aff_{/{U}_{\dR} \times T}}^{\op}\ar[l, hookrightarrow, ""'] 
\end{tikzcd}
\]
commutes.

\begin{lm}\label{enhanced left adjoint}
	The diagram 
	\[
	\begin{tikzcd}
	\CommAlg(\D(X) \otimes \QCoh(T))  &   \PreSt_{\Aff_{/X_{\dR} \times T}}^{\op}\ar[l, hookrightarrow, ""']    \\
	\CommAlg(\D(U) \otimes \QCoh(T)) \arrow[u,dashrightarrow, "((j^! \otimes \Id_{\QCoh(T)})^{\text{enh}})^L"] &  \PreSt_{\Aff_{/{U}_{\dR} \times T}}^{\op}\arrow[u, dashrightarrow,"R^R"]\ar[l, hookrightarrow, ""'] 
	\end{tikzcd}
	\]
	commutes.
\end{lm}

\begin{proof}
	Denote by $\iota_{X}$ (resp. $\iota_{U}$) the inclusion $$\PreSt_{\Aff_{/X_{\dR} \times T}}^{\op} \hookrightarrow \CommAlg(\D(X) \otimes \QCoh(T))$$ ( resp. $\PreSt_{\Aff_{/{U}_{\dR} \times T}}^{\op} \hookrightarrow \CommAlg(\D(U) \otimes \QCoh(T)) $). These functors admit right adjoints $\iota_{X}^R$ (resp. $\iota_{U}^R$) given by truncation. 
	
	We have that the diagram 
	\[
	\begin{tikzcd}
	\CommAlg(\D(X) \otimes \QCoh(T)) \arrow[d, "(j^! \otimes \Id_{\QCoh(T)})^{\text{enh}}"] \arrow[r, "\iota_{X}^R"]&   \PreSt_{\Aff_{/X_{\dR} \times T}}^{\op} \arrow[d, "R"]    \\
	\CommAlg(\D(U) \otimes \QCoh(T))\arrow[r, "\iota_{U}^R"] &  \PreSt_{\Aff_{/{U}_{\dR} \times T}}^{\op}
	\end{tikzcd}
	\]
	commutes. Therefore, if the functor $R^R$ is defined on $$c \in \PreSt_{\Aff_{/{U}_{\dR} \times T}}^{\op},$$ then $$((j^! \otimes \Id_{\QCoh(T)})^{\text{enh}})^L$$ is defined on $\iota_{U}(c)$ and is given by $\iota_{X}\circ R^R(c)$.
	
	Then it suffices to show that for $\cM \in \CommAlg(\D(U) \otimes \QCoh(T)) $ lying in non-positive degrees $$((j^! \otimes \Id_{\QCoh(T)})^{\text{enh}})^L (\cM)$$ also lies in non-positive degrees. Without loss of generality we may assume that $\cM \cong \Sym(V)$ for $V \in \D(U) \otimes \QCoh(T)$ in non-positive degrees.
	
	Then notice that $$((j^! \otimes \Id_{\QCoh(T)})^{\text{enh}})^L (\Sym(V)) \cong \Sym((j_! \otimes \Id_{\QCoh(S)})V).$$ Thus the result follows since $(j_! \otimes \Id_{\QCoh(S)})$ is right t-exact. 
\end{proof}

\begin{proof}[Proof of Theorem \ref{factmod commutative}]
	By virtue of Lemma \ref{enhanced left adjoint} we see that $M$ is given by multiplication by global functions on the fiber of $R^R (\cY)$ at $x \times T$. Lemma \ref{description of R_B^R} implies that $$R^R (\cY) \cong \hormerxJetsF_{\nabla, T \times U_{\dR}}^{\mer, x}(\cY),$$
	and we get the result. 
\end{proof}

\subsubsection{}
Let now $\cY_U$ be a prestack affine over $U_{\dR} \times T$, i.e. $\cY_U = \Spec_{U_{\dR} \times T}(A_U)$, where $A_U \in \CommAlg(\D(U) \otimes \QCoh(T))$. 
\begin{rem}\label{factmod commutative U}
	By the same argument as in the Theorem \ref{factmod commutative}, we get that for $$A_U \in \D_{\indhol}(U) \otimes \QCoh(T)$$ we have 
		$$\QCoh(\hormerxJetsF_{\nabla, T \times U_{\dR}}^{\mer, x}(\cY)_x)\cong A_U\FactMod^T.$$
\end{rem}

\bibliographystyle{alpha}
\bibliography{FFFv8}

\end{document}